\documentclass{article}
\input{macros-arXiv.sty}

\usepackage{url}            %
\usepackage{booktabs}       %
\usepackage{amsfonts, bm, amssymb}       %
\usepackage{nicefrac}       %
\usepackage{microtype}      %
\usepackage{tikz} 
\usepackage{float}
\usepackage{booktabs}

\usepackage{array}
\usepackage{diagbox}
\usepackage{mathtools}
\usepackage{multirow, makecell}
\usepackage{rotating}
\usepackage{graphicx}
\usepackage{mathdots}

\usepackage{caption}
\usepackage[labelformat=simple]{subcaption}

\usepackage{float}
\usepackage{setspace}
\usepackage{fullpage}

\usepackage[bottom]{footmisc}
\usepackage[colorlinks,linkcolor={blue!75!black},urlcolor=red,citecolor=ForestGreen,backref=page]{hyperref}
 \renewcommand*{\backrefalt}[4]{%
    \ifcase #1%
     \or (Cited on page~#2.)%
     \else (Cited on pages~#2.)%
    \fi%
    }

\usepackage{threeparttable}

\usepackage{arydshln}

\usepackage{enumerate}
\usepackage{enumitem}

\usepackage[sort, round]{natbib}

\usepackage{algorithm, algorithmic, setspace}

\renewcommand{\algorithmiccomment}[1]{\bgroup\small\hfill//~#1\egroup}
\usepackage{float}
\floatstyle{plain}
\newfloat{floatalgo}{ht}{floatalgo}

\usepackage{caption}
\usepackage[labelformat=simple]{subcaption}

\usepackage{mathrsfs}

\newcommand\numberthis{\addtocounter{equation}{1}\tag{\theequation}}

\usepackage{wrapfig}

\usepackage{xiangpackage}
\crefname{assumption}{assumption}{assumptions}

\usepackage{enumitem}

\title{\bf TiAda: A Time-scale Adaptive Algorithm for Nonconvex Minimax Optimization}

\addtocounter{footnote}{1} %
\author{
Xiang Li\thanks{Department of Computer Science, ETH Zurich, Switzerland.
\texttt{xiang.li@inf.ethz.ch},
\texttt{junchi.yang@inf.ethz.ch},
\texttt{niao.he@inf.ethz.ch}.}
\and
Junchi Yang\footnotemark[2]
\and
Niao He\footnotemark[2]
}

\date{\vspace{1ex}}

\begin{document}
\maketitle

\begin{abstract}
   Adaptive gradient methods have shown their ability to adjust the
   stepsizes on the fly in a
   parameter-agnostic manner, and empirically achieve faster convergence for solving minimization problems. When it comes to nonconvex minimax optimization, however, current convergence
   analyses of gradient descent ascent~(GDA) combined with adaptive
   stepsizes require careful tuning of hyper-parameters and the knowledge of
   problem-dependent parameters.
   Such a discrepancy arises from the primal-dual nature of minimax problems
   and 
   the necessity of delicate \emph{time-scale separation} between the primal and dual updates in attaining convergence.
   In this work, we propose a \emph{single-loop} adaptive GDA algorithm called TiAda
   for nonconvex minimax optimization that automatically adapts to the
   time-scale separation.
   Our algorithm is \emph{fully parameter-agnostic} and can
   achieve \emph{near-optimal complexities} simultaneously in deterministic
   and stochastic settings of nonconvex-strongly-concave minimax problems. The
   effectiveness of the proposed method is further justified numerically for a number of machine learning applications.

\end{abstract}

\vspace{0.25cm}
\section{Introduction}
\label{sec:intro}

Adaptive gradient methods, such as AdaGrad~\citep{duchi2011adaptive},
Adam~\citep{kingma2015adam} and AMSGrad~\citep{reddi2018convergence},
have become the default choice of optimization
algorithms in many machine learning applications owing to their robustness to 
hyper-parameter selection and fast empirical convergence. These advantages
are especially prominent in nonconvex regime with success in
training deep neural networks~(DNN). 
Classic analyses of gradient descent for smooth functions require the 
stepsize to be less than $2/l$, where $l$ is the smoothness parameter and often
unknown for complicated models like DNN.
Many adaptive schemes, usually with diminishing stepsizes
based on cumulative gradient information, can adapt
to such parameters and thus reducing the burden of hyper-parameter tuning ~\citep{ward2020adagrad,xie2020linear}. Such tuning-free
algorithms are called \emph{parameter-agnostic}, 
as they do not require any prior knowledge of  problem-specific parameters, e.g., the smoothness or strong-convexity parameter. 

In this work, we aim to bring the benefits of adaptive stepsizes to solving the following problem: 
\begin{align*}
  \min_{x \in \bR^{d_1}} \max_{y \in \cY} f(x, y) = \mathbb{E}_{\xi \in P}\left[ F(x, y;\xi)\right], \numberthis \label{eq:minimax} 
\end{align*}
where $P$ is an unknown distribution from which we can drawn i.i.d. samples, $\cY \subset \bR^{d_2}$ is closed and convex, and $f : \bR^{d_1} \times \bR^{d_2}
\rightarrow \bR$ is nonconvex in $x$. We call $x$ the primal variable and 
$y$ the dual variable. This  minimax formulation
has found vast applications in modern machine learning, notably generative adversarial 
networks~\citep{goodfellow2014generative,arjovsky2017wasserstein}, adversarial
learning~\citep{goodfellow2015explaining,miller2020adversarial}, reinforcement 
learning~\citep{dai2017learning,modi2021model}, sharpness-aware 
minimization~\citep{foret2021sharpnessaware}, domain-adversarial
training~\citep{ganin2016domain}, etc.
Albeit theoretically underexplored, 
adaptive methods are widely deployed in these applications in combination with
popular minimax optimization algorithms such as (stochastic) gradient descent ascent~(GDA),
extragradient~(EG)~\citep{korpelevich1976extragradient}, and optimistic GDA~\citep{popov1980modification,rakhlin2013online}; see, e.g., ~\citep{gulrajani2017improved,daskalakis2018training,mishchenko2020revisiting, reisizadeh2020robust}, just to list a few. 

\begin{figure}[t]
    \centering
    \begin{subfigure}[b]{0.31\textwidth}
      \centering
      \includegraphics[width=\textwidth]{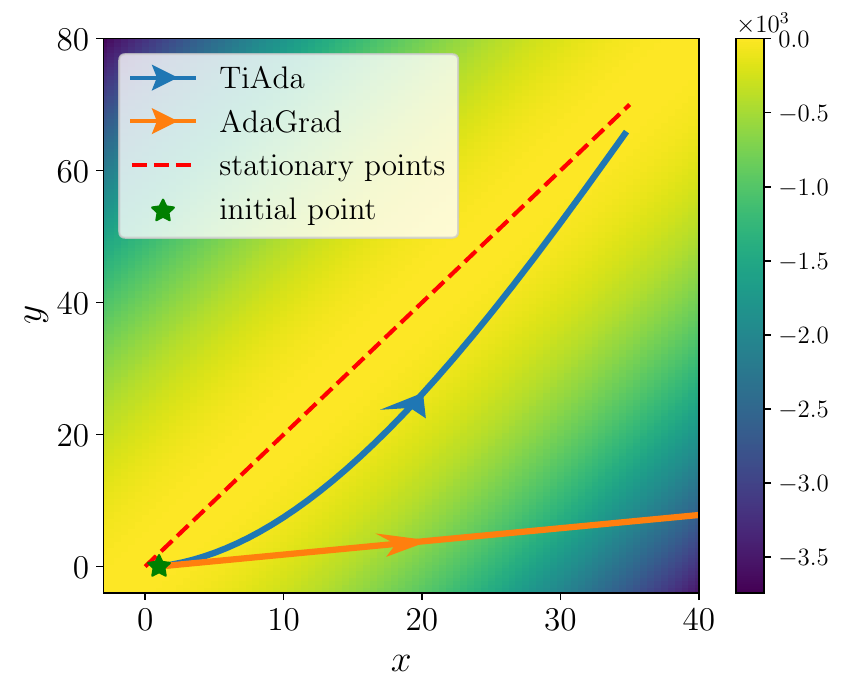}
      \caption{trajectory}
      \label{subfig:quad_trajectory}
    \end{subfigure}
    \hfill
    \begin{subfigure}[b]{0.332\textwidth}
      \centering
      \includegraphics[width=\textwidth]{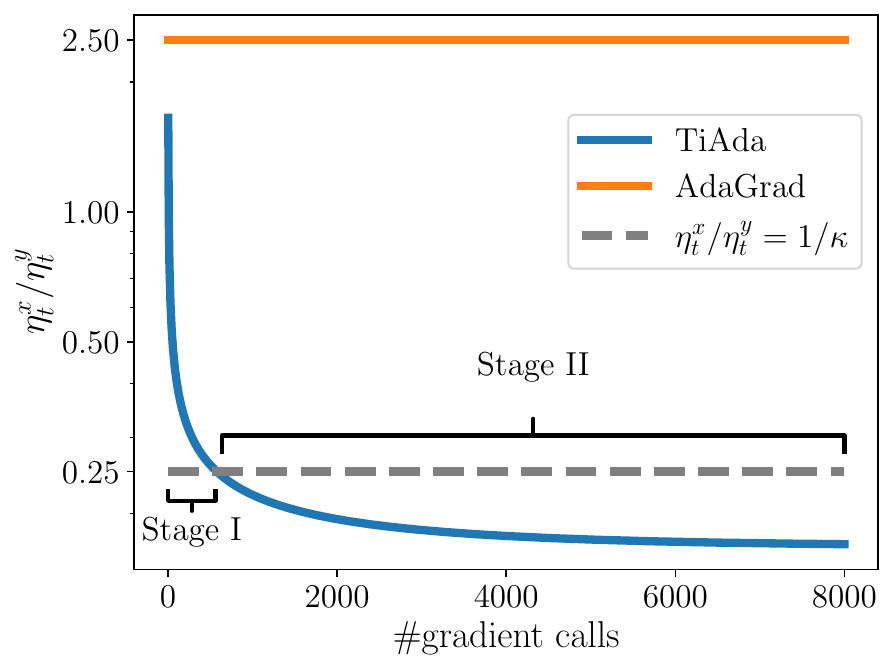}
      \caption{effective stepsize ratio}
      \label{subfig:quad_ratio}
    \end{subfigure}
    \hfill
    \begin{subfigure}[b]{0.324\textwidth}
      \centering
      \includegraphics[width=\textwidth]{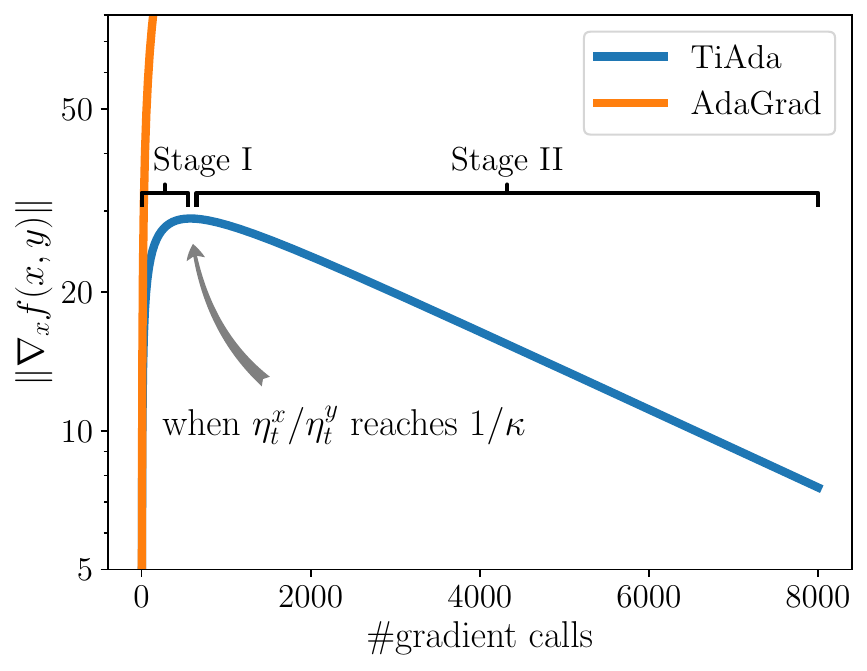}
      \caption{convergence}
      \label{subfig:quad_converge}
    \end{subfigure}
    \caption{Comparison between TiAda and vanilla GDA with AdaGrad  
    stepsizes (labeled as AdaGrad)  on the quadratic function~\eqref{eq:quad} with $L=2$ under a poor initial stepsize ratio, i.e., $\eta^x / \eta^y = 5$.
      Here, $\eta^x_t$ and $\eta^y_t$ are the effective stepsizes respectively for $x$ and $y$,
    and $\rk$ is the condition number\protect\footnotemark.
    (a) shows the trajectory of the two algorithms and the background color demonstrates the function value $f(x, y)$. In (b),  while the effective stepsize ratio stays unchanged for AdaGrad, TiAda adapts to the desired \emph{time-scale separation} $1/\rk$, which divides
  the training process into two stages. In (c), after entering Stage II, TiAda converges fast, whereas AdaGrad diverges.}
    \label{fig:compare_ratio}
\end{figure}

\footnotetext{Please refer to \Cref{sec:method} for formal definitions of initial stepsize and effective stepsize. Note that the initial stepsize ratio, $\eta^x / \eta^y$,  does not necessarily equal to the first effective stepsize ratio, $\eta_0^x / \eta_0^y$. }

While it seems natural to directly extend adaptive
stepsizes to  minimax optimization algorithms, a recent work
by~\citet{yang2022nest}
pointed out that such schemes
may not always converge without knowing problem-dependent
parameters. 
Unlike the case of
minimization, convergent analyses of GDA and EG for nonconvex minimax
optimization are subject to \emph{time-scale separation} ~\citep{boct2020alternating, lin2020gradient, sebbouh2022randomized, yang2021faster} --- the stepsize ratio of primal and dual variables needs to be
smaller than a problem-dependent threshold --- which is recently shown to be necessary even when the objective is strongly concave in $y$ with true gradients~\citep{li2022convergence}.
Moreover, \citet{yang2022nest} showed that GDA with standard adaptive stepsizes, that chooses the stepsize of each variable based only on the (moving) average of its own past gradients, fails
to adapt to the time-scale separation requirement.
Take the following nonconvex-strongly-concave function
as a concrete example:
\[
  f(x, y) = -\frac{1}{2}y^2 + Lxy - \frac{L^2}{2}x^2, \numberthis \label{eq:quad}
\]
where $L > 0$ is a constant. \citet{yang2022nest} proved that directly using adaptive stepsizes
like AdaGrad, Adam and AMSGrad will fail to converge if the ratio of initial stepsizes
of $x$ and $y$ (denoted as $\eta^x$ and $\eta^y$) is large.
We illustrate this phenomenon in \Cref{subfig:quad_trajectory,subfig:quad_converge},
where AdaGrad diverges. 
To sum up, adaptive stepsizes designed for minimization, are not
\emph{time-scale adaptive} for minimax optimization and thus not \emph{parameter-agnostic}. 

To circumvent this time-scale separation bottleneck, \citet{yang2022nest} introduced an adaptive algorithm, NeAda, for problem~\eqref{eq:minimax} with nonconvex-strongly-concave objectives. NeAda is a two-loop algorithm built upon GDmax \citep{lin2020gradient} that after one primal variable update, updates the dual variable for multiple steps until a stopping criterion is satisfied in the inner loop. Although the algorithm is agnostic to the smoothness and strong-concavity parameters, there are several limitations that may undermine its performance in large-scale training: (a) In the stochastic setting, it gradually increases the number of inner loop
steps ($k$ steps for the $k$-th outer loop) to
improve the inner maximization problem accuracy, resulting in a possible waste of inner loop
updates if the maximization problem is already well solved;  (b) NeAda needs a large batchsize of order $\Omega\left(\epsilon^{-2}\right)$ to achieve the near-optimal convergence rate in theory; (c) It is not fully adaptive to the gradient noise, since it deploys different
strategies for deterministic and stochastic settings.

In this work, we address all of the issues above by proposing TiAda~(\textbf{Ti}me-scale \textbf{Ada}ptive Algorithm),
a single-loop algorithm with time-scale
adaptivity for minimax optimization.
Specifically, one of our major modifications is setting 
the effective stepsize, i.e., the scale of (stochastic) gradient used in the updates, of the primal variable
to the reciprocal of the \emph{maximum} between
the primal and dual variables' second moments, i.e.,
the sums of their past gradient norms.
This ensures the effective stepsize ratio of $x$ and $y$
being upper bounded by a decreasing sequence, which eventually reaches the desired
time-scale separation.
Taking the test function~\eqref{eq:quad} as an example,
\Cref{fig:compare_ratio} illustrates the time-scale adaptivity of TiAda: In
Stage I, the stepsize ratio quickly decreases below the threshold; in Stage II, the ratio is stabilized and
the gradient norm starts to  converge fast.

We focus on the minimax optimization (\ref{eq:minimax}) that is strongly-concave in $y$, since other nonconvex regimes are far less understood even without adaptive stepsizes. Moreover, near stationary point may not  exist in nonconvex-nonconcave~(NC-NC) problems and finding first-order local minimax point is already PPAD-complete \citep{daskalakis2021complexity}. We 
consider a constraint for the dual variable,
which is common in convex optimization with adaptive stepsizes \citep{levy2017online, levy2018online} and in the minimax optimization with non-adaptive stepsizes \citep{lin2020gradient}. 
In summary, our contributions are as follows:
\begin{itemize}
  \item We introduce the first \emph{single-loop} and \emph{fully parameter-agnostic} adaptive algorithm, TiAda,
    for nonconvex-strongly-concave (NC-SC) minimax optimization. It adapts to the necessary time-scale separation without large batchsize or any knowledge of problem-dependant parameters or target accuracy.
    TiAda finds an $\epsilon$-stationary point with an optimal complexity 
    of $\cO\left(\epsilon^{-2}\right)$ in the deterministic case, and a near-optimal sample complexity of
    $\cO\left(\epsilon^{-(4+\delta)}\right)$ for any small $\delta > 0$
    in the stochastic case. It shaves off the extra logarithmic terms in the complexity of NeAda with AdaGrad stepsize for both primal and dual variables
    \citep{yang2022nest}.
    TiAda is proven to be noise-adaptive, which is the first of its kind among nonconvex minimax optimization algorithms.

  \item While TiAda is based on AdaGrad stepsize, we generalize TiAda with other existing adaptive schemes,
    and conduct experiments on several tasks. The tasks include 1) test functions
     by \citet{yang2022nest} for showing the nonconvergence of
    GDA with adaptive schemes under poor initial stepsize
    ratios, 2) distributional robustness optimization~\citep{sinha2018certifiable}
    on MNIST dataset with a NC-SC objective,
    and 3) training the NC-NC generative adversarial networks
    on CIFAR-10 dataset. In all tasks,
    we show that TiAda converges faster and 
    is more robust
    compared with NeAda or
    GDA with other existing adaptive stepsizes.
\end{itemize}

\subsection{Related Work}

\paragraph{Adaptive gradient methods.} AdaGrad brings about an adaptive mechanism for gradient-based optimization algorithm that adjusts its stepsize by keeping the averaged past gradients. The original AdaGrad was introduced for online convex optimization and maintains coordinate-wise stepsizes. In nonconvex stochastic optimization,  AdaGrad-Norm with one learning rate for all directions is shown to achieve the same complexity as SGD \citep{ward2020adagrad, li2019convergence}, even with the high probability bound \citep{kavis2022high, li2020high}. In comparison, RMSProp \citep{hinton2012neural} and Adam \citep{kingma2015adam} use the decaying moving average of past gradients, but may suffer from divergence \citep{reddi2018convergence}. Many variants of Adam are proposed, and a wide family of them, including AMSGrad, are provided with convergence guarantees \citep{zhou2018convergence, chen2018convergence, defossez2020simple, zhang2022adam}. 
One of the distinguishing traits of adaptive algorithms is that they can achieve order-optimal rates without knowledge about the problem parameters, such as smoothness and variance of the noise, even in nonconvex optimization \citep{ward2020adagrad, levy2021storm+, kavis2019unixgrad}.

\paragraph{Adaptive minimax optimization algorithms.} The adaptive stepsize schemes are naturally extended to minimax optimization, both in theory and practice, notably in the training of GANs \citep{goodfellow2016nips, gidel2018variational}. In the convex-concave regime, several adaptive algorithms are designed based on EG and AdaGrad stepsize, and they inherit the parameter-agnostic characteristic \citep{bach2019universal, antonakopoulos2019adaptive}. In sharp contrast, when the objective function is nonconvex about one variable, most existing adaptive algorithms require knowledge of the problem parameters \citep{huang2021adagda, huang2021efficient, guo2021novel}. Very recently, it was proved that a parameter-dependent ratio between two stepsizes is necessary for GDA  in NC-SC minimax problems with non-adaptive stepsize \citep{li2022convergence} and most existing adaptive stepsizes \citep{yang2022nest}. %
\citet{heusel2017gans} shows the two-time-scaled GDA with non-adaptive stepsize or Adam will converge, 
but assuming the existence of an asymptotically stable attractor.

\paragraph{Other NC-SC minimax optimization algorithms.}
In the NC-SC setting, the most popular algorithms are GDA and GDmax, in which one primal variable update is followed by one or multiple steps of dual variable updates
Both of them can achieve $\mathcal{O}(\epsilon^{-2})$ complexity in the deterministic setting and $\mathcal{O}(\epsilon^{-4})$ sample complexity in the stochastic setting \citep{lin2020gradient, chen2021closing, nouiehed2019solving, yang2020global}, which are not improvable in the dependency on $\epsilon$ given the existing lower complexity bounds \citep{zhang2021complexity, li2021complexity}. Later, several works further improved the dependency on the condition number with more complicated algorithms in deterministic \citep{yang2021faster, lin2020near} and stochastic settings \citep{zhang2022sapd+}. All of the algorithms above do not use adaptive stepsizes and rely on prior knowledge of the problem parameters.

\subsection{Notations}

We  denote $l$ as the smoothness parameter, $\mu$ as the strong-concavity
parameter, whose formal definitions will be introduced in \Cref{assume:smoothness,assume:strong-convex}, and $\rk := l / \mu$ as the condition number.  We assume access to stochastic gradient oracle returning
$[\nabla_x F(x, y; \xi), \nabla_y F(x, y; \xi)]$.
For the minimax problem~\eqref{eq:minimax}, we denote $y^*(x) := \argmax_{y\in \mathcal{Y}} f(x, y)$ as the
solution of the inner maximization problem, $\Phi(x) := f(x, y^*(x))$
as the primal function, and $\cP_{\cY}(\cdot)$ as projection operator onto set $\cY$. 
For notational simplicity, we will use the name of an existing adaptive algorithm to
refer to the simple combination of GDA and it, i.e., setting the stepsize
of GDA to that adaptive scheme separately for both $x$ and $y$. %
For instance ``AdaGrad'' for minimax problems stands for the algorithm that uses AdaGrad stepsizes separately for $x$ and $y$ in GDA.

\section{Method}
\label{sec:method}

We formally introduce the TiAda method in \Cref{alg:tiada}, and
the major difference with AdaGrad lies in line~\ref{tiada_update}.
Like AdaGrad, TiAda stores the accumulated squared (stochastic) gradient norm of the primal and dual variables in $v_t^x$ and $v_t^y$, respectively. 
We refer to hyper-parameters $\eta^x$ and $\eta^y$ as the \emph{initial
stepsizes}, and the actual stepsizes for updating in line~\ref{tiada_update} as \emph{effective
stepsizes} which are denoted by $\eta^x_t$ and $\eta^y_t$. 
TiAda adopts effective stepsizes $\eta_t^x = \eta^x / \max\left\{v^x_{t+1}, v^y_{t+1}\right\}^{\alpha}$
and $\eta_t^y  = \eta^y / \left(v^y_{t+1}\right)^{\beta}$, while AdaGrad uses
$\eta^x / \left(v_{t+1}^x\right)^{1/2}$ and $\eta^y / \left(v_{t+1}^y\right)^{1/2}$. In \Cref{sec:theory}, our theoretical analysis suggests to choose $\alpha > 1/2 > \beta$. We will also illustrate in the next subsection that the  $\max$ structure and different $\alpha, \beta$ make our algorithm adapt to the desired time-scale separation.

For simplicity of analysis, similar to AdaGrad-Norm~\citep{ward2020adagrad},
we use the norms of gradients for updating the effective stepsizes. 
A  more practical coordinate-wise variant
that can be used for high-dimensional models is presented in
\Cref{sec:extension}.

\begin{algorithm}[ht] 
    \caption{TiAda~(Time-scale Adaptive Algorithm)}
    \setstretch{1.23}
    \begin{algorithmic}[1]
      \STATE \textbf{Input:} $(x_0, y_0)$, $v^x_0 > 0$, $v^y_0 > 0$, $\eta^x > 0$,
          $\eta^y > 0$, $\alpha > 0$, $\beta > 0$ and $\alpha > \beta$.
        \FOR{$t = 0,1,2,...$}
            \STATE sample i.i.d. $\xi^x_t$ and $\xi^y_t$, and let $g_t^x = \nabla_x F(x_t, y_t; \xi^x_t)$ and $g_t^y = \nabla_y F(x_t, y_t; \xi^y_t)$ \label{firstm_update}
            \STATE $v_{t+1}^x = v_t^x + \norm*{g_t^x}^2$ and $v_{t+1}^y = v_t^y + \norm*{g_t^y}^2$ \label{secondm_update}
            \STATE $x_{t+1} = x_t - \frac{\eta^x}{\max\left\{v^x_{t+1}, v^y_{t+1}\right\}^{\alpha}} g^x_{t}$ and
            $ y_{t+1} = \cP_{\cY}\left( y_t + \frac{\eta^y}{\left(v^y_{t+1}\right)^{\beta}} g^y_{t} \right)$
            \label{tiada_update}
        \ENDFOR
    \end{algorithmic} \label{alg:tiada}
\end{algorithm}

\subsection{The Time-Scale Adaptivity of TiAda}
\label{sec:time_adaptivity}

Current analyses of GDA with non-adaptive stepsizes require the time-scale, $\eta^x_t / \eta^y_t$, to be
smaller than a threshold depending on problem constants such as the smoothness and the strong-concavity parameter ~\citep{lin2020gradient,yang2021faster}.
The intuition is that we should not aggressively update $x$ if the
inner maximization problem has not yet been solved accurately, i.e.,
we have not found a good approximation of $y^*(x)$.
Therefore, the effective stepsize of $x$ should be small compared with that of $y$.
It is tempting to expect
adaptive stepsizes to automatically find a suitable time-scale separation.
However, the quadratic example~\eqref{eq:quad} given by \citet{yang2022nest} shattered the illusion.
In this example, the effective stepsize ratio stays the same along the run of existing adaptive algorithms, including AdaGrad (see \Cref{subfig:quad_ratio}), Adam and AMSGrad, and they 
fail to converge if the initial stepsizes are not carefully chosen (see \citet{yang2022nest} for details).
As $v_{t}^x$ and $v_{t}^y$ only separately contain the gradients of $x$ and $y$, the effective stepsizes of two variables in these adaptive methods depend on their own history, which prevents them from cooperating to adjust the ratio.

Now we explain how TiAda adapts to both the required time-scale separation and small enough stepsizes. First,
the ratio of our modified effective stepsizes is
upper bounded by a decreasing sequence when $\alpha > \beta$:
\begin{align*}
  \frac{\eta^x_t}{\eta^y_t} =
  \frac{\eta^x / \max\left\{v^x_{t+1}, v^y_{t+1}\right\}^{\alpha}}{\eta^y / \left(v^y_{t+1}\right)^{\beta}}
  \leq 
  \frac{\eta^x / \left(v^y_{t+1}\right)^{\alpha}}{\eta^y / \left(v^y_{t+1}\right)^{\beta}}
  = \frac{\eta^x}{\eta^y \left(v^y_{t+1}\right)^{\alpha - \beta}},
  \numberthis \label{eq:ratio_bound}
\end{align*}
as $v^y_{t}$ is the sum of previous 
gradient norms and is increasing. Regardless of the initial stepsize  ratio $\eta^x/\eta^y$, we expect the effective stepsize ratio to eventually drop below the desirable threshold for convergence. On the other hand, the effective stepsizes for the primal and dual variables are also upper bounded by decreasing sequences, $\eta^x / \left(v^x_{t+1}\right)^{\alpha}$ and $\eta^y / \left(v^y_{t+1}\right)^{\beta}$, respectively. Similar to AdaGrad, such adaptive stepsizes will reduce to small enough, e.g., $\cO(1/l)$, to ensure convergence.

Another way to look at the effective stepsize of $x$ is
\[
    \eta^x_t = 
  \frac{\eta^x}{\max\left\{v^x_{t+1}, v^y_{t+1}\right\}^{\alpha}}
  = 
  \frac{\left(v^x_{t+1}\right)^{\alpha}}{\max\left\{v^x_{t+1}, v^y_{t+1}\right\}^{\alpha}} 
  \cdot \frac{\eta^x}{\left(v^x_{t+1}\right)^{\alpha}}.
  \numberthis \label{eq:factor_ex}
\]
If the
gradients of $y$ are small
(i.e., $v^y_{t+1} < v^x_{t+1}$), meaning the inner maximization problem is
well solved, then the first factor becomes $1$ and the effective stepsize of $x$ is just
the second factor, similar to the AdaGrad updates. If the term $v_{t+1}^y$ dominates over $v_{t+1}^x$, the first
factor would be smaller than $1$, allowing to slow down the update of $x$ and waiting
for a better approximation of $y^*(x)$.

To demonstrate the time-scale adaptivity of TiAda, we conducted
experiments on the quadratic minimax example~\eqref{eq:quad} with $L=2$. As shown in
\Cref{subfig:quad_ratio}, while the effective stepsize ratio of AdaGrad stays
unchanged for this particular function, TiAda progressively decreases the
ratio. According to Lemma~2.1 of \citet{yang2022nest}, $1/\rk$
is the threshold where GDA starts to converge. We label the time period
before reaching this threshold as Stage I, during which as shown in
\Cref{subfig:quad_converge}, the gradient norm for TiAda increases. However, as soon as it
enters Stage II, i.e., when the ratio drops below $1/\rk$, TiAda converges fast to the stationary point.
In contrast, since the stepsize ratio of AdaGrad never reaches this threshold, the gradient norm keeps growing. 

\section{Theoretical Analysis of TiAda}
\label{sec:theory}

In this section, we study the convergence of TiAda under NC-SC setting with both deterministic and stochastic gradient oracles. 
We make the following assumptions to develop our convergence results. 

\begin{assumption}[smoothness] \label{assume:smoothness}
  Function $f(x, y)$ is $l$-smooth ($l > 0$) in both $x$ and $y$, that is, for any
  $x_1, x_2 \in \bR^{d_1}$ and $y_1, y_2 \in \cY$, we have
  \[
    \max\{\norm*{\nabla_x f(x_1, y_1) - \nabla_x f(x_2, y_2)}, \norm*{\nabla_y f(x_1, y_1) - \nabla_y f(x_2, y_2)}\}
    \leq l \left(\norm*{x_1 - x_2} + \norm*{y_1 - y_2}\right).
  \]
\end{assumption}
\begin{assumption}[strong-concavity in $y$] \label{assume:strong-convex}
  Function $f(x, y)$ is $\mu$-strongly-concave ($\mu > 0 $) in $y$, that is, for any $x \in \bR^{d_1}$ and $y_1, y_2 \in \cY$, we have
  \[
    f(x, y_1) \geq f(x, y_2) + \inp*{\nabla_y f(x, y_1)}{y_1 - y_2}
    + \frac{\mu}{2} \norm*{y_1 - y_2}^2.
  \]
\end{assumption}
\begin{assumption}[interior optimal point] \label{assume:interior_optimal}
  For any $x \in \bR^{d_1}$, $y^*(x)$ is in the interior of $\cY$.
\end{assumption}
\begin{remark}
  The last assumption ensures $\nabla_y f(x, y^*(x)) = 0$, which is important for
  AdaGrad-like stepsizes that use the sum of squared norms of past gradients
  in the denominator. 
  If the gradient about $y$ is not 0 at $y^*(x)$, the stepsize will keep decreasing even near the optimal point, leading to slow convergence.
  This assumption could be potentially alleviated by using generalized AdaGrad stepsizes \citep{bach2019universal}. 
 \end{remark}

We aim to find a near stationary point for the minimax problem~\eqref{eq:minimax}. Here, $(x, y)$ is defined to be an $\epsilon$ stationary point if $\|\nabla_x f(x, y)\|\leq \epsilon$ and  $\|\nabla_y f(x, y)\|\leq \epsilon$ in the deterministic setting, or $\mathbb{E}\|\nabla_x f(x, y)\|^2\leq \epsilon^2$ and  $\mathbb{E}\|\nabla_y f(x, y)\|^2\leq \epsilon^2$ in the stochastic setting, where the expectation is taken over all the randomness in the algorithm. This stationarity notion can be easily translated to the near-stationarity of the primal function $\Phi(x) = \max_{y \in \mathcal{Y}} (x, y)$ \citep{yang2021faster}.
Under our analyses,
TiAda is able to achieve the optimal $\cO\left(\epsilon^{-2}\right)$
complexity in the deterministic setting and a near-optimal
$\cO\left(\epsilon^{-(4+\delta)}\right)$ sample complexity for any small $\delta > 0$
in the stochastic setting.

\subsection{Deterministic Setting}

In this subsection, we assume to have access to the exact gradients of 
$f(\cdot, \cdot)$, and therefore we can replace $\nabla_x F(x_t, y_t; \xi^x_t)$ and $ \nabla_y F(x_t, y_t; \xi^y_t)$ by $\nabla_x f(x_t, y_t)$ and $\nabla_y f(x_t, y_t)$ in Algorithm \ref{alg:tiada}. 

\begin{theorem}[deterministic setting] \label{theorem:tiada_determ}

  Under \Cref{assume:strong-convex,assume:smoothness,assume:interior_optimal},
  \Cref{alg:tiada} with deterministic gradient
  oracles satisfies that for any $0 < \beta < \alpha < 1$, after $T$ iterations,
  \begin{align*}
    \frac{1}{T}\sum_{t=0}^{T-1}\norm*{\nabla_x f(x_t, y_t)}^2
    + \frac{1}{T}\sum_{t=0}^{T-1}\norm*{\nabla_y f(x_t, y_t)}^2
  \leq \cO\left(\frac{1}{T}\right).
  \end{align*}
\end{theorem}

This theorem implies that for any initial stepsizes, TiAda finds an $\epsilon$-stationary point within $\cO(\epsilon^{-2})$ iterations. Such complexity is comparable to that of nonadaptive methods, such as vanilla GDA \citep{lin2020gradient}, and is optimal in the dependency of $\epsilon$ \citep{zhang2021complexity}. Like NeAda \citep{yang2022nest}, TiAda does not need any prior knowledge about $\mu$ and $l$, but it improves over NeAda by removing the  logarithmic term in the complexity. 
Notably,  we provide a unified analysis for a wide range of $\alpha$ and $\beta$, while most existing literature on 
AdaGrad-like stepsizes only validates a specific hyper-parameter, e.g., $\alpha = 1/2$ %
in minimization problems \citep{ward2020adagrad, kavis2019unixgrad}.

\subsection{Stochastic Setting}
\label{subsec:stochastic}

In this subsection, we assume the access to a stochastic gradient oracle, that returns 
unbiased noisy gradients, $\nabla_x F(x, y; \xi)$ and $\nabla_y F(x, y; \xi)$. 
Also, we make the following additional assumptions.

\begin{assumption}[stochastic gradients]
\label{assume:stochastic_grad} 
For $z \in \{x, y\}$,
we have $\Ep[\xi]{\nabla_z F(x, y, \xi)} = \nabla_z f(x, y)$.
In addition, there exists a constant $G$ such that
$\norm*{\nabla_z F(x, y, \xi)} \leq G$
for any $x \in \bR^{d_1}$ and $y \in \cY$.
\end{assumption}

\begin{assumption}[bounded primal function value] \label{assume:bound_func_val}
  There exists a constant $\Phi_{\max} \in \bR$ such that for any $x \in
  \bR^{d_1}$, $\Phi(x)$ is upper bounded by $\Phi_{\max}$.
\end{assumption}
\begin{remark}
  The bounded gradients and function value are assumed in many works on adaptive algorithms \citep{kavis2022high, levy2021storm+}. 
  This implies the domain of $y$ %
  is bounded,
  which is also assumed in the analyses of 
  AdaGrad   \citep{levy2017online, levy2018online}.
  In neural networks with rectified activations,
  because of its scale-invariance property~\citep{dinh2017sharp},
  imposing boundedness of $y$ does not affect the expressiveness.
  Wasserstein GANs~\citep{arjovsky2017wasserstein} also use
  projections on the critic to restrain the weights on a small cube
  around the origin.
\end{remark}

\begin{assumption}[second order Lipschitz continuity for $y$] \label{assume:lipschitz}
  For any $x_1, x_2 \in \bR^{d_1}$ and $y_1, y_2 \in \cY$,
  there exists constant $L$ such that 
  \[
    \norm*{\nabla_{xy}^2 f(x_1, y_1) - \nabla_{xy}^2 f(x_2, y_2)} 
    \leq L \left(\norm*{x_1 - x_2} + \norm*{y_1 - y_2}\right)
  \]
  and 
  \[
    \norm*{\nabla_{yy}^2 f(x_1, y_1) - \nabla_{yy}^2 f(x_2, y_2)} 
    \leq L \left(\norm*{x_1 - x_2} + \norm*{y_1 - y_2}\right).
  \]
\end{assumption}
\begin{remark}
  
  \citet{chen2021closing}
  also impose this assumption to achieve the optimal $\cO\left(\epsilon^{-4}\right)$ complexity for GDA with non-adaptive stepsizes for solving NC-SC minimax problems. 
  Together with \Cref{assume:interior_optimal}, we can show that
  $y^*(\cdot)$ is smooth. Nevertheless, 
  without this assumption, \citet{lin2020gradient} only show a worse complexity of $\cO\left(\epsilon^{-5}\right)$ for GDA without large batchsize.
\end{remark}

\begin{theorem}[stochastic setting] \label{theorem:tiada_stoc}
  Under \Cref{assume:strong-convex,assume:smoothness,assume:stochastic_grad,assume:lipschitz,assume:interior_optimal,assume:bound_func_val},
  \Cref{alg:tiada} with stochastic gradient
  oracles satisfies that for any $0 < \beta < \alpha < 1$, after $T$ iterations,
  \begin{align*}
  \frac{1}{T} \Ep{\sum_{t=0}^{T-1}\norm*{\nabla_x f(x_t, y_t)}^2
  + \sum_{t=0}^{T-1}\norm*{\nabla_y f(x_t, y_t)}^2}
  \leq \cO\left(T^{\ra - 1} + T^{-\ra}
  + T^{\beta - 1} + T^{-\beta}\right).
  \end{align*}
\end{theorem}

  TiAda can achieve the complexity arbitrarily close to the optimal sample complexity, $\cO\left(\epsilon^{-4}\right)$ \citep{li2021complexity},
  by choosing 
  $\alpha$ and $\beta$ arbitrarily close to $0.5$.
  Specifically, TiAda achieves a complexity of 
  $\cO\left(\epsilon^{-(4+\delta)}\right)$ for any small $\delta > 0$ if we set
  $\alpha = 0.5 + \delta / (8 + 2 \delta)$ and $\beta = 0.5 - \delta / (8 + 2 \delta)$. Notably, this matches the complexity of 
  NeAda with AdaGrad stepsizes for both variables \citep{yang2022nest}. %
  NeAda may attain $\widetilde{\cO}(\epsilon^{-4})$ complexity with more complicated subroutines for $y$.

\Cref{theorem:tiada_stoc} implies that TiAda is fully agnostic to problem parameters, e.g., $\mu$, $l$ and $\sigma$. GDA with non-adaptive stepsizes \citep{lin2020gradient} and vanilla single-loop adaptive methods \citep{huang2021adagda}, such as AdaGrad and AMSGrad,  all require knowledge of these parameters. Compared with the only parameter-agnostic algorithm, NeAda, our algorithm has several advantages. First, TiAda is a single-loop algorithm, while NeAda \citep{yang2022nest}
needs increasing inner-loop steps and a huge batchsize of order $\Omega\left(\epsilon^{-2}\right)$ to achieve its best complexity. Second, our stationary guarantee is for $\mathbb{E} \|\nabla_x f(x, y)\|^2 \leq \epsilon^2$, which is stronger than $\mathbb{E} \|\nabla_x f(x, y)\| \leq \epsilon$ guarantee in NeAda. Last but not least, although NeAda does not need to know the exact value of variance $\sigma$ in the stochastic setting when $\sigma > 0$, NeAda uses a different stopping criterion for the inner loop 
in the deterministic setting when $\sigma = 0$, so it still needs partial information about $\sigma$. In comparison, TiAda achieves the (near) optimal complexity in both settings with the same strategy.

  Consistent with the intuition of time-scale adaptivity in \Cref{sec:time_adaptivity},
  the convergence result can be derived in two stages. In Stage I,
  according to the upper bound of the ratio in \Cref{eq:ratio_bound},
  we expect the term $1/\left(v^y_{t+1}\right)^{\alpha-\beta}$
   reduces to a constant $c$, a desirable time-scale separation. 
  This means that $v^y_{t+1}$ has to grow to
  nearly
  $(1/c)^{1/(\alpha - \beta)}$.
In Stage II,
  when the time-scale separation is satisfied, TiAda converges at a speed
  specified in \Cref{theorem:tiada_stoc}. This indicates
  that the proximity between $\alpha$ and $\beta$ affects
  the speed trade-off between Stage I and II. When $\alpha$ and $\beta$ are close,
  we have a faster overall convergence rate close to the optimality, but
  suffer from a longer transition phase in Stage I, albeit by only a constant term. We also present an empirical ablation study on the convergence behavior with different choices of $\alpha$ and $\beta$ in \Cref{subsec:ablation}.

\begin{remark}
  In TiAda, the update of $x$ requires to know the gradients of $y$ (or
  $v^y_{t+1}$). However, in some applications that
  concern about privacy, one player might not access the information 
  about the other player~\citep{koller1995generating,foster2006regret,he2016opponent}.
  Therefore, we also consider a variant of TiAda without taking the maximum of gradient norms, i.e.,
  setting the effective stepsize of $x$ in \Cref{alg:tiada} to
  $\eta^x / \left(v^x_{t+1}\right)^{\alpha}$. This variant  achieves
  a sub-optimal complexity of $\cOt\left(\epsilon^{-6}\right)$. This result further justifies the importance of coordination between adaptive stepsizes of two players for achieving faster convergence in minimax optimization. The algorithm and convergence results are presented in \Cref{sec:no_max}.
 \end{remark}

\section{Experiments}
\label{sec:exp}

In this section, we first present extensions of TiAda
that accommodate other adaptive schemes besides AdaGrad and 
are more practical in deep models.
Then we present empirical results of TiAda and compare
it with (i) simple combinations of GDA and adaptive stepsizes, which are commonly
used in practice, and (ii) NeAda with different adaptive mechanisms~\citep{yang2022nest}. Our experiments include test functions proposed by
\citet{yang2022nest}, the NC-SC distributional robustness optimization~\citep{sinha2018certifiable},
and training the NC-NC 
Wasserstein GAN with gradient penalty~\citep{gulrajani2017improved}.
We believe that this 
not only validates our theoretical results but also shows the potential of our algorithm in real-world scenarios.
To show the strength of being parameter-agnostic of TiAda,
in all the experiments, we merely select $\alpha=0.6$ and $\beta=0.4$
without further tuning those two hyper-parameters.
All experimental details including the neural network structure and
hyper-parameters are described in \Cref{sec:exp_detail}.

\subsection{Extensions to Other Adaptive Stepsizes and High-dimensional Models}
\label{sec:extension}

Although we design TiAda upon AdaGrad-Norm, it is easy and intuitive to apply
other adaptive schemes like Adam and AMSGrad.
To do so, for $z \in \{x, y\}$, we replace
the definition of $g^z_t$ and $v^z_{t+1}$ in line~\ref{firstm_update} and~\ref{secondm_update}
of \Cref{alg:tiada} to
\[
  g^z_t = \beta_t^z g^z_{t-1} + (1 - \beta_t^z) \nabla_z F(x_t, y_t; \xi^z_t), \quad
  v^z_{t+1} = \psi\left(v_0, \left\{ \norm*{\nabla_z F(x_i, y_i; \xi^z_i)}^2 \right\}_{i=0}^t \right),
\]
where $\{\beta_t^z\}$ is the momentum parameters and $\psi$ is the second moment
function. Some common stepsizes that fit in this generalized framework can be
seen in \Cref{tab:adaptive_schemes} in the appendix.
Since Adam is widely used in many deep learning tasks, we also implement
generalized TiAda with Adam stepsizes in our experiments for real-world applications,
and we label it ``TiAda-Adam''.

Besides generalizing TiAda to accommodate different stepsize schemes,
for high-dimensional models,
we also provide a coordinate-wise version of TiAda. 
Note that we cannot simply change everything in \Cref{alg:tiada} to be
coordinate-wise, because we use the gradients of $y$ in the stepsize of $x$
and there are no corresponding relationships between the coordinates of $x$ and
$y$. Therefore, in light of our intuition in \Cref{eq:factor_ex},
we use the global accumulated gradient norms to dynamically adjust the stepsize of $x$.
Denote the second moment (analogous to $v^x_{t+1}$ in \Cref{alg:tiada}) for the $i$-th coordinate of $x$ at the
$t$-th step as $v^x_{t+1, i}$ and globally $v^x_{t+1} := \sum_{i=1}^{d_1} v^x_{t+1, i}$.
We also use similar notations for $y$.
Then, the update for the $i$-th
parameter, i.e., $x^i$ and $y^i$, can be written as
\[
  \begin{cases}
  x^i_{t+1} = x_t^{i} - \frac{\left(v^x_{t+1}\right)^{\ra}}{\max\left\{v^x_{t+1}, v^y_{t+1}\right\}^{\ra}}
  \cdot \frac{\eta^x}{\left(v^x_{t+1, i}\right)^{\ra}} \nabla_{x^i} f(x_t, y_t) \\
  y^i_{t+1} = y_t^{i} + \frac{\eta^y}{\left(v^y_{t+1, i}\right)^{\rb}} \nabla_{y^i} f(x_t, y_t).
  \end{cases}
\]
Our results in the following subsections provide strong empirical evidence for the effectiveness of these TiAda variants, and  developing convergence guarantees for
them would be an interesting future work. We believe our proof techniques for TiAda, together with existing convergence results for coordinate-wise AdaGrad and AMSGrad~\citep{zhou2018convergence,chen2018convergence,defossez2020simple}, can shed light on the theoretical analyses of these variants.

\subsection{Test Functions}

\begin{figure}[t]
    \centering
    \begin{subfigure}[b]{0.255\textwidth}
      \centering
      \includegraphics[width=\textwidth]{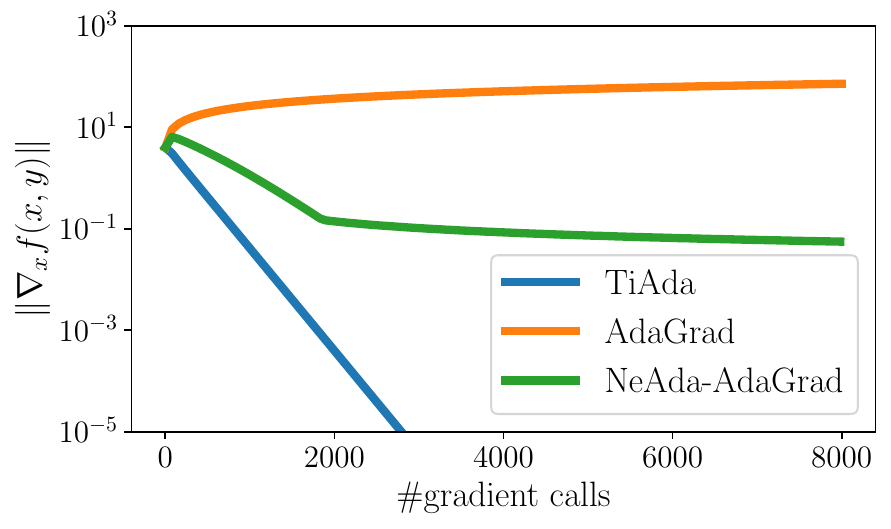}
      \caption{$r = 1$}
    \end{subfigure}
    \begin{subfigure}[b]{0.245\textwidth}
      \centering
      \includegraphics[width=\textwidth]{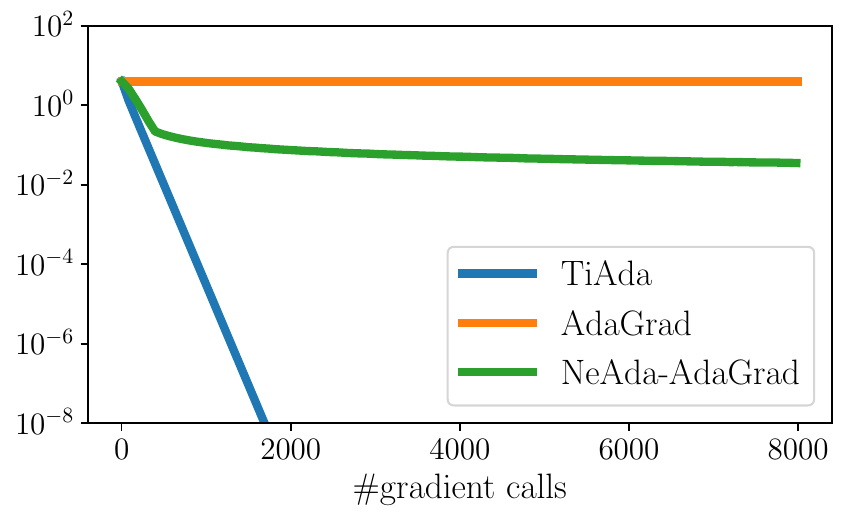}
      \caption{$r = 1/2$}
    \end{subfigure}
    \begin{subfigure}[b]{0.24\textwidth}
      \centering
      \includegraphics[width=\textwidth]{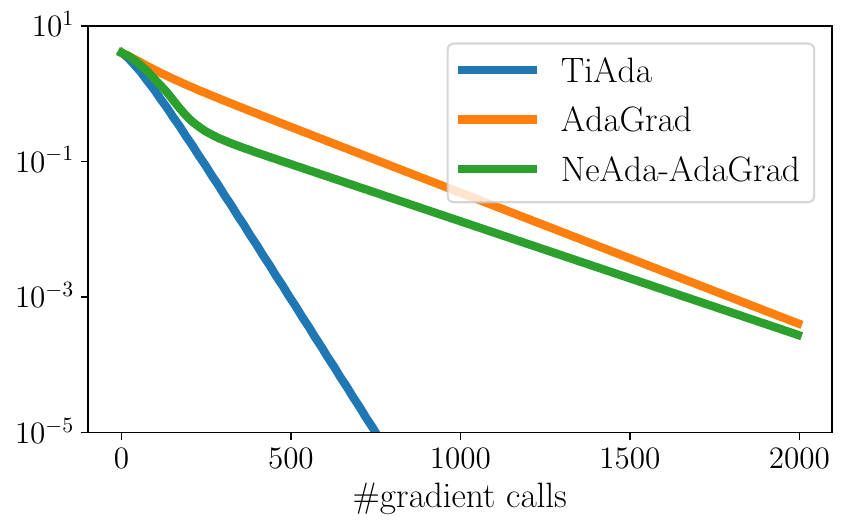}
      \caption{$r = 1/4 $}
    \end{subfigure}
    \begin{subfigure}[b]{0.24\textwidth}
      \centering
      \includegraphics[width=\textwidth]{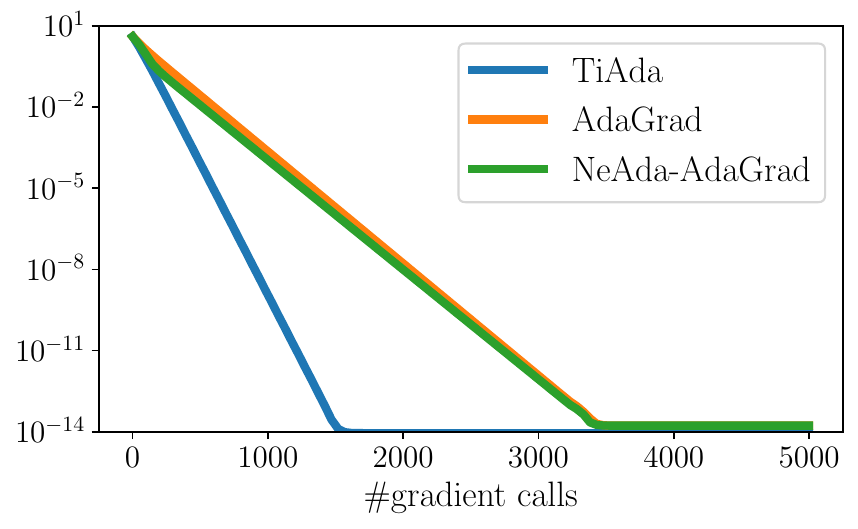}
      \caption{$r = 1/8 $}
    \end{subfigure}
    \begin{subfigure}[b]{0.25\textwidth}
      \centering
      \includegraphics[width=\textwidth]{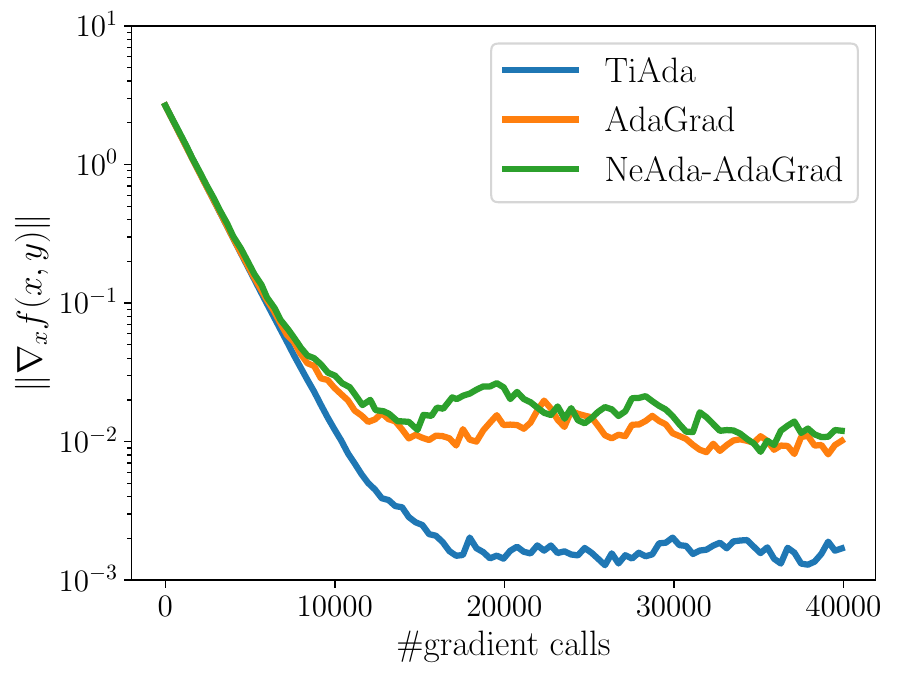}
      \caption{$r = 1/0.01$}
    \end{subfigure}
    \begin{subfigure}[b]{0.24\textwidth}
      \centering
      \includegraphics[width=\textwidth]{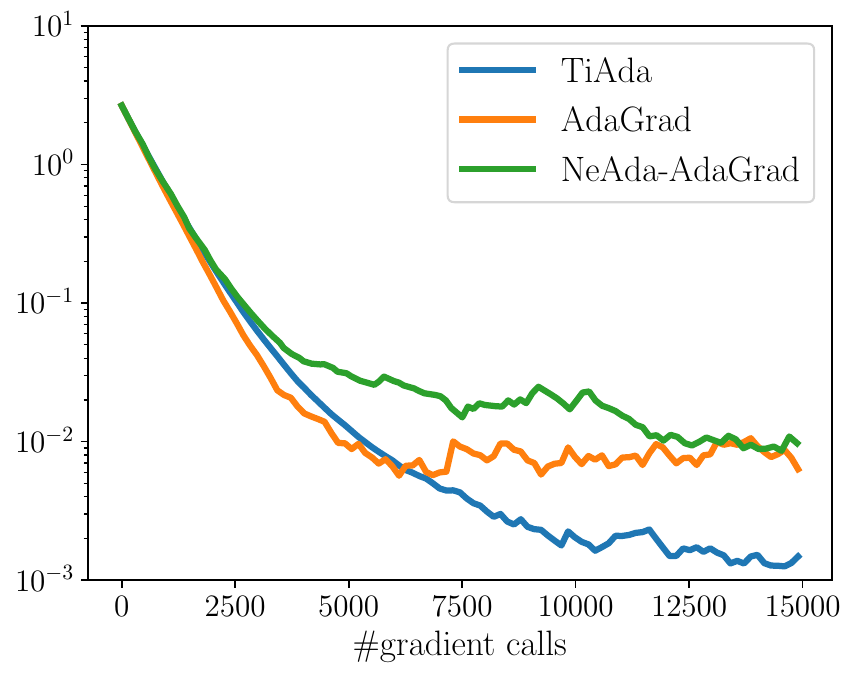}
      \caption{$r = 1/0.03$}
    \end{subfigure}
    \begin{subfigure}[b]{0.24\textwidth}
      \centering
      \includegraphics[width=\textwidth]{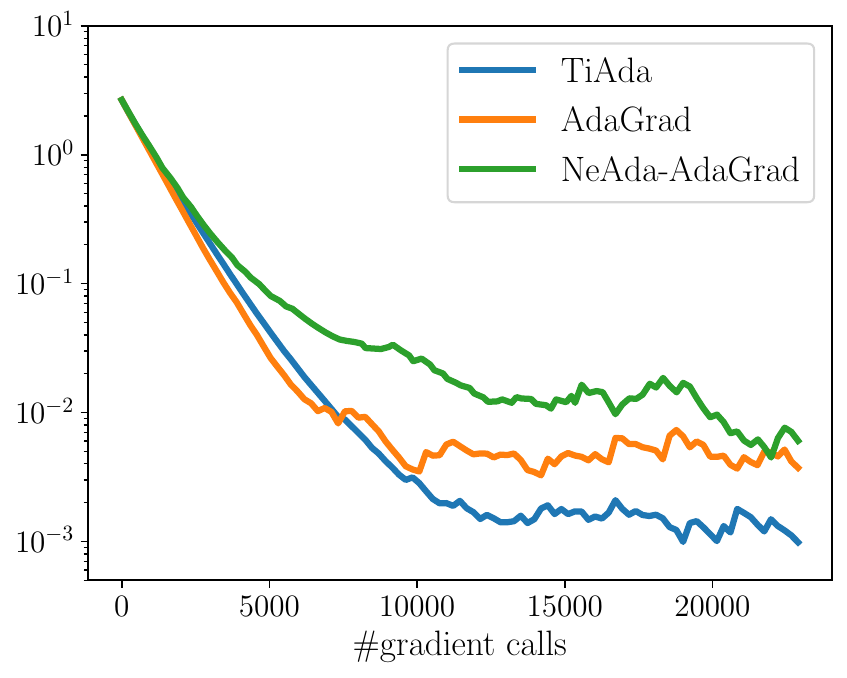}
      \caption{$r = 1/0.05 $}
    \end{subfigure}
    \begin{subfigure}[b]{0.24\textwidth}
      \centering
      \includegraphics[width=\textwidth]{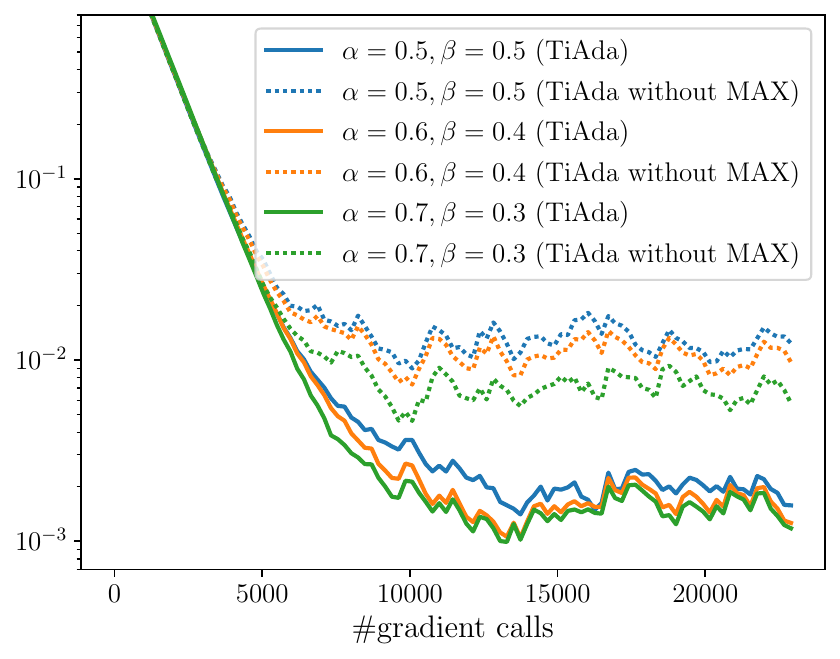}
      \caption{ablation study}
      \label{subfig:ablation_max}
    \end{subfigure}
    \caption{Comparison of algorithms on test functions. $r=\eta^x/\eta^y$ is
      the initial stepsize ratio. In the first row, we use the quadratic function~\eqref{eq:quad} with $L=2$ under deterministic gradient oracles. 
      For the second row, we test the methods on the McCormick function
      with noisy gradients.}
    \label{fig:quad_compare_other}
\end{figure}

Firstly, we examine TiAda on the quadratic function~\eqref{eq:quad} that shows the non-convergence of simple combinations of
GDA and adaptive stepsizes~\citep{yang2022nest}. Since our TiAda is based on AdaGrad, we compare it
to GDA with AdaGrad stepsize and NeAda-AdaGrad~\citep{yang2022nest}. The results are
shown in the first row of \Cref{fig:quad_compare_other}.
When the initial ratio is poor, TiAda and NeAda-AdaGrad
always converge while AdaGrad diverges. NeAda also suffers from slow convergence when the initial
ratio is poor, e.g., $1$ and $1/2$ after 2000 iterations.
In contrast, TiAda automatically balances the stepsizes and
converges fast under all ratios. 

For the stochastic case, we follow \citet{yang2022nest} and conduct
experiments on the McCormick function which is more complicated and 2-dimensional:
$
f(x, y) = \sin(x_1 + x_2) + (x_1 - x_2)^2 - \frac{3}{2}x_1 + \frac{5}{2}x_2 + 1 + x_1 y_1 + x_2 y_2 - \frac{1}{2}(y_1^2 + y_2^2).
$
TiAda consistently outperforms AdaGrad and NeAda-AdaGrad as demonstrated in
the second row of \Cref{fig:quad_compare_other} regardless of the initial ratio.
In this function, we also run an ablation study on the effect of
our design that uses max-operator in the update of $x$. We compare
TiAda with and its
variant without the max-operator, TiAda without MAX (\Cref{alg:tiada_wo_max} in the appendix) whose effective stepsizes of $x$ are
$\eta^x / \left(v^x_{t+1}\right)^{\alpha}$.
According to \Cref{subfig:ablation_max}, TiAda converges to
smaller gradient norms under all configurations of $\alpha$ and $\beta$.

\subsection{Distributional robustness optimization}

\begin{figure}[t]
    \centering
    \begin{subfigure}[b]{0.25\textwidth}
      \centering
      \includegraphics[width=\textwidth]{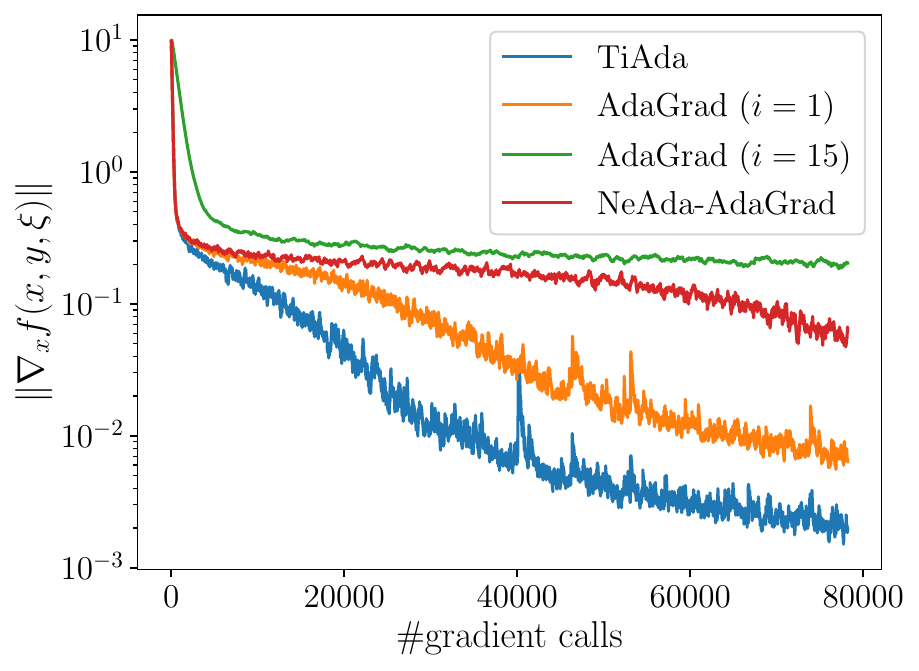}
      \caption{\scriptsize $\eta^x=0.1, \eta^y=0.05$}
    \end{subfigure}
    \begin{subfigure}[b]{0.237\textwidth}
      \centering
      \includegraphics[width=\textwidth]{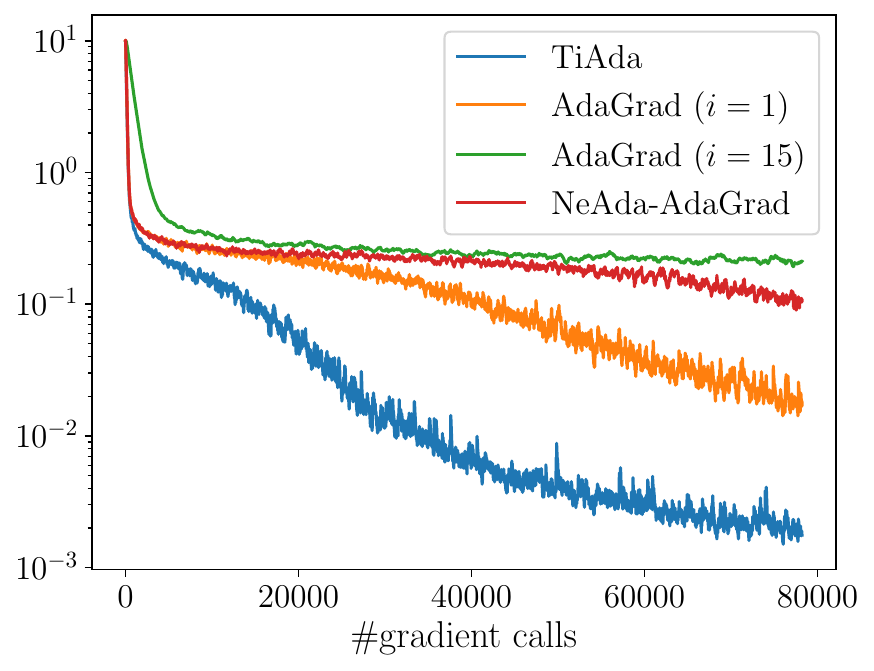}
      \caption{\scriptsize $\eta^x=0.1, \eta^y=0.1$}
    \end{subfigure}
    \begin{subfigure}[b]{0.237\textwidth}
      \centering
      \includegraphics[width=\textwidth]{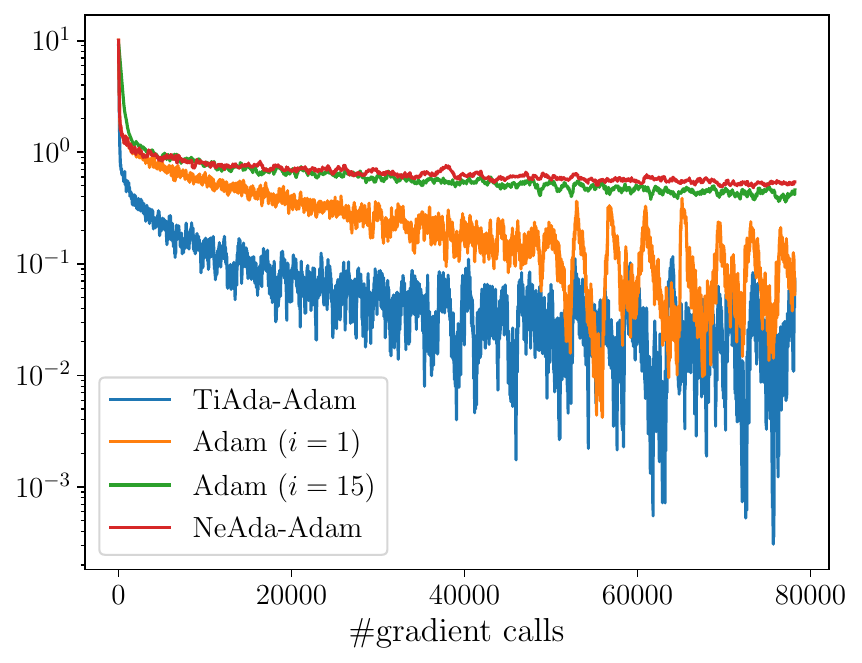}
      \caption{\scriptsize $\eta^x=0.001, \eta^y=0.1$}
      \label{subfig:adv_robust_adam_1}
    \end{subfigure}
    \begin{subfigure}[b]{0.237\textwidth}
      \centering
      \includegraphics[width=\textwidth]{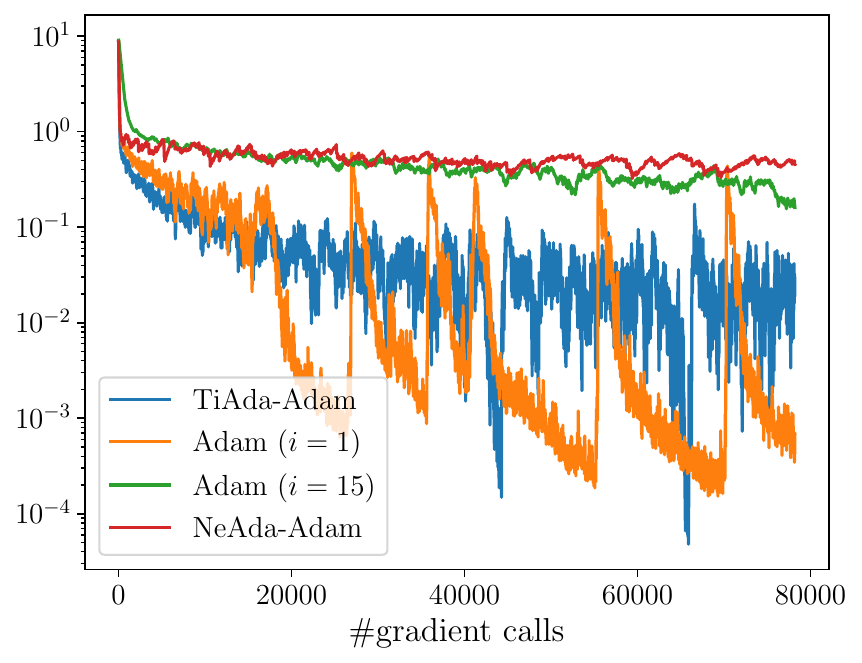}
      \caption{\scriptsize $\eta^x=0.001, \eta^y=0.001$}
      \label{subfig:adv_robust_adam_2}
    \end{subfigure}
    \caption{Comparison of the algorithms on distributional robustness optimization~\eqref{eq:dist_robust}.
    We use $i$ in the legend to indicate the number of inner loops. Here we present
    two sets of stepsize configurations for the comparisons
    of AdaGrad-like and Adam-like algorithms. Please refer to
    \Cref{sec:add_exp} for  extensive experiments on larger ranges of
    stepsizes, and it will be shown that TiAda is the best among all stepsize combinations in our grid.}
    \label{fig:adv_robust}
\end{figure}

In this subsection, we consider the %
distributional robustness optimization~\citep{sinha2018certifiable}.
We target training the model weights, the primal variable $x$, to be robust
to the perturbations in the image inputs, the dual variable $y$. The 
problem
can be formulated as:
\[
    \min_x \max_{y = [y_1,\dots,y_n]} \frac{1}{n} \sum_{i=1}^n f_i(x, y_i) - \gamma \norm*{y_i - v_i}^2,
    \numberthis \label{eq:dist_robust}
\]
where $f_i$ is the loss function of the $i$-th sample,
$v_i$ is the $i$-th input image, and $y_i$ is the corresponding perturbation.
There are a total of $n$ samples and $\gamma$ is a trade-off hyper-parameter between the original loss and the penalty of the perturbations.
If $\gamma$ is large enough, the problem is NC-SC.

We conduct the experiments on the MNIST dataset~\citep{lecun1998mnist}.
In the left two plots of \Cref{fig:adv_robust}, we compare TiAda with
AdaGrad and NeAda-AdaGrad in terms of convergence.
Since it is common in practice to update $y$ 15 times after each $x$
update~\citep{sinha2018certifiable} for better generalization error,
we implement AdaGrad using both single and 15 iterations of inner loop~(update of $y$).
In order to show that TiAda is more robust to the initial stepsize ratio, we compare
two sets of initial stepsize configurations with two different ratios.
In both cases, TiAda outperforms NeAda and AdaGrad, especially when $\eta^x = \eta^y=0.1$,
the performance gap is large. In the right two plots of \Cref{fig:adv_robust},
the Adam variants are compared. In this case, we find that TiAda 
is not only faster, but also more stable comparing to Adam with one inner loop iteration.

\subsection{Generative Adversarial Networks}

\begin{figure}[t]
  \begin{center}
    \includegraphics[width=0.6\textwidth]{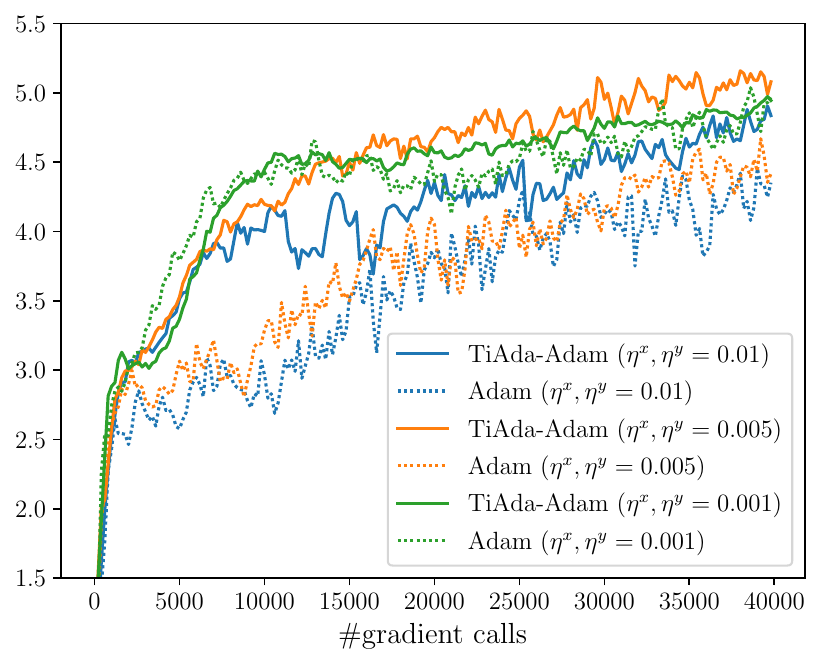}
  \caption{Inception score  on WGAN-GP.}
  \label{fig:gan}
  \end{center}
\end{figure}

Another successful and popular application of minimax optimization is
generative adversarial networks. In this task, a discriminator (or critic) is trained
to distinguish whether an image is from the dataset.
At the same time, a generator is mutually trained to synthesize samples with the
same distribution as the training dataset so as to fool the discriminator.
We use WGAN-GP loss~\citep{gulrajani2017improved}, which imposes the
discriminator to be a $1$-Lipschitz function,
with CIFAR-10 dataset~\citep{krizhevsky2009learning}
in our experiments.

Since TiAda is a single-loop algorithm, for fair comparisons, we also update
the discriminator only once for each generator update in  Adam.
In \Cref{fig:gan}, we plot the inception scores~\citep{salimans2016improved}
of TiAda-Adam and Adam under different initial stepsizes.
We use the same color for the same initial stepsizes, and different line styles
to distinguish the two methods, i.e., solid lines for TiAda-Adam and
dashed lines for Adam. For all the three initial stepsizes we consider,
TiAda-Adam achieves higher inception scores.
Also, TiAda-Adam is more robust to initial stepsize selection, as the gap between
different solid lines at the end of training is smaller than the dashed lines.

\section{Conclusion}

In this work, we bring in adaptive stepsizes to nonconvex minimax problems
in a parameter-agnostic manner. 
We designed the first time-scale adaptive algorithm, TiAda,
which  progressively adjusts the effective stepsize ratio and
reaches the desired time-scale separation. TiAda is also noise adaptive
and does not require large batchsizes compared with the existing parameter-agnostic algorithm for nonconvex minimax optimization.
Furthermore,  TiAda is able to achieve
optimal and near-optimal complexities respectively wtih deterministic and stochastic
gradient oracles.
We also empirically showcased the advantages of TiAda over NeAda and GDA with
adaptive stepsizes on several tasks, including simple test
functions, as well as NC-SC and NC-NC
real-world applications. It remains an interesting problem to study whether TiAda can escape stationary points that are not local optimum, like adaptive methods for minimization problems \citep{staib2019escaping}.

\bibliography{iclr2023_conference}
\bibliographystyle{plainnat}

\clearpage

\appendix

\section{Supplementary to Experiments}

\begin{table}[t]
  \caption{Stepsize schemes fit in generalized TiAda. See also \citet{yang2022nest}.}
  \label{tab:adaptive_schemes}
  \begin{center}
  \begin{tabular}{lll} \toprule
  Algorithms & first moment parameter $\beta_t$ &  second moment function $\psi\left(v_0, \{u^2_i\}_{i=0}^{t}\right)$   \\ \midrule
  AdaGrad~(TiAda) & $\beta_t = 0$     & $v_0 + \sum_{i=0}^{t} u_i^2$                                                    \\
  GDA & $\beta_t = 0$     &  $1$                                                    \\
  Adam    & $0 < \beta_t < 1$ & $\gamma^{t+1}v_0 + (1-\gamma)\sum_{i=0}^t\gamma^{t-i}u_i^2$                     \\
  AMSGrad & $0 < \beta_t < 1$ & $\max_{m=0,\dots, t} \gamma^{m+1}v_0 + (1-\gamma)\sum_{i=0}^m\gamma^{m-i}u_i^2$
  \\ \bottomrule
  \end{tabular}
  \end{center}
\end{table}

\subsection{Experimental Details}
\label{sec:exp_detail}

In this section, we will summarize the experimental settings and hyper-parameters
used. As we mentioned, since we try to develop a parameter-agnostic algorithm
without tuning the hyper-parameters much, if not specified, we simply use
$\alpha=0.6$ and $\beta=0.4$ for all experiments. For fair comparisons,
we used the same hyper-parameters when comparing our TiAda with other algorithms.

\paragraph{Test Functions}
For \Cref{fig:compare_ratio} and the first row of \Cref{fig:quad_compare_other},
we conduct experiments
on problem~\eqref{eq:quad} with $L=2$. We use initial stepsize $\eta^y=0.2$
and initial point $(1, 0.01)$ for all runs.
As for the McCormick function
used in the second row of
\Cref{fig:quad_compare_other}, we chose $\eta^y=0.01$, and the
noises added to the gradients are from zero-mean Gaussian distribution
with variance $0.01$.

\paragraph{Distributional Robustness Optimization}
For results shown in \Cref{fig:adv_robust,fig:adv_robust_all_adagrad,fig:adv_robust_all_adam}, we adapt code from \citet{dist_robust_code},
and used the same hyper-parameter setting as \citet{sinha2018certifiable,sebbouh2022randomized},
i.e., $\gamma=1.3$.
The model we used is a three layer convolutional neural network~(CNN) with
a final fully-connected layer. For each layer, batch normalization and ELU activation
are used. The width of each layer is $(32, 64, 128, 512)$. The setting is the
same as \citet{sinha2018certifiable,yang2022nest}. We set the batchsize as $128$,
and for the Adam-like optimizers,
including Adam, NeAda-Adam and TiAda-Adam, we use $\beta_1=0.9, \beta_2=0.999$
for the first moment and second moment parameters.

\paragraph{Generative Adversarial Networks}
For this part, we use the code adapted from \citet{gan_code}.
To produce the results in \Cref{fig:gan}, a four layer CNN and a four layer
CNN with transpose convolution layers are used respectively for the discriminator
and generator. Following a similar setting as \citet{daskalakis2018training},
we set batchsize as $512$, the dimension of latent variable as $50$ and
the weight of gradient penalty term as $10^{-4}$.
For the Adam-like optimizers, we set $\beta_1=0.5, \beta_2=0.9$.
To get the inception score, we feed the pre-trained inception network
with 8000 synthesized samples.

\subsection{Ablation Study on Convergence Behavior with Different $\alpha$ and $\beta$}
\label{subsec:ablation}

We conduct experiments on the quadratic minimax problem (\ref{eq:quad}) with $L = 2$ to study the effect of hyper-parameters $\alpha$ and $\beta$ on the convergence behavior of TiAda. As discussed in \Cref{sec:intro,subsec:stochastic}, we refer to the period before the stepsize ratio reduce to the convergence threshold as Stage I, and the period after that as Stage II. In order to accentuate the difference between these two stages, we pick a large initial stepsize ratio $\eta^x/\eta^y = 20$. We compare 4 different pairs of $\alpha$ and $\beta$: $\alpha \in \{0.59, 0.6,0.61, 0.62\}$ and $\beta = 1-\alpha$. From \Cref{fig:compare_para}, we observed that as soon as TiAda enters Stage II, the norm of gradients start to drop. Moreover, the closer $\alpha$ and $\beta$ are to 0.5, the more time TiAda remains in Stage I, which confirms the intuitions behind our analysis in \Cref{subsec:stochastic}.

\begin{figure}[ht]
    \centering
    \begin{subfigure}[b]{0.45\textwidth}
      \centering
      \includegraphics[width=\textwidth]{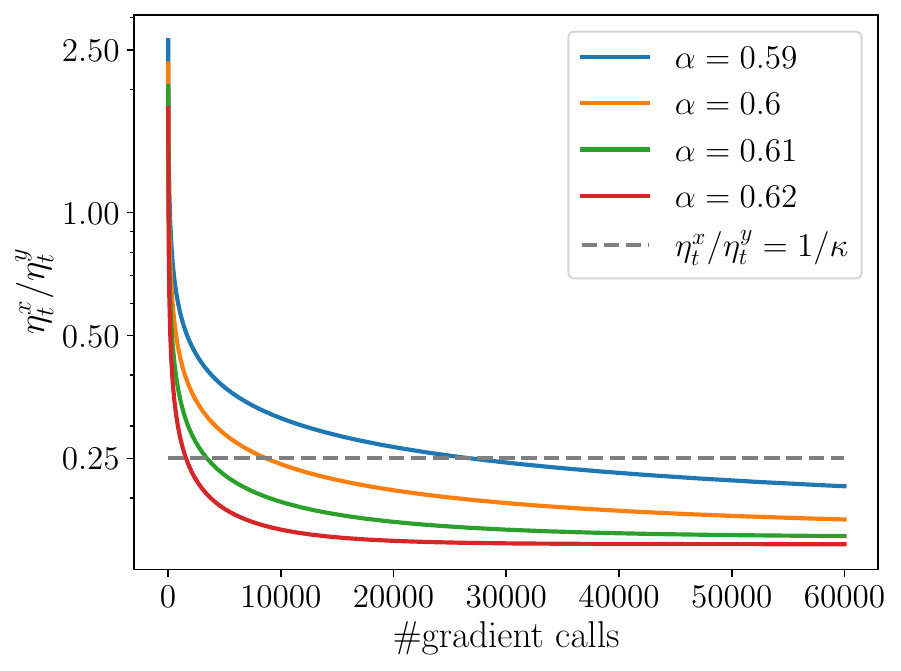}
      \caption{effective stepsize ratio}
    \end{subfigure}
    \begin{subfigure}[b]{0.45\textwidth}
      \centering
      \includegraphics[width=\textwidth]{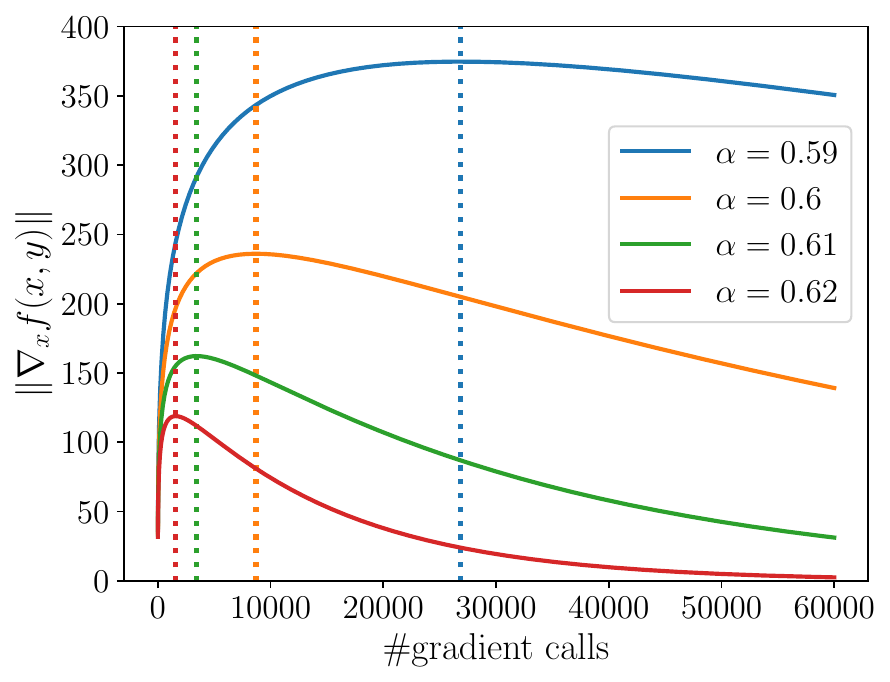}
      \caption{convergence}
    \end{subfigure}
    \caption{Illustration of the effect of $\alpha$ and $\beta$ on the two stages in TiAda's time-scale adaptation process. We set $\beta=1 - \alpha$. The dashed line on the right plot represents
    the first iteration when the effective stepsize ratio is below $1/\rk$.}
    \label{fig:compare_para}
\end{figure}

\subsection{Additional Experiments on Distributional Robustness Optimization}
\label{sec:add_exp}

\begin{figure}[!ht]
    \centering
    \begin{subfigure}[b]{0.24\textwidth}
      \centering
      \includegraphics[width=\textwidth]{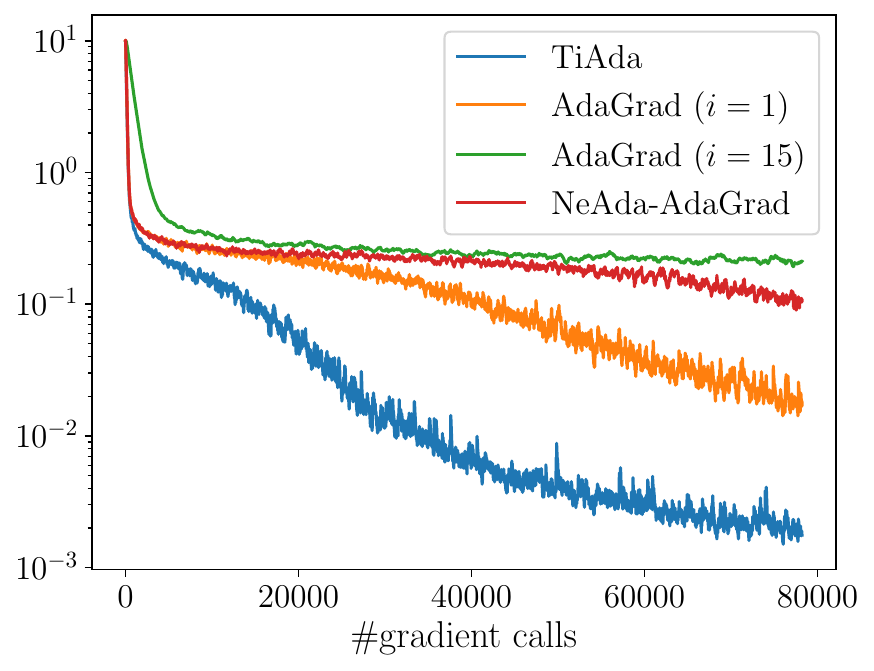}
      \caption{\scriptsize $\eta^x=0.1, \eta^y=0.1$}
    \end{subfigure}
    \begin{subfigure}[b]{0.24\textwidth}
      \centering
      \includegraphics[width=\textwidth]{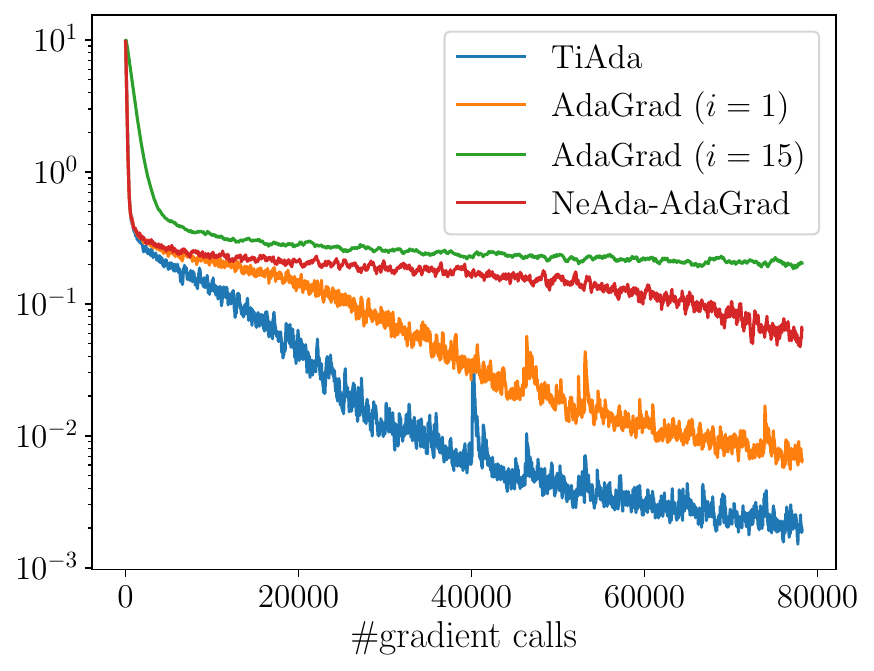}
      \caption{\scriptsize $\eta^x=0.1, \eta^y=0.05$}
    \end{subfigure}
    \begin{subfigure}[b]{0.24\textwidth}
      \centering
      \includegraphics[width=\textwidth]{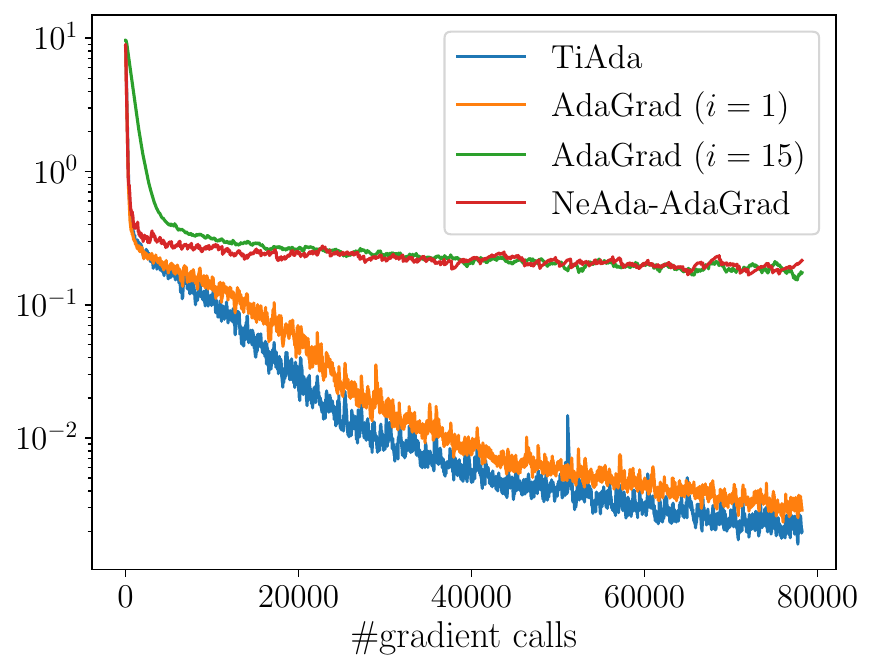}
      \caption{\scriptsize $\eta^x=0.1, \eta^y=0.01$}
    \end{subfigure}
    \begin{subfigure}[b]{0.24\textwidth}
      \centering
      \includegraphics[width=\textwidth]{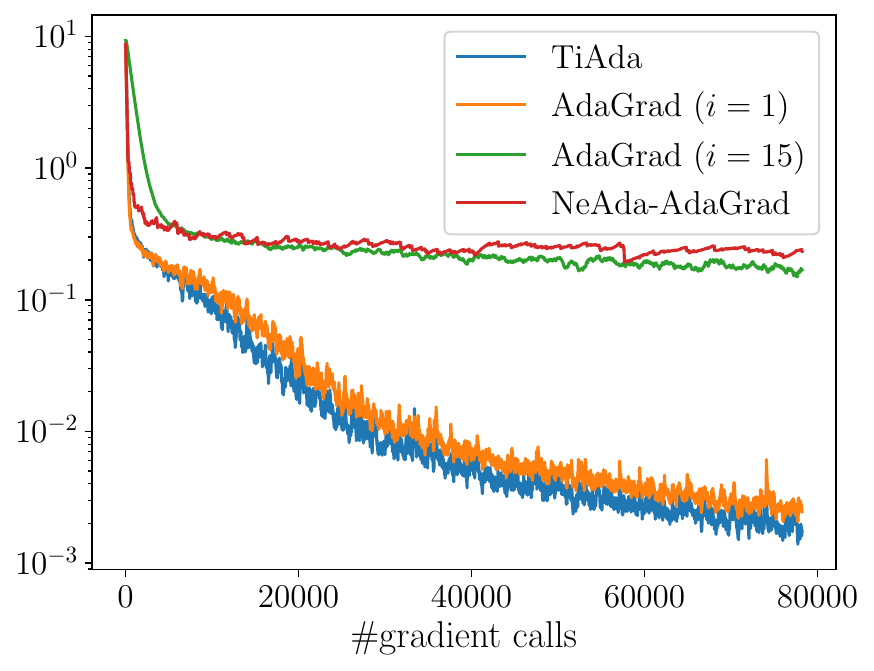}
      \caption{\scriptsize $\eta^x=0.1, \eta^y=0.005$}
    \end{subfigure}
    \begin{subfigure}[b]{0.24\textwidth}
      \centering
      \includegraphics[width=\textwidth]{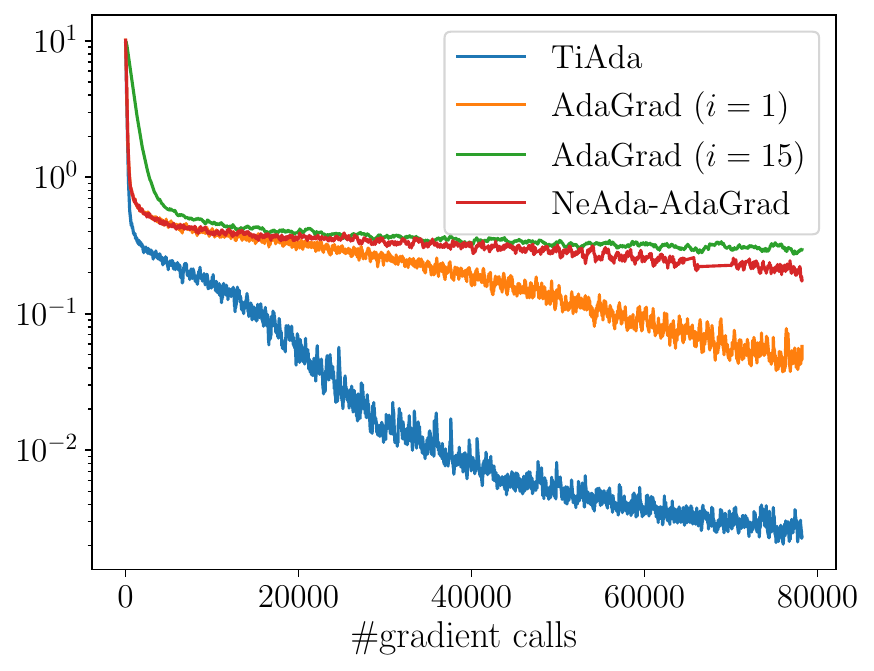}
      \caption{\scriptsize $\eta^x=0.05, \eta^y=0.1$}
    \end{subfigure}
    \begin{subfigure}[b]{0.24\textwidth}
      \centering
      \includegraphics[width=\textwidth]{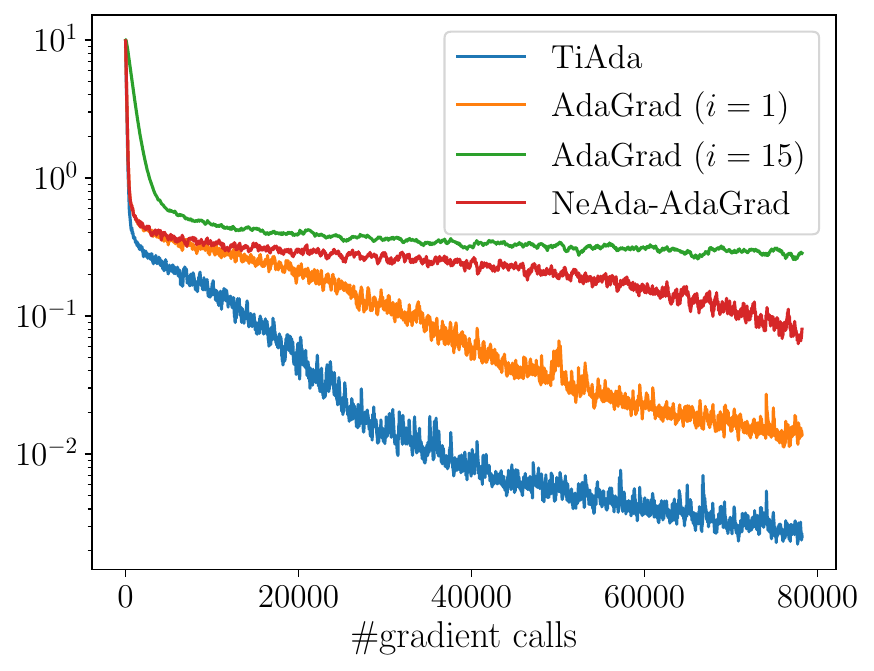}
      \caption{\scriptsize $\eta^x=0.05, \eta^y=0.05$}
    \end{subfigure}
    \begin{subfigure}[b]{0.24\textwidth}
      \centering
      \includegraphics[width=\textwidth]{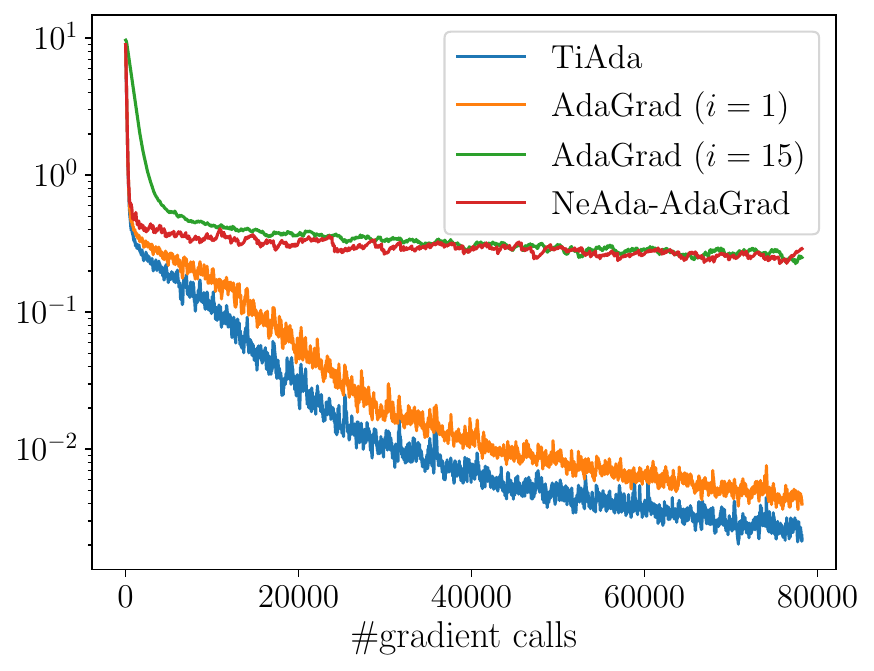}
      \caption{\scriptsize $\eta^x=0.05, \eta^y=0.01$}
    \end{subfigure}
    \begin{subfigure}[b]{0.24\textwidth}
      \centering
      \includegraphics[width=\textwidth]{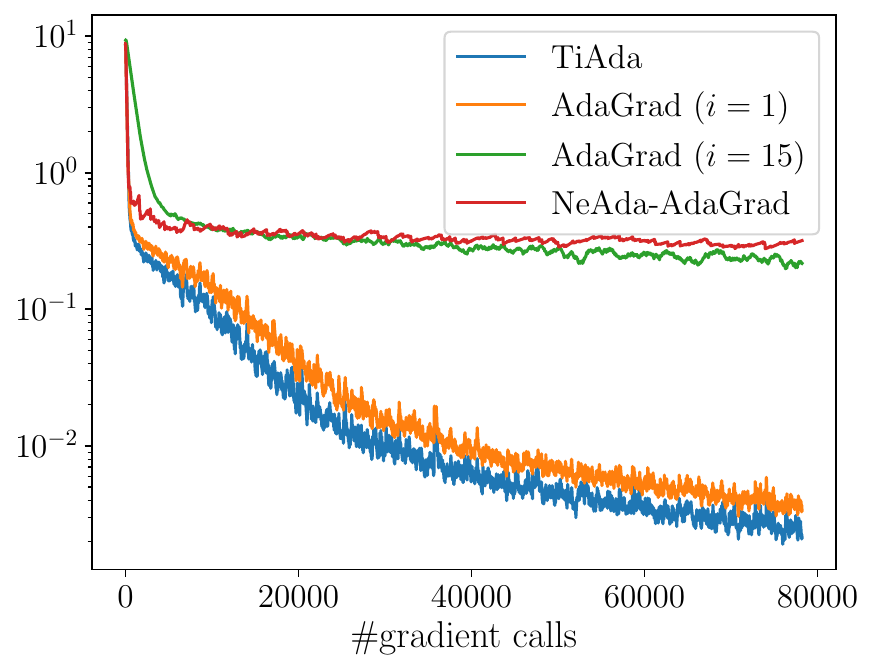}
      \caption{\scriptsize $\eta^x=0.05, \eta^y=0.005$}
    \end{subfigure}
    \begin{subfigure}[b]{0.24\textwidth}
      \centering
      \includegraphics[width=\textwidth]{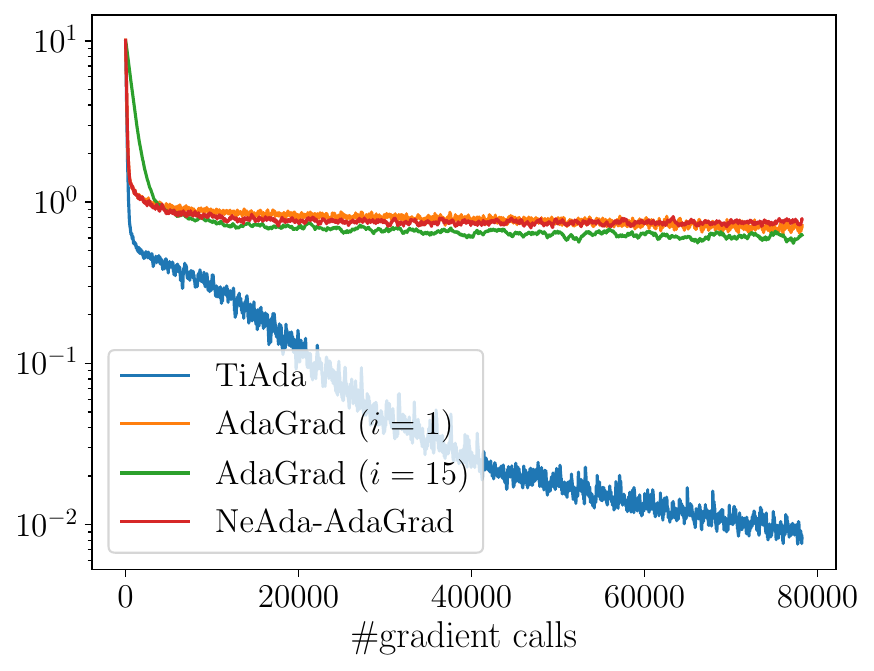}
      \caption{\scriptsize $\eta^x=0.01, \eta^y=0.1$}
    \end{subfigure}
    \begin{subfigure}[b]{0.24\textwidth}
      \centering
      \includegraphics[width=\textwidth]{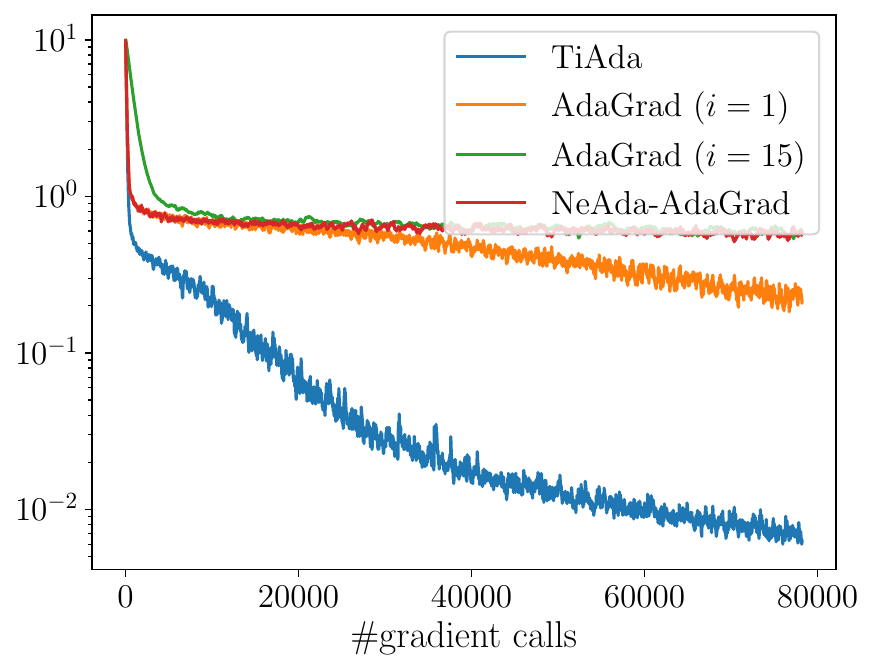}
      \caption{\scriptsize $\eta^x=0.01, \eta^y=0.05$}
    \end{subfigure}
    \begin{subfigure}[b]{0.24\textwidth}
      \centering
      \includegraphics[width=\textwidth]{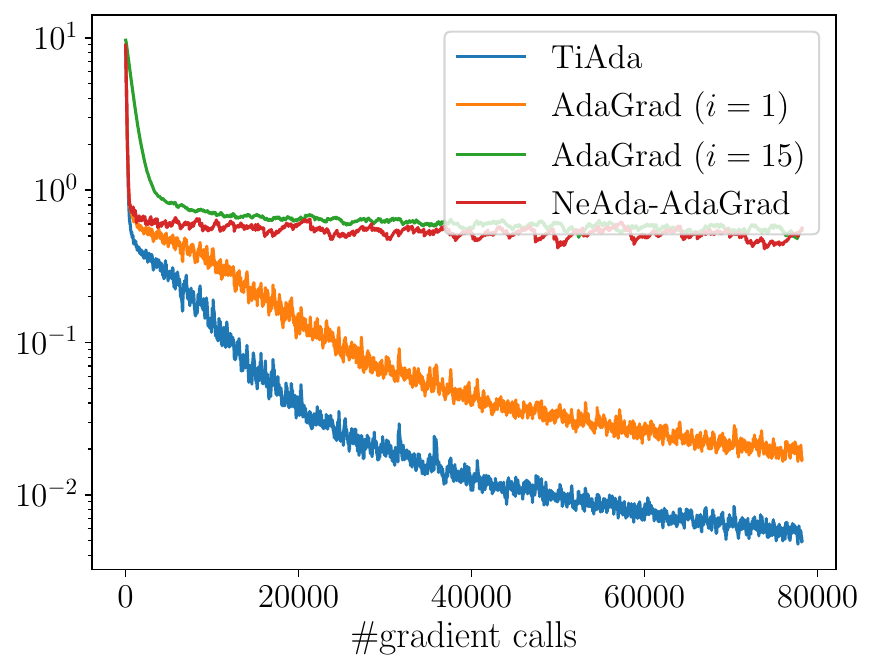}
      \caption{\scriptsize $\eta^x=0.01, \eta^y=0.01$}
    \end{subfigure}
    \begin{subfigure}[b]{0.24\textwidth}
      \centering
      \includegraphics[width=\textwidth]{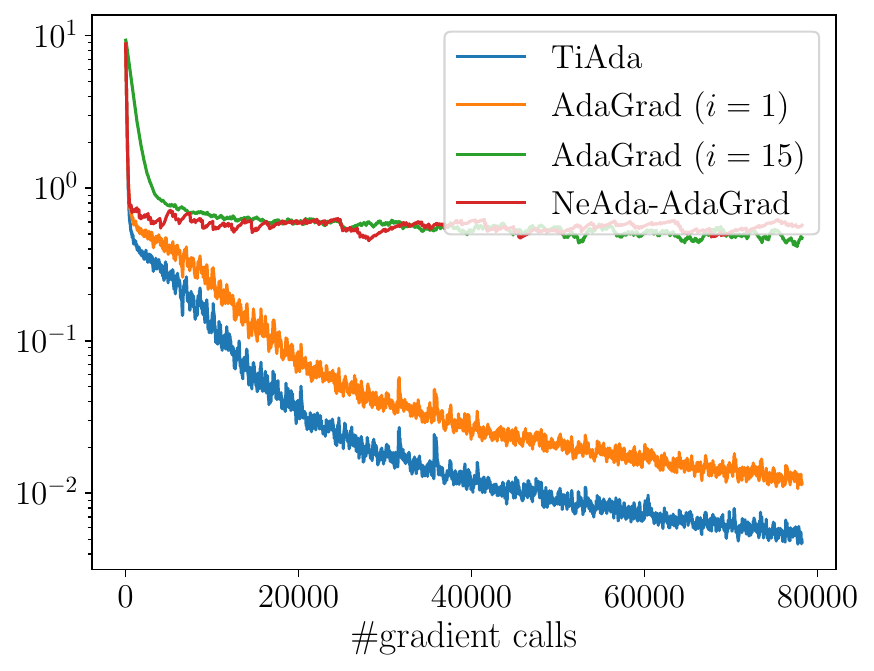}
      \caption{\scriptsize $\eta^x=0.01, \eta^y=0.005$}
    \end{subfigure}
    \begin{subfigure}[b]{0.24\textwidth}
      \centering
      \includegraphics[width=\textwidth]{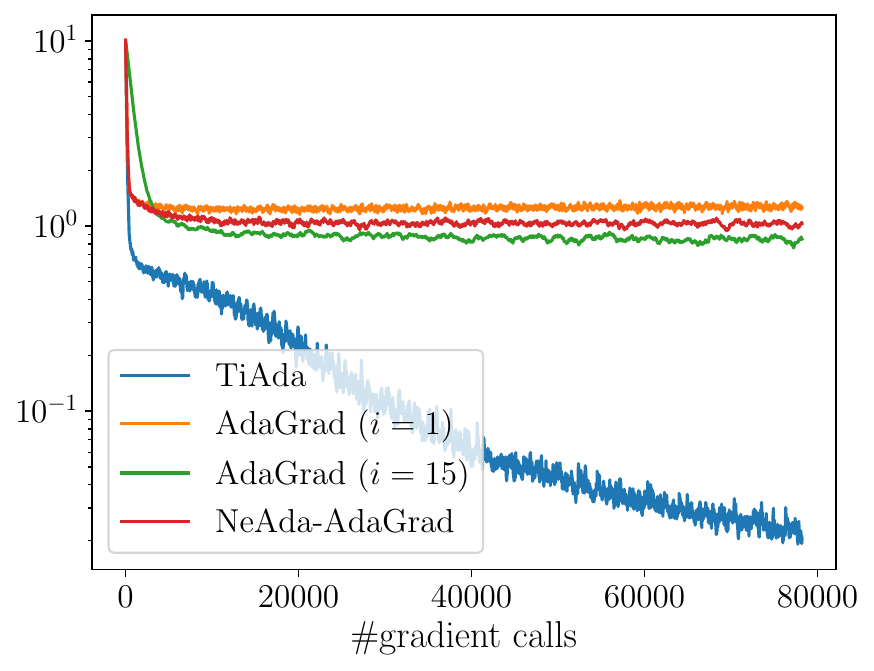}
      \caption{\scriptsize $\eta^x=0.005, \eta^y=0.1$}
    \end{subfigure}
    \begin{subfigure}[b]{0.24\textwidth}
      \centering
      \includegraphics[width=\textwidth]{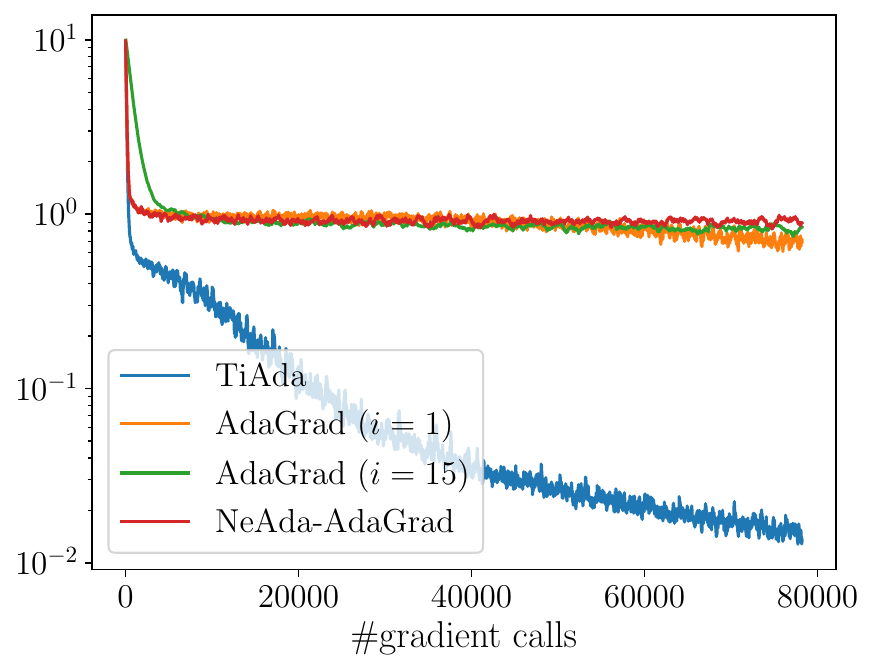}
      \caption{\scriptsize $\eta^x=0.005, \eta^y=0.05$}
    \end{subfigure}
    \begin{subfigure}[b]{0.24\textwidth}
      \centering
      \includegraphics[width=\textwidth]{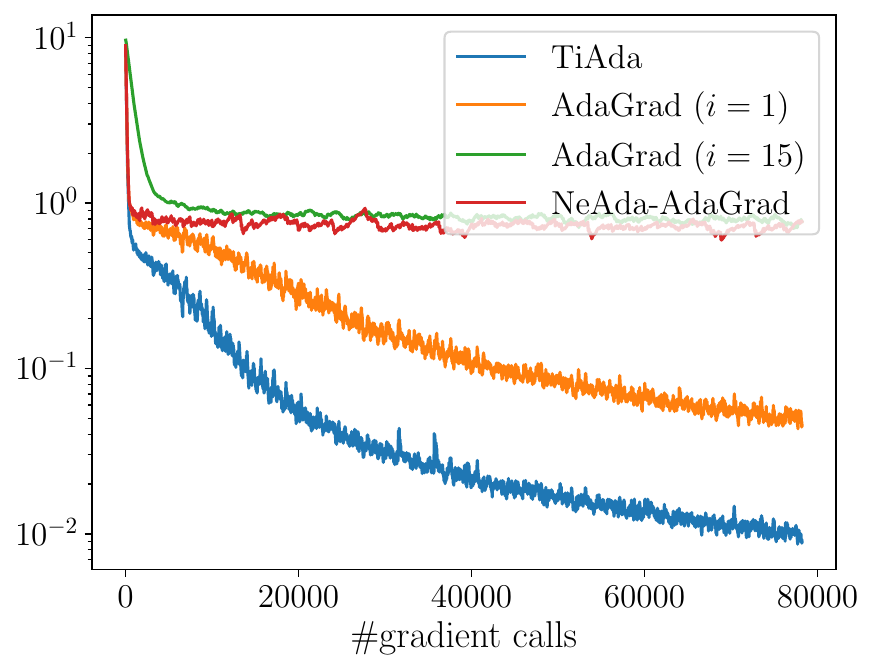}
      \caption{\scriptsize $\eta^x=0.005, \eta^y=0.01$}
    \end{subfigure}
    \begin{subfigure}[b]{0.24\textwidth}
      \centering
      \includegraphics[width=\textwidth]{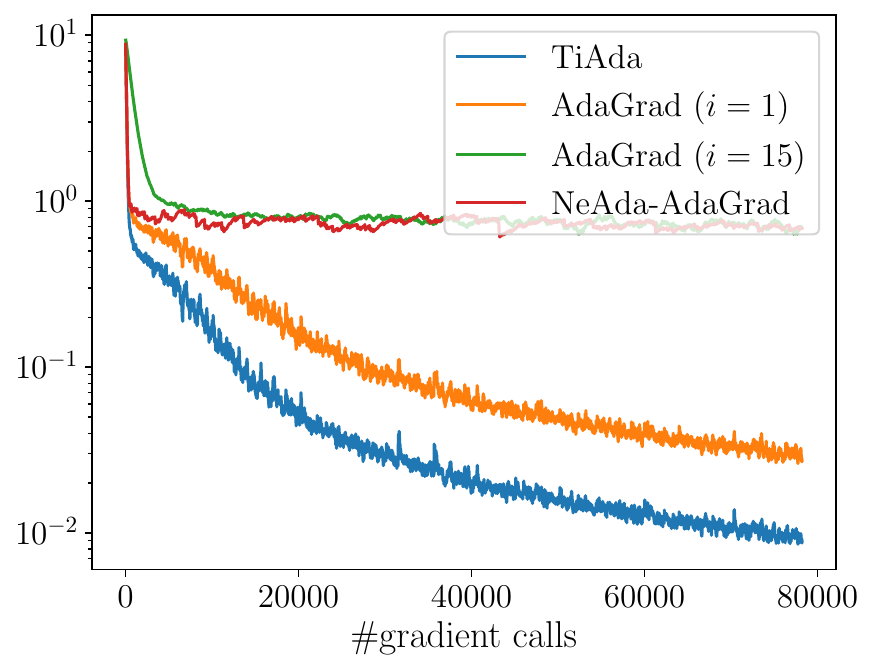}
      \caption{\scriptsize $\eta^x=0.005, \eta^y=0.005$}
    \end{subfigure}
    \caption{Gradient norms in $x$ of AdaGrad-like algorithms on distributional robustness optimization~\eqref{eq:dist_robust}. We use $i$ in the legend to indicate the number of inner loops.}
    \label{fig:adv_robust_all_adagrad}
\end{figure}

\begin{figure}[!ht]
    \centering
    \begin{subfigure}[b]{0.24\textwidth}
      \centering
      \includegraphics[width=\textwidth]{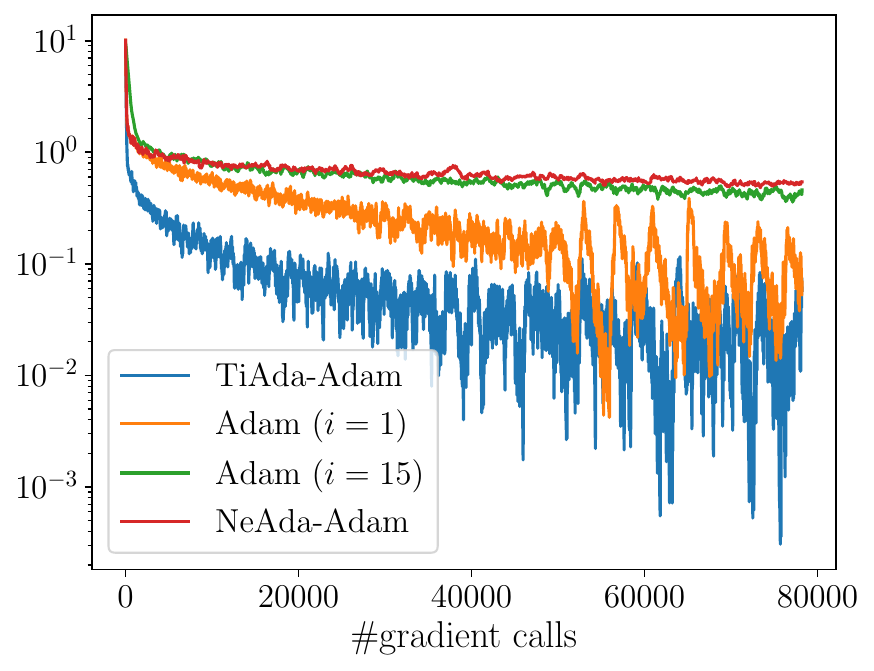}
      \caption{\tiny $\eta^x=0.001, \eta^y=0.1$}
    \end{subfigure}
    \begin{subfigure}[b]{0.24\textwidth}
      \centering
      \includegraphics[width=\textwidth]{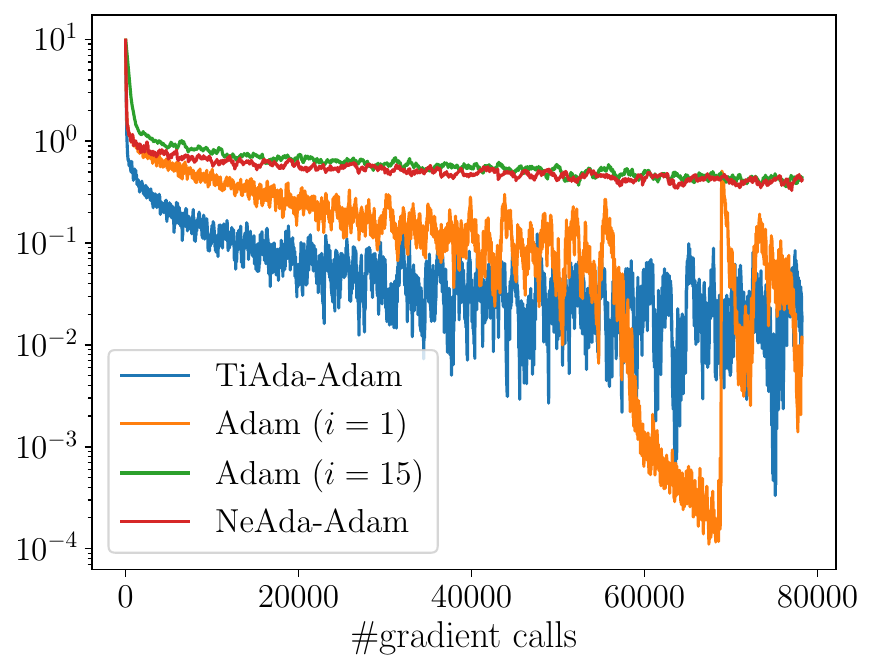}
      \caption{\tiny $\eta^x=0.001, \eta^y=0.05$}
    \end{subfigure}
    \begin{subfigure}[b]{0.24\textwidth}
      \centering
      \includegraphics[width=\textwidth]{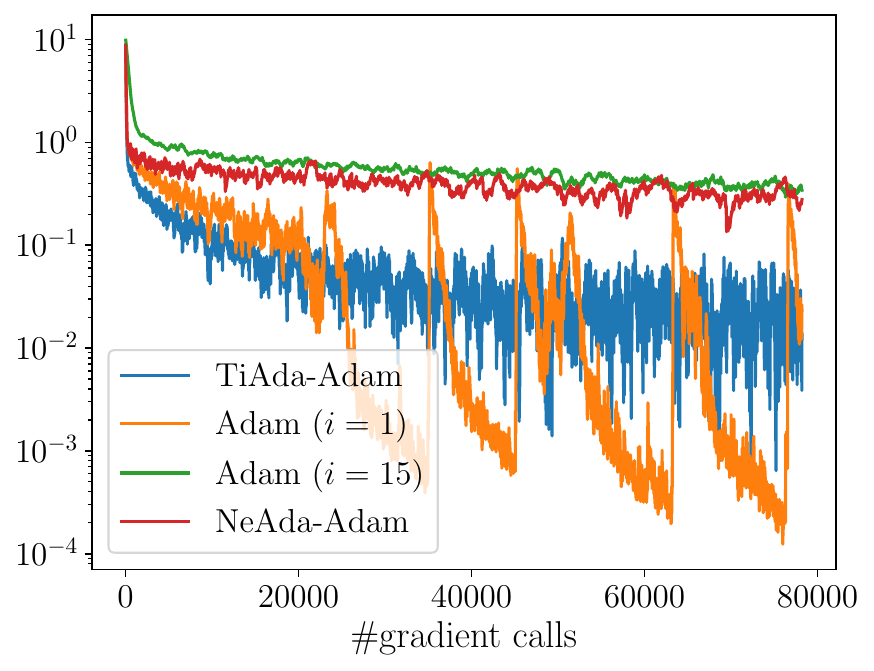}
      \caption{\tiny $\eta^x=0.001, \eta^y=0.005$}
    \end{subfigure}
    \begin{subfigure}[b]{0.24\textwidth}
      \centering
      \includegraphics[width=\textwidth]{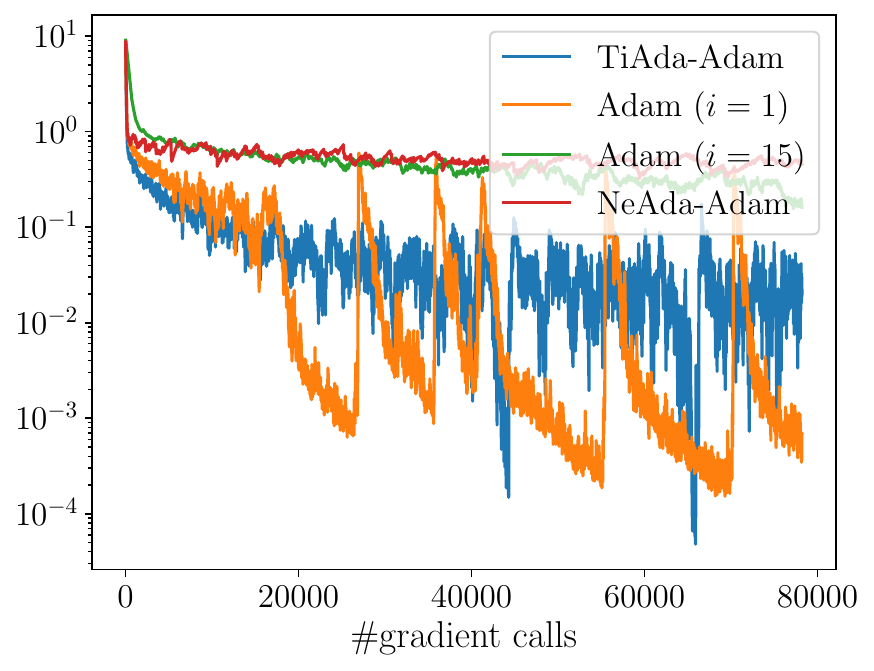}
      \caption{\tiny $\eta^x=0.001, \eta^y=0.001$}
    \end{subfigure}
    \begin{subfigure}[b]{0.24\textwidth}
      \centering
      \includegraphics[width=\textwidth]{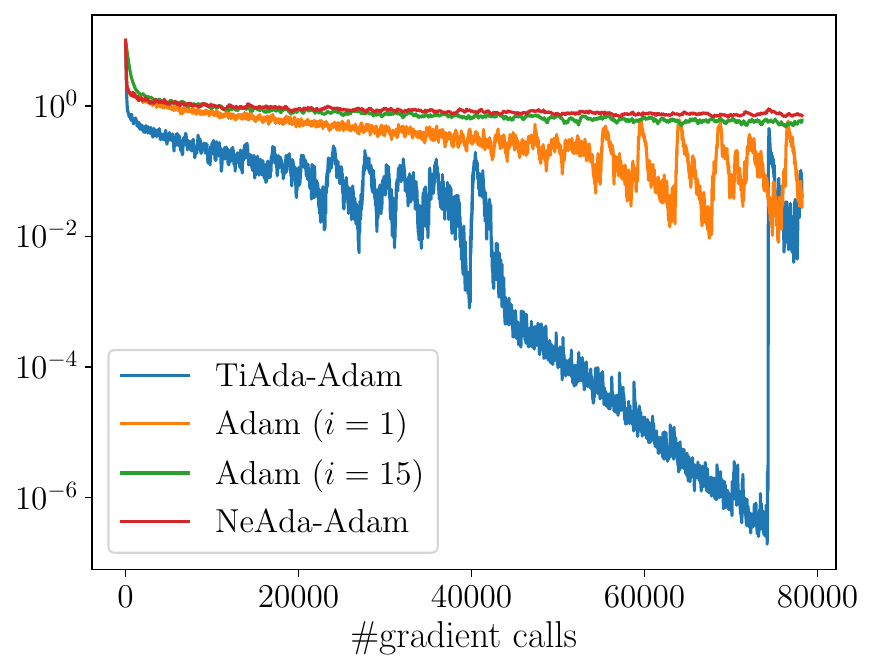}
      \caption{\tiny $\eta^x=0.0005, \eta^y=0.1$}
    \end{subfigure}
    \begin{subfigure}[b]{0.24\textwidth}
      \centering
      \includegraphics[width=\textwidth]{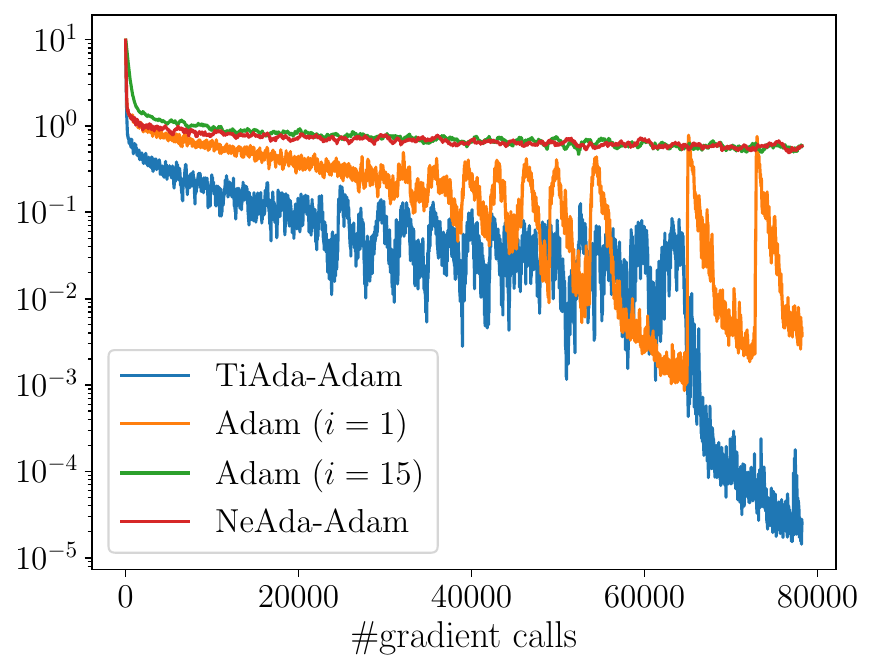}
      \caption{\tiny $\eta^x=0.0005, \eta^y=0.05$}
    \end{subfigure}
    \begin{subfigure}[b]{0.24\textwidth}
      \centering
      \includegraphics[width=\textwidth]{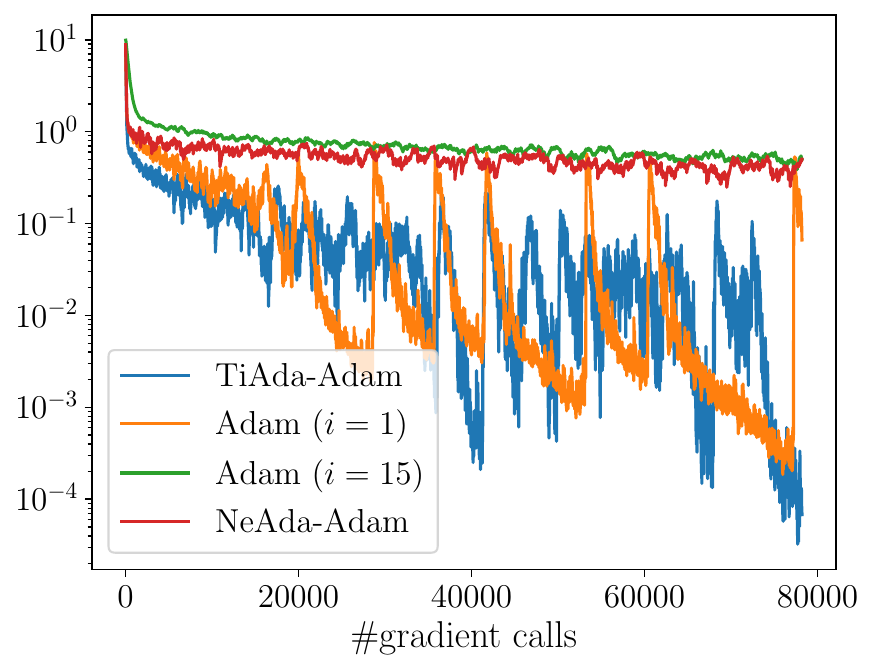}
      \caption{\tiny $\eta^x=0.0005, \eta^y=0.005$}
    \end{subfigure}
    \begin{subfigure}[b]{0.24\textwidth}
      \centering
      \includegraphics[width=\textwidth]{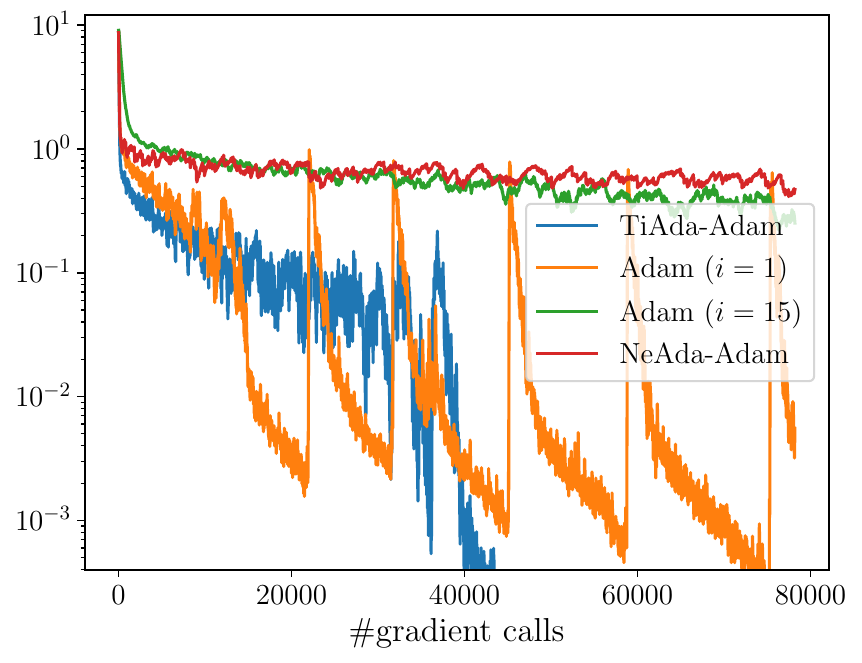}
      \caption{\tiny $\eta^x=0.0005, \eta^y=0.001$}
    \end{subfigure}
    \begin{subfigure}[b]{0.24\textwidth}
      \centering
      \includegraphics[width=\textwidth]{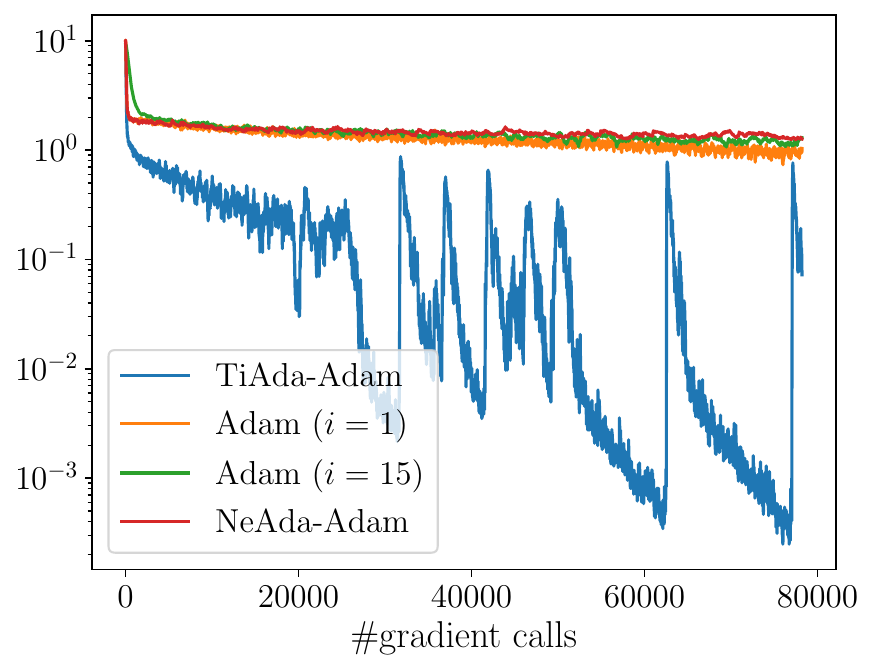}
      \caption{\tiny $\eta^x=0.0001, \eta^y=0.1$}
    \end{subfigure}
    \begin{subfigure}[b]{0.24\textwidth}
      \centering
      \includegraphics[width=\textwidth]{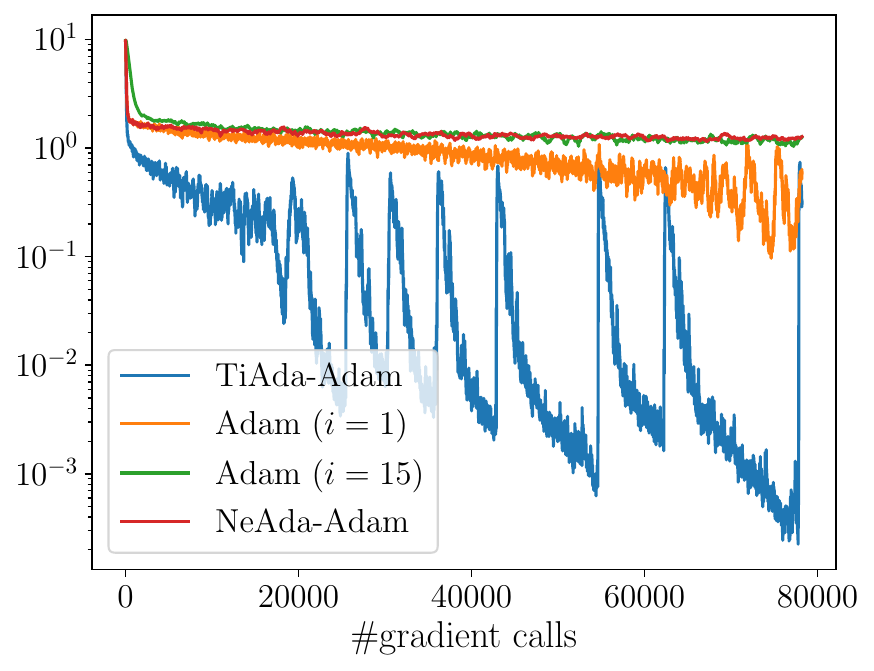}
      \caption{\tiny $\eta^x=0.0001, \eta^y=0.05$}
    \end{subfigure}
    \begin{subfigure}[b]{0.24\textwidth}
      \centering
      \includegraphics[width=\textwidth]{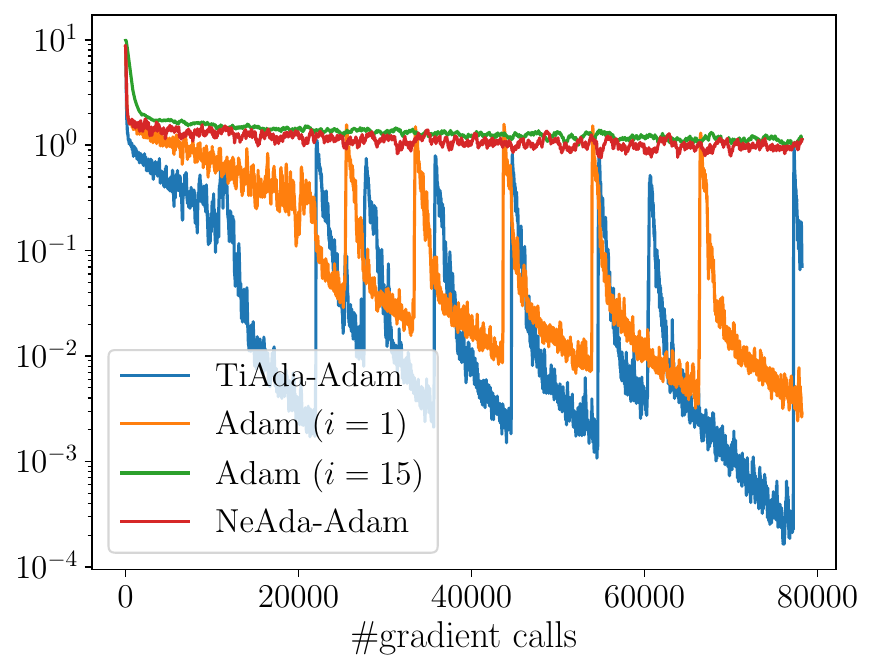}
      \caption{\tiny $\eta^x=0.0001, \eta^y=0.005$}
    \end{subfigure}
    \begin{subfigure}[b]{0.24\textwidth}
      \centering
      \includegraphics[width=\textwidth]{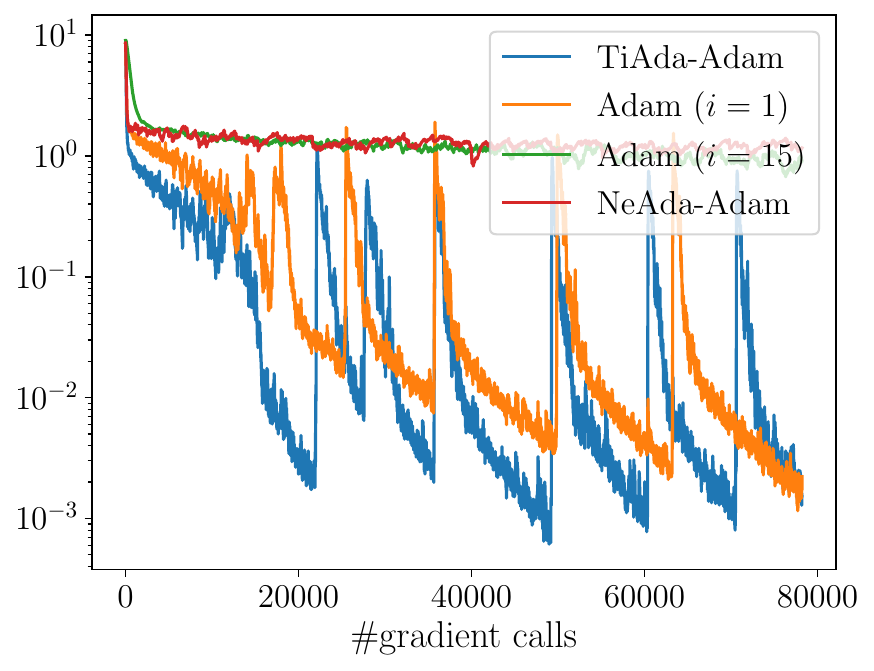}
      \caption{\tiny $\eta^x=0.0001, \eta^y=0.001$}
    \end{subfigure}
    \caption{Gradient norms in $x$ of Adam-like algorithms on distributional
      robustness optimization~\eqref{eq:dist_robust}.
    We use $i$ in the legend to indicate the number of inner loops.}
    \label{fig:adv_robust_all_adam}
\end{figure}

We use a grid of stepsize combinations to evaluate TiAda and compare it with NeAda and GDA with corresponding adaptive stepsizes. For AdaGrad-like algorithms, we use $\{0.1, 0.05, 0.01, 0.0005\}$ for both $\eta^x$ and $\eta^y$, and the results are reported in \Cref{fig:adv_robust_all_adagrad}.
For Adam-like algorithms, we use $\{0.001, 0.0005, 0.0001\}$
for $\eta^x$ and $\{0.1, 0.05, 0.005, 0.001\}$ for $\eta^y$,
and the results are shown in \Cref{fig:adv_robust_all_adam}.
We note that since Adam uses the reciprocal of the moving average of gradient norms, it is extremely unstable when the gradients are small. Therefore, Adam-like algorithms often experience instability when they are near stationary points.

\section{Helper Lemmas}

\begin{lemma}[Lemma~A.2 in \citet{yang2022nest}]
  \label{lemma:sum}
  Let $x_1, ..., x_T$ be a sequence of non-negative real numbers, $x_1 > 0$ and
  $0 < \alpha < 1$. Then we have 
  \[
    \left(\sum_{t=1}^T x_t\right)^{1-\alpha} 
      \leq \sum_{t=1}^T \frac{x_t}{\left(\sum_{k=1}^t x_k\right)^{\alpha}}
      \leq \frac{1}{1-\alpha} \left(\sum_{t=1}^T x_t\right)^{1-\alpha}.
  \]
  When $\alpha = 1$, we have
  \[
      \sum_{t=1}^T \frac{x_t}{\left(\sum_{k=1}^t x_k\right)^{\alpha}}
      \leq 1 + \log\left(\frac{\sum_{t=1}^t x_t}{x_1}\right).
  \]
\end{lemma}

\begin{lemma}[smoothness of $\Phi(\cdot)$ and Lipschitzness of $y^*(\cdot)$. Lemma~4.3 in \citet{lin2020gradient}]
  \label{lemma:lip_y_star}
  Under \Cref{assume:strong-convex,assume:smoothness}, we have
  $\Phi(\cdot)$ is $(l+\rk l)$-smooth with $\nabla \Phi(x) = \nabla_x f(x, y^*(x))$,
  and $y^*(\cdot)$ is $\rk$-Lipschitz.
\end{lemma}

\begin{lemma}[smoothness of $y^*(\cdot)$. Lemma~2 in \citet{chen2021closing}]
  \label{lemma:smooth_y_star}
  Under \Cref{assume:strong-convex,assume:smoothness,assume:lipschitz},
  we have that with $\widehat{L}=\frac{L + L\rk}{\mu} + \frac{l(L+L\rk)}{\mu^2}$,
  \[
    \norm*{\nabla y^*(x_1) - \nabla y^*(x_2)} \leq \widehat{L} \norm*{x_1 - x_2}.
  \]
\end{lemma}

\section{Proofs}

For notational convenience in the proofs, we denote the stochastic gradient as
$\nabla_x \widetilde{f}(x_t, y_t)$
and $\nabla_y \widetilde{f}(x_t, y_t)$. Also denote 
$y^*_t = y^*(x_t)$, 
$\eta_t = \frac{\eta^x}{\max\left\{v^x_{t+1}, v^y_{t+1}\right\}^{\alpha}}$,
$\gamma_t = \frac{\eta^y}{\left(v^y_{t+1}\right)^{\beta}}$, $\Phi^* = \min_{x \in \bR^{d_1}} \Phi(x)$,
and $\Delta \Phi = \Phi_{\max} - \Phi^*$.
We use $\mathbf{1}$ as the indicator function.

\subsection{Proof of Theorem~\ref{theorem:tiada_determ}}

We present a formal version of \Cref{theorem:tiada_determ}.
\begin{theorem}[deterministic setting] \label{theorem:tiada_determ_formal}
  Under \Cref{assume:strong-convex,assume:smoothness,assume:interior_optimal},
  \Cref{alg:tiada} with deterministic gradient
  oracles satisfies that for any $0 < \beta < \alpha < 1$, after $T$ iterations,
  \begin{align*}
    \sum_{t=0}^{T-1}\norm*{\nabla_x f(x_t, y_t)}^2 \leq
    \max\left\{5 C_1, 2 C_2\right\},
  \end{align*}
  where 
  \begin{align*}
    C_1 &= v_0^x + \left(\frac{2\Delta \Phi}{\eta^x}\right)^{\frac{1}{1-\alpha}} + \left( \frac{4\rk l e^{(1-\alpha)(1-\log v_0^x)/2}}{e(1-\alpha) \left(v_{0}^x\right)^{2\alpha-1} } \right)^{\frac{2}{1-\alpha}} \mathbf{1}_{2\alpha \geq 1} 
    +  \left(\frac{2\rk l}{1-2\alpha}\right)^{\frac{1}{\alpha}} \mathbf{1}_{2\alpha < 1}  \\
    &\fakeeq + \left(\frac{c_1 c_5}{\eta^x}\right)^{\frac{1}{1-\alpha}} + \left(\frac{2 c_1 c_4 \eta^x e^{(1-\alpha)(1-\log v_0^x)/2}}{e(1-\alpha) \left(v_{0}^x\right)^{2\alpha - \beta - 1}}\right)^{\frac{2}{1-\alpha}} \mathbf{1}_{2\alpha -\beta \geq 1}
    + \left(\frac{c_1 c_4 \eta^x}{1-2\alpha + \beta}\right)^{\frac{1}{\alpha - \beta}} \mathbf{1}_{2\alpha -\beta < 1} \\
    C_2 &= v_0^x + \Bigg[  \left( \frac{2\Delta \Phi + c_1 c_5 }{\eta^x \left(v_{0}^x\right)^{1-2\alpha+\beta}}
    + \frac{c_1 c_4 \eta^x }{1-2\alpha + \beta}
    +  \frac{2\rk l e^{(1-2\alpha+\beta)(1-\log v_0^x)}}{ e (1-2\alpha + \beta) \left(v_{0}^x\right)^{2\alpha-1} } \mathbf{1}_{2\alpha \geq 1} 
    +  \frac{2\rk l}{(1-2\alpha)\left(v_{0}^x\right)^{\beta}} \mathbf{1}_{2\alpha < 1} 
     \right) \\
    &\fakeeq \left( \frac{c_5}{\left(v_{0}^x\right)^{1-2\alpha+\beta}} + \frac{c_4 \left(\eta^x\right)^2 }{1-2\alpha + \beta} \right)^{\frac{\alpha}{1-\beta}} \Bigg]^{\frac{1}{1 - (1-2\alpha+\beta) \left(1 +\frac{\alpha}{1-\beta}\right)}}
    \mathbf{1}_{2 \alpha - \beta < 1} \\
    &\fakeeq+ \Bigg[ \left( \frac{2\Delta \Phi + c_1 c_5}{\eta^x \left(v_{0}^x\right)^{1/4}} 
  + \frac{8\rk l e^{(1-\log v_0^x)/4}}{e \left(v_{0}^x\right)^{2\alpha-1} } 
   +  \frac{4 c_1 c_4 \eta^x e^{(1-\log v_0^x)/4}}{e \left(v_{0}^x\right)^{2\alpha - \beta - 1} } \right) \\
    &\fakeeq \left( \frac{c_5}{\left(v_{0}^x\right)^{\frac{(1-\beta)}{4\alpha}}} 
    +  \frac{4 c_4 \alpha \left(\eta^x\right)^2 e^{(1-\beta)(1-\log v_0^x)/(4\alpha)}}{e (1-\beta) \left(v_{0}^x\right)^{2\alpha - \beta - 1} }
\right)^{\frac{\alpha}{1-\beta}} \Bigg]^2 \mathbf{1}_{2\alpha \geq 1}, \\
\text{ with }
  \Delta \Phi &= \Phi(x_0) - \Phi^*, \quad
  c_1 =  \frac{\eta^x \rk^2 }{\eta^y \left(v^y_{0}\right)^{\alpha - \beta} }, \quad
  c_2 = \max\left\{\frac{4\eta^y \mu l}{\mu + l}, \eta^y(\mu + l)\right\}, \\
  c_3 &=  4 (\mu + l) \left(\frac{1}{\mu^2} + \frac{\eta^y}{\left(v^y_{0}\right)^{\beta}} \right) c_2^{1/\beta}, \quad
  c_4 =  (\mu + l)\left( \frac{2\rk^2}{\left(v^y_{0}\right)^{\alpha}} + \frac{(\mu + l) \rk^2 }{\eta^y \mu l} \right), \quad
  c_5 =  c_3 + \frac{\eta^y v^y_0}{\left(v_0^y\right)^{\beta}} + \frac{\eta^y c_2^{\frac{1-\beta}{\beta}}}{1-\beta}.
  \end{align*}
  In addition, denoting the above upper bound for $\sum_{t=0}^{T-1}\norm*{\nabla_x f(x_t, y_t)}^2$
  as $C_3$, we have
  \begin{align*}
  \sum_{t=0}^{T-1}\norm*{\nabla_y f(x_t, y_t)}^2
  \leq \left( c_5 + c_4 \left(\eta^x\right)^2 \left( \frac{1 + \log C_3 - \log v^x_0}{\left(v_{0}^x\right)^{2\alpha - \beta - 1}} \mathbf{1}_{2\alpha -\beta \geq 1}
      + \frac{C_3^{1-2\alpha + \beta} }{1-2\alpha + \beta} \mathbf{1}_{2\alpha -\beta < 1} \right) \right)^{\frac{1}{1-\beta}}.
  \end{align*}
\end{theorem}

\begin{proof}
  Let us start from the smoothness of the primal function 
  $\Phi(\cdot)$. By \Cref{lemma:lip_y_star},
  \begin{align*}
    &\fakeeq \RP(x_{t+1})  \\
    &\leq \Phi(x_t) - \eta_t \inp*{\Phi(x_{t+1})}{\nabla_x f(x_t, y_t)}
      + kl\eta_t^2 \norm*{\nabla_x f(x_t, y_t)}^2 \\
    &= \Phi(x_t) - \eta_t \norm*{\nabla_x f(x_t, y_t)}^2 + \eta_t 
    \inp*{\nabla_x f(x_t, y_t) - \nabla \Phi(x_t)}{\nabla_x f(x_t, y_t)}
      + kl\eta_t^2 \norm*{\nabla_x f(x_t, y_t)}^2 \\
    &\leq \Phi(x_t) - \eta_t \norm*{\nabla_x f(x_t, y_t)}^2 + \frac{\eta_t}{2} 
    \norm*{\nabla_x f(x_t, y_t)}^2 + \frac{\eta_t}{2}\norm*{\nabla_x f(x_t, y_t) - \nabla \Phi(x_t)}^2
      + kl\eta_t^2 \norm*{\nabla_x f(x_t, y_t)}^2 \\
    &= \Phi(x_t) - \frac{\eta_t}{2} \norm*{\nabla_x f(x_t, y_t)}^2
      + kl \eta_t^2 \norm*{\nabla_x f(x_t, y_t)}^2 
      + \frac{\eta_t}{2} \norm*{\nabla_x f(x_t, y_t) - \nabla \Phi(x_{t})}^2 \\
    &= \Phi(x_t) - \frac{\eta_t}{2} \norm*{\nabla_x f(x_t, y_t)}^2
      + kl \eta_t^2 \norm*{\nabla_x f(x_t, y_t)}^2 
      + \frac{\eta^x}{2 \max\left\{v^x_{t+1}, v^y_{t+1}\right\}^{\alpha}} \norm*{\nabla_x f(x_t, y_t) - \nabla \Phi(x_{t})}^2 \\
    &\leq \Phi(x_t) - \frac{\eta_t}{2} \norm*{\nabla_x f(x_t, y_t)}^2
      + kl \eta_t^2 \norm*{\nabla_x f(x_t, y_t)}^2 
      + \frac{\eta^x}{2 \left(v^y_{0}\right)^{\alpha - \beta} \left(v^y_{t+1}\right)^{\beta}} \norm*{\nabla_x f(x_t, y_t) - \nabla \Phi(x_{t})}^2 \\
    &\leq \Phi(x_t) - \frac{\eta_t}{2} \norm*{\nabla_x f(x_t, y_t)}^2
      + kl \eta_t^2 \norm*{\nabla_x f(x_t, y_t)}^2 
      + \frac{\eta^x \rk^2}{2 \left(v^y_{0}\right)^{\alpha - \beta} \left(v^y_{t+1}\right)^{\beta}} \norm*{\nabla_y f(x_t, y_t)}^2 \\
    &\leq \Phi(x_t) - \frac{\eta_t}{2} \norm*{\nabla_x f(x_t, y_t)}^2
      + kl \eta_t^2 \norm*{\nabla_x f(x_t, y_t)}^2 
      + \frac{\eta^x \rk^2 }{2 \eta^y \left(v^y_{0}\right)^{\alpha - \beta} } \cdot \rg_t \norm*{\nabla_y f(x_t, y_t)}^2,
  \end{align*}
  where in the second to last inequality, we used the strong-concavity of $f(x, \cdot)$:
  \[
    \norm*{\nabla_x f(x_t, y_t) - \nabla \Phi(x_{t})} \leq l \norm*{y_t - y^*_t} \leq \rk \norm*{\nabla_y f(x_t, y_t)}.
  \]
  Telescoping and rearranging the terms, we have
  \begin{align*}
    &\fakeeq 
    \sum_{t=0}^{T-1} \eta_t \norm*{\nabla_x f(x_t, y_t)}^2 \\
    &\leq 2 \underbrace{ \left(\Phi(x_0) - \Phi^*\right)}_{\Delta \Phi}
      + 2 \rk l \sum_{t=0}^{T-1} \eta_t^2 \norm*{\nabla_x f(x_t, y_t)}^2 + 
      \underbrace{ \frac{\eta^x \rk^2 }{\eta^y \left(v^y_{0}\right)^{\alpha - \beta} } }_{c_1} \sum_{t=0}^{T-1} \rg_t \norm*{\nabla_y f(x_t, y_t)}^2 \\
    &= 2\Delta \Phi + \sum_{t=0}^{T-1} \frac{2 \rk l \eta^x}{\max\left\{v^x_{t+1}, v^y_{t+1}\right\}^{2\alpha}} \norm*{\nabla_x f(x_t, y_t)}^2 
    + c_1 \sum_{t=0}^{T-1} \rg_t \norm*{\nabla_y f(x_t, y_t)}^2 \\
    &\leq 2\Delta \Phi + \sum_{t=0}^{T-1} \frac{2 \rk l \eta^x}{\left(v^x_{t+1}\right)^{2\alpha}} \norm*{\nabla_x f(x_t, y_t)}^2 
    + c_1 \sum_{t=0}^{T-1} \rg_t \norm*{\nabla_y f(x_t, y_t)}^2 \\
    &\leq 2\Delta \Phi + 2\rk l \eta^x \left( \frac{1 + \log v_T^x - \log v_0^x}{\left(v_{0}^x\right)^{2\alpha-1}} \cdot \mathbf{1}_{2\alpha \geq 1} 
    +  \frac{\left(v_{T}^x\right)^{1-2\alpha}}{1-2\alpha} \cdot \mathbf{1}_{2\alpha < 1} \right)
    + c_1 \sum_{t=0}^{T-1} \rg_t \norm*{\nabla_y f(x_t, y_t)}^2. \numberthis \label{eq:mm1}
  \end{align*}

  We proceed to bound $\sum_{t=0}^{T-1} \rg_t \norm*{\nabla_y f(x_t, y_t)}^2$. 
  Let $t_0$ be the first iteration such that 
  $\left(v^y_{t_0+1}\right)^{\beta} > c_2 :=  \max\left\{\frac{4\eta^y \mu l}{\mu + l}, \eta^y(\mu + l)\right\}$.
  If $t_0$ does not exists, then the sum can be trivially bounded with $c_2$, therefore
  we consider the case when $t_0$ exists.
  We have $v^y_{t_0} \leq c_2^{1/\beta}$, and
  for $t \geq t_0$,
  \begin{align*}
    &\fakeeq \norm*{y_{t+1} - y_{t+1}^*}^2 \\
    &\leq (1+\rl_t) \norm*{y_{t+1} - y_t^*}^2
      + \left(1 + \frac{1}{\rl_t}\right) \norm*{y_{t+1}^* - y_t^*}^2 \\
    &\leq \underbrace{(1+\rl_t) \left(\norm*{y_t - y_t^*}^2 + \frac{\left(\eta^y\right)^2}{\left(v^y_{t+1}\right)^{2\rb}}
      \norm*{\nabla_y f(x_t, y_t)}^2 
  + \frac{2\eta^y}{\left(v^y_{t+1}\right)^{\rb}} \inp*{y_t - y_t^*}{\nabla_y f(x_t, y_t)} \right)}_{\text{\labelterm{term:a}}} \\
    &\fakeeq + \left(1 + \frac{1}{\rl_t}\right) \norm*{y_{t+1}^* - y_t^*}^2,
  \end{align*}
  where $\lambda_t > 0$ will be determined later.
  For $l$-smooth and $\mu$-strongly convex function $g(x)$, according to Theorem~2.1.12 in
  \citet{nesterov2003introductory}, we have
  \[
    \inp*{\nabla g(x) - \nabla g(y)}{x - y} \geq \frac{\mu l}{\mu + l} \norm*{x - y}^2
    + \frac{1}{\mu + l} \norm*{\nabla g(x) - \nabla g(y)}^2.
  \]
  Therefore,
  \begin{align*}
    &\fakeeq \text{\Refterm{term:a}}  \\
    &\leq (1+\rl_t) \left(
      \left(1 - \frac{2\eta^y \mu l}{(\mu + l)\left(v^y_{t+1}\right)^{\beta} }\right) \norm*{y_t - y_t^*}^2
      + \left(\frac{\left(\eta^y\right)^2}{\left(v^y_{t+1}\right)^{2\rb}} - 
      \frac{2\eta^y}{(\mu + l)\left(v^y_{t+1}\right)^{\beta}}\right) \norm*{\nabla_y f(x_t, y_t)}^2
    \right).
  \end{align*}
  Let $\lambda_t = \frac{\eta^y \mu l}{(\mu + l) \left(v^y_{t+1}\right)^{\beta} - 2 \eta^y \mu l}$.
  Note that $\lambda_t > 0$ after $t_0$. Then
  \begin{align*}
    &\fakeeq \text{\Refterm{term:a}}  \\
    &\leq \left(1 - \frac{\eta^y \mu l}{(\mu + l) \left(v^y_{t+1}\right)^{\beta}}\right)
      \norm*{y_t - y_t^*}^2 
      + (1+\lambda_t)\left(\frac{\left(\eta^y\right)^2}{\left(v^y_{t+1}\right)^{2\rb}} - 
      \frac{2\eta^y}{(\mu + l)\left(v^y_{t+1}\right)^{\beta}}\right) \norm*{\nabla_y f(x_t, y_t)}^2 \\
    &\leq \norm*{y_t - y_t^*}^2 
    + \underbrace{(1+\lambda_t)\left(\frac{\left(\eta^y\right)^2}{\left(v^y_{t+1}\right)^{2\rb}} - 
    \frac{2\eta^y}{(\mu + l)\left(v^y_{t+1}\right)^{\beta}}\right)}_{\text{\labelterm{term:b}}} \norm*{\nabla_y f(x_t, y_t)}^2.
  \end{align*}
  As $1 + \lambda_t \geq 1$ and
  $\left(v^y_{t+1}\right)^{\beta} \geq \eta^y(\mu + l)$,
  we have $\text{\refterm{term:b}} \leq - \frac{\eta^y}{(\mu + l)\left(v^y_{t+1}\right)^{\beta}}$.
  Putting them back, we can get
  \begin{align*}
    &\fakeeq \norm*{y_{t+1} - y_{t+1}^*}^2 \\
    &\leq \norm*{y_t - y_t^*}^2 - \frac{\eta^y }{(\mu + l)\left(v^y_{t+1}\right)^{\beta}}
    \norm*{\nabla_y f(x_t, y_t)}^2 + \left(1 + \frac{1}{\rl_t}\right) \norm*{y_{t+1}^* - y_t^*}^2 \\
    &\leq \norm*{y_t - y_t^*}^2 - \frac{\eta^y }{(\mu + l)\left(v^y_{t+1}\right)^{\beta}}
    \norm*{\nabla_y f(x_t, y_t)}^2 + \frac{(\mu + l) \left(v^y_{t+1}\right)^{\beta}}{\eta^y \mu l} \norm*{y_{t+1}^* - y_t^*}^2 \\
    &\leq \norm*{y_t - y_t^*}^2 - \frac{\eta^y }{(\mu + l)\left(v^y_{t+1}\right)^{\beta}}
    \norm*{\nabla_y f(x_t, y_t)}^2 + \frac{(\mu + l) \rk^2 \left(v^y_{t+1}\right)^{\beta}}{\eta^y \mu l} \norm*{x_{x+1} - x_t}^2 \\
    &= \norm*{y_t - y_t^*}^2 - \frac{\eta^y }{(\mu + l)\left(v^y_{t+1}\right)^{\beta}}
    \norm*{\nabla_y f(x_t, y_t)}^2 + \frac{(\mu + l) \rk^2 \left(v^y_{t+1}\right)^{\beta}\eta_t^2}{\eta^y \mu l} \norm*{\nabla_x f(x_t, y_t)}^2.
  \end{align*}
  Then, by telescoping, we have
  \begin{align*}
    \sum_{t=t_0}^{T-1} \frac{\eta^y }{(\mu + l)\left(v^y_{t+1}\right)^{\beta}} \norm*{\nabla_y f(x_t, y_t)}^2
    &\leq \norm*{y_{t_0} - y_{t_0}^*}^2
    + \sum_{t=t_0}^{T-1} \frac{(\mu + l) \rk^2 \left(v^y_{t+1}\right)^{\beta}\eta_t^2}{\eta^y \mu l} \norm*{\nabla_x f(x_t, y_t)}^2. \numberthis \label{eq:pe3}
  \end{align*}
  For the first term in the RHS, using Young's inequality with $\tau$ to be
  determined later, we have
  \begin{align*}
    \norm*{y_{t_0} - y_{t_0}^*}^2
    &\leq 2 \norm*{y_{t_0} - y_{t_0 - 1}^*}^2 + 2 \norm*{y_{t_0}^* - y_{t_0 - 1}^*}^2 \\
    &= 2 \norm*{\cP_{\cY}\left(y_{t_0 - 1} + \rg_{t_0 - 1} \nabla_y f(x_{t_0 - 1}, y_{t_0 - 1})\right) - y_{t_0 - 1}^*}^2 + 2 \norm*{y_{t_0}^* - y_{t_0 - 1}^*}^2 \\
    &\leq 2 \norm*{y_{t_0 - 1} + \rg_{t_0 - 1} \nabla_y f(x_{t_0 - 1}, y_{t_0 - 1}) - y_{t_0 - 1}^*}^2 + 2 \norm*{y_{t_0}^* - y_{t_0 - 1}^*}^2 \\
    &\leq 4 \left(\norm*{y_{t_0 - 1} - y_{t_0 - 1}^*}^2 + \rg_{t_0 - 1}^2 \norm*{\nabla_y f(x_{t_0 - 1}, y_{t_0 - 1})}^2 \right) + 2 \norm*{y_{t_0}^* - y_{t_0 - 1}^*}^2 \\
    &\leq 4 \left(\frac{1}{\mu^2} \norm*{\nabla_y f(x_{t_0 - 1}, y_{t_0 - 1})}^2 + \rg_{t_0 - 1}^2 \norm*{\nabla_y f(x_{t_0 - 1}, y_{t_0 - 1})}^2 \right) + 2 \norm*{y_{t_0}^* - y_{t_0 - 1}^*}^2 \\
    &= 4 \left(\frac{1}{\mu^2} + \rg_{t_0 - 1}^2 \right) \norm*{\nabla_y f(x_{t_0 - 1}, y_{t_0 - 1})}^2  + 2 \norm*{y_{t_0}^* - y_{t_0 - 1}^*}^2 \\
    &\leq 4 \left(\frac{1}{\mu^2} + \rg_{0}^2 \right) v^y_{t_0} + 2 \norm*{y_{t_0}^* - y_{t_0 - 1}^*}^2 \\
    &\leq 4 \left(\frac{1}{\mu^2} + \frac{\eta^y}{\left(v^y_{t_0}\right)^{\beta}} \right) c_2^{1/\beta} + 2 \norm*{y_{t_0}^* - y_{t_0 - 1}^*}^2 \\
    &\leq 4 \left(\frac{1}{\mu^2} + \frac{\eta^y}{\left(v^y_{t_0}\right)^{\beta}} \right) c_2^{1/\beta} + 2 \rk^2 \norm*{x_{t_0} - x_{t_0 - 1}}^2 \\
    &\leq 4 \left(\frac{1}{\mu^2} + \frac{\eta^y}{\left(v^y_{t_0}\right)^{\beta}} \right) c_2^{1/\beta} + 2 \rk^2 \eta_{t_0 - 1}^2 \norm*{\nabla_x f(x_{t_0 - 1}, y_{t_0 - 1})}^2 \\
    &\leq 4 \left(\frac{1}{\mu^2} + \frac{\eta^y}{\left(v^y_{t_0}\right)^{\beta}} \right) c_2^{1/\beta} + \frac{2 \rk^2 \left(v^y_{t+1}\right)^{\beta}}{\left(v^y_{0}\right)^{\beta}} \eta_{t_0 - 1}^2 \norm*{\nabla_x f(x_{t_0 - 1}, y_{t_0 - 1})}^2 \\
    &\leq 4 \left(\frac{1}{\mu^2} + \frac{\eta^y}{\left(v^y_{0}\right)^{\beta}} \right) c_2^{1/\beta} + \frac{2 \rk^2 \left(v^y_{t+1}\right)^{\beta}}{\left(v^y_{0}\right)^{\beta}} \eta_{t_0 - 1}^2 \norm*{\nabla_x f(x_{t_0 - 1}, y_{t_0 - 1})}^2.
  \end{align*}
  Combined with \Cref{eq:pe3}, we have
  \begin{align*}
    &\fakeeq \sum_{t=t_0}^{T-1} \frac{\eta^y }{\left(v^y_{t+1}\right)^{\beta}} \norm*{\nabla_y f(x_t, y_t)}^2 \\
    &\leq \underbrace{ 4 (\mu + l) \left(\frac{1}{\mu^2} + \frac{\eta^y}{\left(v^y_{0}\right)^{\beta}} \right) c_2^{1/\beta} }_{c_3}
    + \underbrace{ (\mu + l)\left( \frac{2\rk^2}{\left(v^y_{0}\right)^{\alpha}} + \frac{(\mu + l) \rk^2 }{\eta^y \mu l} \right) }_{c_4} \sum_{t=t_0 - 1}^{T-1}  \left(v^y_{t+1}\right)^{\beta}\eta_t^2 \norm*{\nabla_x f(x_t, y_t)}^2.
  \end{align*}
  By adding terms from $0$ to $t_0 - 1$ and $\frac{\eta^y v^y_0}{\left(v_0^y\right)^{\beta}}$ from both sides,
  and use \Cref{lemma:sum}, we have
  \begin{align*}
    &\fakeeq \frac{\eta^y v^y_0}{\left(v_0^y\right)^{\beta}} + \sum_{t=0}^{T-1} \frac{\eta^y }{\left(v^y_{t+1}\right)^{\beta}} \norm*{\nabla_y f(x_t, y_t)}^2 \\
    &\leq c_3 + \frac{\eta^y v^y_0}{\left(v_0^y\right)^{\beta}} + c_4 \sum_{t=0}^{T-1} \left(v^y_{t+1}\right)^{\beta}\eta_t^2 \norm*{\nabla_x f(x_t, y_t)}^2
    + \sum_{t=0}^{t_0-1} \frac{\eta^y }{\left(v^y_{t+1}\right)^{\beta}} \norm*{\nabla_y f(x_t, y_t)}^2 \\
    &\leq c_3 + \frac{\eta^y v^y_0}{\left(v_0^y\right)^{\beta}} + c_4 \sum_{t=0}^{T-1} \left(v^y_{t+1}\right)^{\beta}\eta_t^2 \norm*{\nabla_x f(x_t, y_t)}^2
    + \frac{\eta^y}{1-\beta} v_{t_0}^{1-\beta} \\
    &\leq c_3 + \frac{\eta^y v^y_0}{\left(v_0^y\right)^{\beta}} + c_4 \sum_{t=0}^{T-1} \left(v^y_{t+1}\right)^{\beta}\eta_t^2 \norm*{\nabla_x f(x_t, y_t)}^2
    + \frac{\eta^y c_2^{\frac{1-\beta}{\beta}}}{1-\beta} \\
    &=  c_3 + \frac{\eta^y v^y_0}{\left(v_0^y\right)^{\beta}} + \frac{\eta^y c_2^{\frac{1-\beta}{\beta}}}{1-\beta}  + c_4 \left(\eta^x\right)^2
    \sum_{t=0}^{T-1} \frac{\left(v^y_{t+1}\right)^{\beta}}{\max\left\{v^x_{t+1}, v^y_{t+1}\right\}^{2\alpha}} \norm*{\nabla_x f(x_t, y_t)}^2 \\
    &= \underbrace{ c_3 + \frac{\eta^y v^y_0}{\left(v_0^y\right)^{\beta}} + \frac{\eta^y c_2^{\frac{1-\beta}{\beta}}}{1-\beta} }_{c_5} + c_4 \left(\eta^x\right)^2
    \sum_{t=0}^{T-1} \frac{1}{\left(v^x_{t+1}\right)^{2\alpha - \beta}} \norm*{\nabla_x f(x_t, y_t)}^2 \\
      &\leq c_5 + c_4 \left(\eta^x\right)^2 \left( \frac{1 + \log v_T^x - \log v_0^x}{\left(v_{0}^x\right)^{2\alpha - \beta - 1}} \cdot \mathbf{1}_{2\alpha -\beta \geq 1}
      + \frac{\left(v_{T}^x\right)^{1-2\alpha + \beta} }{1-2\alpha + \beta} \cdot \mathbf{1}_{2\alpha -\beta < 1} \right).
  \end{align*}
  The LHS can be bounded by $\left(v^y_{T}\right)^{1-\beta}$ by \Cref{lemma:sum}.
  Then we get two useful inequalities from above:
  \begin{align*}
    \begin{cases}
      \sum_{t=0}^{T-1} \rg_t \norm*{\nabla_y f(x_t, y_t)}^2 
        \leq c_5 + c_4 \left(\eta^x\right)^2 \left( \frac{1 + \log v_T^x - \log v_0^x}{\left(v_{0}^x\right)^{2\alpha - \beta - 1}} \cdot \mathbf{1}_{2\alpha -\beta \geq 1}
        + \frac{\left(v_{T}^x\right)^{1-2\alpha + \beta} }{1-2\alpha + \beta} \cdot \mathbf{1}_{2\alpha -\beta < 1} \right) \\
      v^y_{T} \leq \left( c_5 + c_4 \left(\eta^x\right)^2 \left( \frac{1 + \log v_T^x - \log v_0^x}{\left(v_{0}^x\right)^{2\alpha - \beta - 1}} \cdot \mathbf{1}_{2\alpha -\beta \geq 1}
      + \frac{\left(v_{T}^x\right)^{1-2\alpha + \beta} }{1-2\alpha + \beta} \cdot \mathbf{1}_{2\alpha -\beta < 1} \right) \right)^{\frac{1}{1-\beta}}. \numberthis \label{eq:mm2}
    \end{cases}
  \end{align*}
  Now bring it back to \Cref{eq:mm1}, we get
  \begin{align*}
    &\fakeeq \sum_{t=0}^{T-1} \eta_t \norm*{\nabla_x f(x_t, y_t)}^2 \\
    &\leq 2\Delta \Phi + 2\rk l \eta^x \left( \frac{1 + \log v_T^x - \log v_0^x}{\left(v_{0}^x\right)^{2\alpha-1}} \cdot \mathbf{1}_{2\alpha \geq 1} 
    +  \frac{\left(v_{T}^x\right)^{1-2\alpha}}{1-2\alpha} \cdot \mathbf{1}_{2\alpha < 1} \right) \\
    &\fakeeq + c_1 c_5 + c_1 c_4 \left(\eta^x\right)^2 \left( \frac{1 + \log v_T^x - \log v_0^x}{\left(v_{0}^x\right)^{2\alpha - \beta - 1}} \cdot \mathbf{1}_{2\alpha -\beta \geq 1}
      + \frac{\left(v_{T}^x\right)^{1-2\alpha + \beta} }{1-2\alpha + \beta} \cdot \mathbf{1}_{2\alpha -\beta < 1} \right).
  \end{align*}
  For the LHS, we have
  \begin{align*}
    \sum_{t=0}^{T-1} \eta_t \norm*{\nabla_x f(x_t, y_t)}^2
    &= \sum_{t=0}^{T-1} \frac{\eta^x}{\max\left\{v^x_{t+1}, v^y_{t+1}\right\}^{\alpha}} \norm*{\nabla_x f(x_t, y_t)}^2 \\
    &\geq \frac{\eta^x}{\max\left\{v^x_{T}, v^y_{T}\right\}^{\alpha}} \sum_{t=0}^{T-1}  \norm*{\nabla_x f(x_t, y_t)}^2
  \end{align*}
 From here, by combining two inequalites above and noting that $\sum_{t=0}^{T-1}  \norm*{\nabla_x f(x_t, y_t)}^2 \leq v_T^x$, we can already conclude that $\sum_{t=0}^{T-1}  \norm*{\nabla_x f(x_t, y_t)}^2 = \cO(1)$. Now we will provide an explicit bound. We consider two cases: 
  \paragraph{(1)} If $v^y_{T} \leq v^x_{T}$, then
  \begin{align*}
    &\fakeeq \sum_{t=0}^{T-1} \norm*{\nabla_x f(x_t, y_t)}^2  \\
    &\leq \frac{2\Delta \Phi \left(v_{T}^x\right)^{\alpha}}{\eta^x} + 2\rk l \left( \frac{\left(v_{T}^x\right)^{\alpha} \left(1 + \log v_T^x - \log v_0^x\right)}{\left(v_{0}^x\right)^{2\alpha-1} } \cdot \mathbf{1}_{2\alpha \geq 1} 
    +  \frac{\left(v_{T}^x\right)^{1-\alpha}}{1-2\alpha} \cdot \mathbf{1}_{2\alpha < 1} \right) \\
    &\fakeeq + \frac{c_1 c_5\left(v_{T}^x\right)^{\alpha}}{\eta^x} + c_1 c_4 \eta^x \left( \frac{\left(v_{T}^x\right)^{\alpha} \left(1 + \log v_T^x - \log v_0^x\right)}{\left(v_{0}^x\right)^{2\alpha - \beta - 1}} \cdot \mathbf{1}_{2\alpha -\beta \geq 1}
      + \frac{\left(v_{T}^x\right)^{1-\alpha + \beta} }{1-2\alpha + \beta} \cdot \mathbf{1}_{2\alpha -\beta < 1} \right) \\
    &= \frac{2\Delta \Phi \left(v_{T}^x\right)^{\alpha}}{\eta^x} + 2\rk l \left( \frac{\left(v_{T}^x\right)^{\alpha} \left(v_{T}^x\right)^{\frac{1-\alpha}{2}} \left(v_{T}^x\right)^{\frac{\alpha - 1}{2}} \left(1 + \log v_T^x - \log v_0^x\right) }{\left(v_{0}^x\right)^{2\alpha-1} } \cdot \mathbf{1}_{2\alpha \geq 1} 
    +  \frac{\left(v_{T}^x\right)^{1-\alpha}}{1-2\alpha} \cdot \mathbf{1}_{2\alpha < 1} \right) \\
    &\fakeeq + \frac{c_1 c_5\left(v_{T}^x\right)^{\alpha}}{\eta^x} + c_1 c_4 \eta^x \left( \frac{\left(v_{T}^x\right)^{\alpha} \left(v_{T}^x\right)^{\frac{1-\alpha}{2}} \left(v_{T}^x\right)^{\frac{\alpha - 1}{2}} \left(1 + \log v_T^x - \log v_0^x\right)}{\left(v_{0}^x\right)^{2\alpha - \beta - 1}} \cdot \mathbf{1}_{2\alpha -\beta \geq 1}
      + \frac{\left(v_{T}^x\right)^{1-\alpha + \beta} }{1-2\alpha + \beta} \cdot \mathbf{1}_{2\alpha -\beta < 1} \right) \\
    &\leq \frac{2\Delta \Phi \left(v_{T}^x\right)^{\alpha}}{\eta^x} + 2\rk l \left( \frac{ 2 e^{(1-\alpha)(1-\log v_0^x)/2} \left(v_{T}^x\right)^{\frac{1+\alpha}{2}} }{e (1 - \alpha) \left(v_{0}^x\right)^{2\alpha-1} } \cdot \mathbf{1}_{2\alpha \geq 1} 
    +  \frac{\left(v_{T}^x\right)^{1-\alpha}}{1-2\alpha} \cdot \mathbf{1}_{2\alpha < 1} \right) \\
    &\fakeeq + \frac{c_1 c_5\left(v_{T}^x\right)^{\alpha}}{\eta^x} + c_1 c_4 \eta^x \left( \frac{ 2 e^{(1-\alpha)(1-\log v_0^x)/2} \left(v_{T}^x\right)^{\frac{1+\alpha}{2}} }{e(1- \alpha) \left(v_{0}^x\right)^{2\alpha - \beta - 1}} \cdot \mathbf{1}_{2\alpha -\beta \geq 1}
      + \frac{\left(v_{T}^x\right)^{1-\alpha + \beta} }{1-2\alpha + \beta} \cdot \mathbf{1}_{2\alpha -\beta < 1} \right),
      \numberthis \label{eq:new_1}
  \end{align*}
  where we used $x^{-m} (c+\log x) \leq \frac{e^{cm}}{em}$ for $x>0$, $m > 0$ and $c \in \mathbb{R}$ in the last inequality.
  Also, if 
  $0 < \alpha_i < 1$ and
  $b_i$ are positive constants, and $x \leq \sum_{i=1}^n b_i x^{\alpha_i}$,
  then we get $x \leq n\sum_{i=1}^n b_i^{1/(1-\alpha_i)}$.
  Now consider $v_{T}^x$ as the $x$ in the previous statement, and
  note that the LHS of \Cref{eq:new_1} equals to $v_{T}^x - v_0^x$. Then we can get 
  \begin{align*}
    v_{T}^x
    &\leq 5 v_0^x + 5 \left(\frac{2\Delta \Phi}{\eta^x}\right)^{\frac{1}{1-\alpha}} + 5\left( \frac{4\rk l e^{(1-\alpha)(1-\log v_0^x)/2}}{e(1-\alpha) \left(v_{0}^x\right)^{2\alpha-1} } \right)^{\frac{2}{1-\alpha}} \cdot \mathbf{1}_{2\alpha \geq 1} 
    +  5 \left(\frac{2\rk l}{1-2\alpha}\right)^{\frac{1}{\alpha}} \cdot \mathbf{1}_{2\alpha < 1}  \\
    &\fakeeq + 5 \left(\frac{c_1 c_5}{\eta^x}\right)^{\frac{1}{1-\alpha}} + 5 \left(\frac{2 c_1 c_4 \eta^x e^{(1-\alpha)(1-\log v_0^x)/2}}{e(1-\alpha) \left(v_{0}^x\right)^{2\alpha - \beta - 1}}\right)^{\frac{2}{1-\alpha}} \cdot \mathbf{1}_{2\alpha -\beta \geq 1}
    + 5 \left(\frac{c_1 c_4 \eta^x}{1-2\alpha + \beta}\right)^{\frac{1}{\alpha - \beta}} \cdot \mathbf{1}_{2\alpha -\beta < 1}.
    \numberthis \label{eq:new_2}
  \end{align*}
  Note that the RHS is a constant and also an upper bound for $\sum_{t=0}^{T-1} \norm*{\nabla_x f(x_t, y_t)}^2$.
  \paragraph{(2)} If $v^y_{T} \leq v^x_{T}$, then we can use the upper bound for $v^y_{T}$ from \Cref{eq:mm2}.
  We now discuss two cases:
  \begin{enumerate}[wide, labelindent=0pt]
    \item $2 \alpha < 1 + \beta$. Then we have
      \begin{align*}
  &\fakeeq \sum_{t=0}^{T-1} \norm*{\nabla_x f(x_t, y_t)}^2 \\
  &\leq \left( \frac{2\Delta \Phi + c_1 c_5 }{\eta^x} + 2\rk l \left( \frac{1 + \log v_T^x - \log v_0^x}{\left(v_{0}^x\right)^{2\alpha-1}} \cdot \mathbf{1}_{2\alpha \geq 1} 
    +  \frac{\left(v_{T}^x\right)^{1-2\alpha}}{1-2\alpha} \cdot \mathbf{1}_{2\alpha < 1} \right)
    + \frac{c_1 c_4 \eta^x \left(v_{T}^x\right)^{1-2\alpha + \beta} }{1-2\alpha + \beta} \right) \\
  &\fakeeq \left( c_5 + \frac{c_4 \left(\eta^x\right)^2 \left(v_{T}^x\right)^{1-2\alpha + \beta} }{1-2\alpha + \beta} \right)^{\frac{\alpha}{1-\beta}} \\
  &\leq  \left( \frac{2\Delta \Phi + c_1 c_5 }{\eta^x \left(v_{0}^x\right)^{1-2\alpha+\beta}}
    + 2\rk l \left( \frac{1 + \log v_T^x - \log v_0^x}{\left(v_{0}^x\right)^{2\alpha-1} \left(v_{T}^x\right)^{1-2\alpha+\beta}} \cdot \mathbf{1}_{2\alpha \geq 1} 
    +  \frac{1}{(1-2\alpha)\left(v_{0}^x\right)^{\beta}} \cdot \mathbf{1}_{2\alpha < 1} \right)
    + \frac{c_1 c_4 \eta^x }{1-2\alpha + \beta} \right) \\
  &\fakeeq \left( \frac{c_5}{\left(v_{0}^x\right)^{1-2\alpha+\beta}} + \frac{c_4 \left(\eta^x\right)^2 }{1-2\alpha + \beta} \right)^{\frac{\alpha}{1-\beta}}
  \cdot \left(v_{T}^x\right)^{1-2\alpha+\beta+\frac{(1-2\alpha+\beta)\alpha}{1-\beta}} \\
  &\leq  \left( \frac{2\Delta \Phi + c_1 c_5 }{\eta^x \left(v_{0}^x\right)^{1-2\alpha+\beta}}
    + 2\rk l \left( \frac{e^{(1-2\alpha+\beta)(1-\log v_0^x)}}{ e (1-2\alpha + \beta) \left(v_{0}^x\right)^{2\alpha-1} } \cdot \mathbf{1}_{2\alpha \geq 1} 
    +  \frac{1}{(1-2\alpha)\left(v_{0}^x\right)^{\beta}} \cdot \mathbf{1}_{2\alpha < 1} \right)
    + \frac{c_1 c_4 \eta^x }{1-2\alpha + \beta} \right) \\
  &\fakeeq \left( \frac{c_5}{\left(v_{0}^x\right)^{1-2\alpha+\beta}} + \frac{c_4 \left(\eta^x\right)^2 }{1-2\alpha + \beta} \right)^{\frac{\alpha}{1-\beta}}
  \cdot \left(v_{T}^x\right)^{1-2\alpha+\beta+\frac{(1-2\alpha+\beta)\alpha}{1-\beta}},
      \end{align*}
  Note that since $\alpha > \beta$, we have
  \begin{align*}
    1-2\alpha+\beta+\frac{(1-2\alpha+\beta)\alpha}{1-\beta}
    \leq \frac{(1 - \alpha) \alpha}{1-\beta} + 1-\alpha
    = 1 + \frac{\alpha (\beta - \alpha)}{1-\beta} < 1.
  \end{align*}
  Therefore, with the same reasoning as \Cref{eq:new_2},
      \begin{align*}
  &\fakeeq \sum_{t=0}^{T-1} \norm*{\nabla_x f(x_t, y_t)}^2 \leq v_{T}^x \\
  &\leq 2 \Bigg[  \left( \frac{2\Delta \Phi + c_1 c_5 }{\eta^x \left(v_{0}^x\right)^{1-2\alpha+\beta}}
    + \frac{c_1 c_4 \eta^x }{1-2\alpha + \beta}
    +  \frac{2\rk l e^{(1-2\alpha+\beta)(1-\log v_0^x)}}{ e (1-2\alpha + \beta) \left(v_{0}^x\right)^{2\alpha-1} } \cdot \mathbf{1}_{2\alpha \geq 1} 
    +  \frac{2\rk l}{(1-2\alpha)\left(v_{0}^x\right)^{\beta}} \cdot \mathbf{1}_{2\alpha < 1} 
     \right) \\
  &\fakeeq \left( \frac{c_5}{\left(v_{0}^x\right)^{1-2\alpha+\beta}} + \frac{c_4 \left(\eta^x\right)^2 }{1-2\alpha + \beta} \right)^{\frac{\alpha}{1-\beta}} \Bigg]^{\frac{1}{1 - (1-2\alpha+\beta) \left(1 +\frac{\alpha}{1-\beta}\right)}}
  + 2 v_0^x,
      \end{align*}
  which gives us constant RHS.
\item $2\alpha \geq 1 + \beta$. Then we have
  \begin{align*}
  &\fakeeq \sum_{t=0}^{T-1} \norm*{\nabla_x f(x_t, y_t)}^2 \\
  &\leq \left( \frac{2\Delta \Phi + c_1 c_5}{\eta^x} 
  + \frac{2\rk l \left(1 + \log v_T^x - \log v_0^x\right)}{\left(v_{0}^x\right)^{2\alpha-1}} 
   +  \frac{c_1 c_4 \eta^x \left(1 + \log v_T^x - \log v_0^x\right)}{\left(v_{0}^x\right)^{2\alpha - \beta - 1}} \right) \\
  &\fakeeq \left( c_5 +  \frac{c_4 \left(\eta^x\right)^2 \left(1 + \log v_T^x - \log v_0^x\right)}{\left(v_{0}^x\right)^{2\alpha - \beta - 1}}
  \right)^{\frac{\alpha}{1-\beta}} \\
  &\leq \left( \frac{2\Delta \Phi + c_1 c_5}{\eta^x \left(v_{0}^x\right)^{1/4}} 
  + \frac{2\rk l \left(1 + \log v_T^x - \log v_0^x\right)}{\left(v_{0}^x\right)^{2\alpha-1} \left(v_{T}^x\right)^{1/4}} 
   +  \frac{c_1 c_4 \eta^x \left(1 + \log v_T^x - \log v_0^x\right)}{\left(v_{0}^x\right)^{2\alpha - \beta - 1} \left(v_{T}^x\right)^{1/4}} \right) \\
  &\fakeeq \left( \frac{c_5}{\left(v_{0}^x\right)^{\frac{(1-\beta)}{4\alpha}}} 
    +  \frac{c_4 \left(\eta^x\right)^2 \left(1 + \log v_T^x - \log v_0^x\right)}{\left(v_{0}^x\right)^{2\alpha - \beta - 1} \left(v_{T}^x\right)^{\frac{(1-\beta)}{4\alpha}}}
  \right)^{\frac{\alpha}{1-\beta}}
  \cdot \left(v_{T}^x\right)^{1/2} \\
  &\leq \left( \frac{2\Delta \Phi + c_1 c_5}{\eta^x \left(v_{0}^x\right)^{1/4}} 
    + \frac{8\rk l e^{(1-\log v_0^x)/4}}{e \left(v_{0}^x\right)^{2\alpha-1} } 
  +  \frac{4 c_1 c_4 \eta^x e^{(1-\log v_0^x)/4}}{e \left(v_{0}^x\right)^{2\alpha - \beta - 1} } \right) \\
  &\fakeeq \left( \frac{c_5}{\left(v_{0}^x\right)^{\frac{(1-\beta)}{4\alpha}}} 
    +  \frac{4 c_4 \alpha \left(\eta^x\right)^2 e^{(1-\beta)(1-\log v_0^x)/(4\alpha)}}{e (1-\beta) \left(v_{0}^x\right)^{2\alpha - \beta - 1} }
  \right)^{\frac{\alpha}{1-\beta}}
  \cdot \left(v_{T}^x\right)^{1/2},
  \end{align*}
  which implies
  \begin{align*}
  &\fakeeq \sum_{t=0}^{T-1} \norm*{\nabla_x f(x_t, y_t)}^2 \leq v_{T}^x \\
  &\leq 2\Bigg[ \left( \frac{2\Delta \Phi + c_1 c_5}{\eta^x \left(v_{0}^x\right)^{1/4}} 
  + \frac{8\rk l e^{(1-\log v_0^x)/4}}{e \left(v_{0}^x\right)^{2\alpha-1} } 
   +  \frac{4 c_1 c_4 \eta^x e^{(1-\log v_0^x)/4}}{e \left(v_{0}^x\right)^{2\alpha - \beta - 1} } \right) \\
  &\fakeeq \left( \frac{c_5}{\left(v_{0}^x\right)^{\frac{(1-\beta)}{4\alpha}}} 
    +  \frac{4 c_4 \alpha \left(\eta^x\right)^2 e^{(1-\beta)(1-\log v_0^x)/(4\alpha)}}{e (1-\beta) \left(v_{0}^x\right)^{2\alpha - \beta - 1} }
\right)^{\frac{\alpha}{1-\beta}} \Bigg]^2 + 2 v_{0}^x.
  \end{align*}
  Now we also get only a constant on the RHS.
  \end{enumerate}

  Summarizing all the cases, we finish the proof.

\end{proof}

\subsection{Intermediate Lemmas for Theorem~\ref{theorem:tiada_stoc}}

\begin{lemma} \label{lemma:basic_strongly_convex_expand}
  Under the same setting as \Cref{theorem:tiada_stoc},
  if for $t=t_0$ to $t_1-1$ and any $\lambda_t > 0$, $S_t$,
  \[
  \norm*{y_{t+1} - y^*_{t+1}}^2 
  \leq \left(1 + \lambda_t \right) \norm*{y_{t+1} - y^*_{t}}^2 + S_t,
  \]
  then we have
  \begin{align*}
  \Ep{\sum_{t=t_0}^{t_1-1} \left( f(x_{t}, y^*_{t}) - f(x_{t}, y_t) \right)}
  &\leq \Ep{\sum_{t=t_0 + 1}^{t_1-1} \left( \frac{1 - \rg_t \mu}{2 \rg_t} \norm*{y_t - y^*_{t}}^2 - \frac{1}{2 \rg_t (1 + \lambda_t)} \norm*{y_{t+1} - y^*_{t+1}}^2 \right)} \\
  &\fakeeq + \Ep{\sum_{t=t_0}^{t_1-1} \frac{\rg_t}{2} \norm*{\nabla_y \widetilde{f}(x_{t}, y_t)}^2}
  + \Ep{\sum_{t=t_0}^{t_1-1} \frac{S_t}{2\rg_t (1 + \lambda_t)}}.
  \end{align*}
\end{lemma}

\begin{proof}
Letting $\lambda_t := \frac{\mu \eta^y}{2 \left(v^y_{t+1}\right)^{\beta}}$, we have
\begin{align*}
  &\fakeeq \norm*{y_{t+1} - y^*_{t+1}}^2 \\
  &\leq (1 + \lambda_t) \norm*{y_{t+1} - y^*_{t}}^2 + S_t \\
  &= (1 + \lambda_t) \norm*{\cP_{\cY}\left(y_t + \rg_t \nabla_y \widetilde{f}(x_{t}, y_t)\right) - y^*_{t}}^2 + S_t \\
  &\leq (1 + \lambda_t) \norm*{y_t + \rg_t \nabla_y \widetilde{f}(x_{t}, y_t) - y^*_{t}}^2 + S_t \\
  &= (1 + \lambda_t) \left(\norm*{y_t - y^*_{t}}^2 + \rg_t^2 \norm*{\nabla_y \widetilde{f}(x_{t}, y_t)}^2 + 2\rg_t \inp*{\nabla_y \widetilde{f}(x_{t}, y_t)}{y_t - y^*_{t}}\right) + S_t \\
  &= (1 + \lambda_t) \bigg(\norm*{y_t - y^*_{t}}^2 + \rg_t^2 \norm*{\nabla_y \widetilde{f}(x_{t}, y_t)}^2 + 2\rg_t \inp*{\nabla_y \widetilde{f}(x_{t}, y_t)}{y_t - y^*_{t}} \\
  &\fakeeq  + \rg_t \mu \norm*{y_t - y^*_{t}}^2 - \rg_t \mu \norm*{y_t - y^*_{t}}^2 \bigg) + S_t \\ 
\end{align*}
By multiplying $\frac{1}{\rg_t (1 + \lambda_t)}$ and rearranging the terms, we can get
\begin{align*}
  &\fakeeq 2 \inp*{\nabla_y \widetilde{f}(x_{t}, y_t)}{y^*_{t} - y_t} - \mu \norm*{y_t - y^*_{t}}^2 \\
  &\leq \frac{1 - \rg_t \mu}{\rg_t} \norm*{y_t - y^*_{t}}^2 - \frac{1}{\rg_t (1 + \lambda_t)} \norm*{y_{t+1} - y^*_{t+1}}^2
  + \rg_t \norm*{\nabla_y \widetilde{f}(x_{t}, y_t)}^2
  + \frac{S_t}{\rg_t (1 + \lambda_t)}.
\end{align*}
By telescoping from $t=t_0$ to $t_1-1$, we have
\begin{align*}
  &\fakeeq \sum_{t=t_0}^{t_1-1} \left(\inp*{\nabla_y \widetilde{f}(x_{t}, y_t)}{y^*_{t} - y_t} - \frac{\mu}{2} \norm*{y_t - y^*_{t}}^2 \right) \\
  &\leq \sum_{t=t_0 + 1}^{t_1-1} \left( \frac{1 - \rg_t \mu}{2 \rg_t} \norm*{y_t - y^*_{t}}^2 - \frac{1}{2 \rg_t (1 + \lambda_t)} \norm*{y_{t+1} - y^*_{t+1}}^2 \right)
  + \sum_{t=t_0}^{t_1-1} \frac{\rg_t}{2} \norm*{\nabla_y \widetilde{f}(x_{t}, y_t)}^2 \\
  &\fakeeq + \sum_{t=t_0}^{t_1-1} \frac{S_t}{2\rg_t (1 + \lambda_t)}.
\end{align*}
Now we take the expectation and get
\begin{align*}
  \Ep{\text{LHS}}
  &\geq \Ep{ \sum_{t=t_0}^{t_1-1} \Ep[\xi^y_t]{\left(\inp*{\nabla_y \widetilde{f}(x_{t}, y_t)}{y^*_{t} - y_t} - \frac{\mu}{2} \norm*{y_t - y^*_{t}}^2 \right)} } \\
  &= \Ep{ \sum_{t=t_0}^{t_1-1} \left(\inp*{\nabla_y f(x_{t}, y_t)}{y^*_{t} - y_t} - \frac{\mu}{2} \norm*{y_t - y^*_{t}}^2 \right)} \\
  &\geq \Ep{\sum_{t=t_0}^{t_1-1} \left( f(x_{t}, y^*_{t}) - f(x_{t}, y_t) \right)},
\end{align*}
where we used strong-concavity in the last inequality.

\end{proof}

\begin{lemma} \label{lemma:v_small}
  Under the same setting as \Cref{theorem:tiada_stoc},
  if $v^y_{t+1} \leq C$ for $t=0, ..., t_0 - 1$, then we have
\begin{align*}
&\fakeeq \Ep{\sum_{t=0}^{t_0-1} \left( f(x_{t}, y^*_{t}) - f(x_{t}, y_t) \right)} \\
&\leq \Ep{\sum_{t=0}^{t_0-1} \left( \frac{1 - \rg_t \mu}{2 \rg_t} \norm*{y_t - y^*_{t}}^2 - \frac{1}{\rg_t (2 + \mu \rg_t)} \norm*{y_{t+1} - y^*_{t+1}}^2 \right)}
+ \Ep{\sum_{t=0}^{t_0-1} \frac{\rg_t}{2} \norm*{\nabla_y \widetilde{f}(x_{t}, y_t)}^2} \\
&\fakeeq + \frac{\rk^2 \left(\mu \eta^y C^{\beta} + 2 C^{2\beta}\right) \left(\eta^x\right)^2}{2 \mu \left(\eta^y\right)^2 } 
 \Ep{  \frac{1 + \log v^x_{t_0} - \log v^x_{0}}{\left(v^x_{0}\right)^{2\alpha-1}} \cdot \mathbf{1}_{\alpha \geq 0.5} 
  +  \frac{\left(v_{t_0}^x\right)^{1-2\alpha}}{1 - 2\alpha} \cdot \mathbf{1}_{\alpha < 0.5}   }.
\end{align*}
\end{lemma}
\begin{proof}
By Young's inequality, we have
\begin{align*}
  \norm*{y_{t+1} - y^*_{t+1}}^2
  \leq (1 + \lambda_t) \norm*{y_{t+1} - y^*_{t}}^2 + \left(1 + \frac{1}{\lambda_t}\right) \norm*{y^*_{t+1} - y^*_{t}}^2.
\end{align*}
Then letting $\lambda_t = \frac{\mu \rg_t}{2}$ and by \Cref{lemma:basic_strongly_convex_expand}, we have
\begin{align*}
&\fakeeq \Ep{\sum_{t=0}^{t_0-1} \left( f(x_{t}, y^*_{t}) - f(x_{t}, y_t) \right)} \\
&\leq \Ep{\sum_{t=0}^{t_0-1} \left( \frac{1 - \rg_t \mu}{2 \rg_t} \norm*{y_t - y^*_{t}}^2 - \frac{1}{\rg_t (2 + \mu \rg_t)} \norm*{y_{t+1} - y^*_{t+1}}^2 \right)} \\
&\fakeeq + \Ep{\sum_{t=0}^{t_0-1} \frac{\rg_t}{2} \norm*{\nabla_y \widetilde{f}(x_{t}, y_t)}^2}
+ \Ep{\sum_{t=0}^{t_0-1} \frac{\left(1 + \frac{2}{\mu \rg_t}\right)}{\rg_t (2 + \mu \rg_t)} \norm*{y^*_{t+1} - y^*_{t}}^2 }.
\end{align*}
We now remain to bound the last term:
\begin{align*}
  &\fakeeq \Ep{\sum_{t=0}^{t_0-1} \frac{\left(1 + \frac{2}{\mu \rg_t}\right)}{\rg_t (2 + \mu \rg_t)} \norm*{y^*_{t+1} - y^*_{t}}^2 } \\
  &\leq \Ep{\sum_{t=0}^{t_0-1} \frac{\left(1 + \frac{2}{\mu \rg_t}\right)}{2\rg_t} \norm*{y^*_{t+1} - y^*_{t}}^2 } \\
  &= \Ep{\sum_{t=0}^{t_0-1} \frac{\mu \eta^y \left(v^y_{t+1}\right)^{\beta} + 2 \left(v^y_{t+1}\right)^{2\beta}}{2 \mu \left(\eta^y\right)^2 }
  \norm*{y^*_{t+1} - y^*_{t}}^2 } \\
  &\leq \frac{\mu \eta^y C^{\beta} + 2 C^{2\beta}}{2 \mu \left(\eta^y\right)^2 } \Ep{\sum_{t=0}^{t_0-1} 
  \norm*{y^*_{t+1} - y^*_{t}}^2 }.
\end{align*}
By \Cref{lemma:lip_y_star} we have
\begin{align*}
  \sum_{t=0}^{t_0-1} \norm*{y^*_{t+1} - y^*_{t}}^2
  &\leq \rk^2 \sum_{t=0}^{t_0-1} \norm*{x_{t+1} - x_{t}}^2 \\
  &= \rk^2 \sum_{t=0}^{t_0-1} \eta_{t}^2 \norm*{\nabla_x \widetilde{f}(x_{t}, y_{t})}^2 \\
  &= \rk^2 \left(\eta^x\right)^2 \sum_{t=0}^{t_0-1} \frac{1}{\max\left\{v^x_{t+1}, v^y_{t+1}\right\}^{2 \alpha}} \norm*{\nabla_x \widetilde{f}(x_{t}, y_{t})}^2 \\
  &\leq \rk^2 \left(\eta^x\right)^2 \sum_{t=0}^{t_0-1} \frac{1}{\left(v^x_{t+1}\right)^{2 \alpha}} \norm*{\nabla_x \widetilde{f}(x_{t}, y_{t})}^2 \\
  &\leq \rk^2 \left(\eta^x\right)^2 \left( \frac{v^x_{0}}{\left(v^x_{0}\right)^{2 \alpha}} + \sum_{t=0}^{t_0-1} \frac{1}{\left(v^x_{t+1}\right)^{2 \alpha}} \norm*{\nabla_x \widetilde{f}(x_{t}, y_{t})}^2\right) \\
  &\leq \rk^2 \left(\eta^x\right)^2 
  \left( \frac{1 + \log v^x_{t_0} - \log v^x_{0}}{\left(v^x_{0}\right)^{2\alpha-1}} \cdot \mathbf{1}_{\alpha \geq 0.5} 
  +  \frac{\left(v_{t_0}^x\right)^{1-2\alpha}}{1 - 2\alpha} \cdot \mathbf{1}_{\alpha < 0.5}  \right)
\end{align*}
where we applied \Cref{lemma:sum} in the last inequality.
Bringing back this result, we finish the proof.

\end{proof}

\begin{lemma} \label{lemma:v_large}
  Under the same setting as \Cref{theorem:tiada_stoc},
  if $t_0$ is the first iteration such that $v^y_{t_0+1} > C$, then we have
\begin{align*}
&\fakeeq \Ep{\sum_{t=t_0}^{T-1} \left( f(x_{t}, y^*_{t}) - f(x_{t}, y_t) \right)} \\
&\leq \Ep{\sum_{t=t_0}^{T-1} \left( \frac{1 - \rg_t \mu}{2 \rg_t} \norm*{y_t - y^*_{t}}^2 - \frac{1}{\rg_t (2 + \mu \rg_t)} \norm*{y_{t+1} - y^*_{t+1}}^2 \right)}
 + \Ep{\sum_{t=t_0}^{T-1} \frac{\rg_t}{2} \norm*{\nabla_y \widetilde{f}(x_{t}, y_t)}^2} \\
&\fakeeq+ \left(\rk^2 + \frac{\widehat{L}^2  G^2 \left(\eta^x\right)^2}{\mu \eta^y \left(v^y_{0}\right)^{2\alpha - \beta}}\right) 
    \frac{\left(\eta^x\right)^2 }{2 (1 - \alpha) \eta^y \left(v^y_{0}\right)^{\alpha - \beta} }
    \Ep{ \left(v^x_{T}\right)^{1-\alpha}} \\
&\fakeeq+ \frac{2\rk^2 \left(\eta^x\right)^2}{\mu \left(\eta^y\right)^2 C^{2\alpha - 2\beta}}
   \Ep{ \sum_{t=t_0}^{T-1} \norm*{\nabla_x f(x_t, y_t)}^2}
+ \left(\frac{1}{\mu} + \frac{\eta^y }{\left(v^y_{0}\right)^{\beta}}\right) 
    \frac{4\rk\eta^x G^2}{\eta^y \left(v^y_{0}\right)^{\alpha}} \Ep{ \left(v^y_{T}\right)^{\beta}}.
\end{align*}
\end{lemma}

\begin{proof}
  By the Lipschitzness of $y^*(\cdot)$ as in \Cref{lemma:lip_y_star}, we have
\begin{align*}
  \norm*{y_{t+1} - y^*_{t+1}}^2 
  &= \norm*{y_{t+1} - y^*_t}^2 + \norm*{y^*_{t} - y^*_{t+1}}^2
  + 2 \inp*{y_{t+1} - y^*_t}{y^*_{t} - y^*_{t+1}} \\
  &\leq \norm*{y_{t+1} - y^*_t}^2 + \rk^2 \eta_t^2 \norm*{\nabla_x \widetilde{f}(x_t, y_t)}^2
  + 2 \inp*{y_{t+1} - y^*_t}{y^*_{t} - y^*_{t+1}} \\
  &\leq \norm*{y_{t+1} - y^*_t}^2 + \rk^2 \eta_t^2 \norm*{\nabla_x \widetilde{f}(x_t, y_t)}^2
  \underbrace{ -2 \left(y_{t+1} - y^*_t\right)^\T\nabla y^*(x_t)\left(x_{t+1} - x_t\right) }_{\mlabelterm{term:p1}} \\
  &\fakeeq + \underbrace{ 2 \left(y_{t+1} - y^*_t\right)^\T\left(y^*_t - y^*_{t+1} + \nabla y^*(x_t)\left(x_{t+1} - x_t\right)\right) }_{\mlabelterm{term:p2}}.
\end{align*}

For \Refterm{term:p1}, by the Cauchy-Schwarz and Lipschitzness of $y^*(\cdot)$,
\begin{align*}
   &\fakeeq -2 \left(y_{t+1} - y^*_t\right)^\T\nabla y^*(x_t)\left(x_{t+1} - x_t\right) \\
   &= 2 \eta_t \left(y_{t+1} - y^*_t\right)^\T\nabla y^*(x_t) \nabla_x f(x_t, y_t)
    + 2 \eta_t \left(y_{t+1} - y^*_t\right)^\T\nabla y^*(x_t) \left(\nabla_x \widetilde{f}(x_t, y_t) - \nabla_x f(x_t, y_t)\right) \\
   &\leq 2 \eta_t \norm*{y_{t+1} - y^*_t} \norm*{\nabla y^*(x_t)} \norm*{\nabla_x f(x_t, y_t)}
    + 2 \eta_t \left(y_{t+1} - y^*_t\right)^\T\nabla y^*(x_t) \left(\nabla_x \widetilde{f}(x_t, y_t) - \nabla_x f(x_t, y_t)\right) \\
   &\leq 2 \norm*{y_{t+1} - y^*_t} \rk \eta_t \norm*{\nabla_x f(x_t, y_t)}
    +2 \eta_t \left(y_{t+1} - y^*_t\right)^\T\nabla y^*(x_t) \left(\nabla_x \widetilde{f}(x_t, y_t) - \nabla_x f(x_t, y_t)\right) \\
   &\leq \lambda_t \norm*{y_{t+1} - y^*_t}^2 + \frac{\rk^2 \eta_t^2}{\lambda_t} \norm*{\nabla_x f(x_t, y_t)}^2
    +2 \eta_t \left(y_{t+1} - y^*_t\right)^\T\nabla y^*(x_t) \left(\nabla_x \widetilde{f}(x_t, y_t) - \nabla_x f(x_t, y_t)\right),
\end{align*}
where we used Young's inequality in the last step and $\lambda_t > 0$ will
be determined later.

For \Refterm{term:p2}, according to Cauchy-Schwarz and the smoothness of
$y^*(\cdot)$ as shown in \Cref{lemma:smooth_y_star},
\begin{align*}
   &\fakeeq 2 \left(y_{t+1} - y^*_t\right)^\T\left(y^*_t - y^*_{t+1} + \nabla y^*(x_t)\left(x_{t+1} - x_t\right)\right) \\
   &\leq 2 \norm*{y_{t+1} - y^*_t} \norm*{y^*_t - y^*_{t+1} + \nabla y^*(x_t)\left(x_{t+1} - x_t\right)} \\
   &\leq 2 \norm*{y_{t+1} - y^*_t} \cdot \frac{\widehat{L}}{2} \norm*{x_{t+1} - x_t}^2 \\
   &= \widehat{L} \eta_t^2 \norm*{y_{t+1} - y^*_t} \norm*{\nabla_x \widetilde{f}(x_t, y_t)}^2 \\
   &\leq \widehat{L} \eta_t^2 \norm*{y_{t+1} - y^*_t}  G \cdot \norm*{\nabla_x \widetilde{f}(x_t, y_t)} \\
   &\leq \frac{\tau \widehat{L}  G^2 \eta_t^2}{2} \norm*{y_{t+1} - y^*_t}^2  + \frac{\widehat{L} \eta_t^2}{2 \tau} \norm*{\nabla_x \widetilde{f}(x_t, y_t)}^2,
\end{align*}
where in the last step we used Young's inequality and $\tau > 0$.

Therefore, in total, we have
\begin{align*}
  \norm*{y_{t+1} - y^*_{t+1}}^2 
  &\leq \left(1 + \lambda_t + \frac{\tau \widehat{L}  G^2 \eta_t^2}{2}\right) \norm*{y_{t+1} - y^*_t}^2 
  + \left(\rk^2 + \frac{\widehat{L}}{2 \tau}\right) \eta_t^2 \norm*{\nabla_x \widetilde{f}(x_t, y_t)}^2 \\
  &\fakeeq + \frac{\rk^2 \eta_t^2}{\lambda_t} \norm*{\nabla_x f(x_t, y_t)}^2 +2 \eta_t \left(y_{t+1} - y^*_t\right)^\T\nabla y^*(x_t) \left(\nabla_x \widetilde{f}(x_t, y_t) - \nabla_x f(x_t, y_t)\right).
\end{align*}
Note that we can upper bound $\eta_t$ by
\[
  \eta_t = \frac{\eta^x}{\max\left\{v^x_{t+1}, v^y_{t+1}\right\}^{\alpha}} 
  \leq \frac{\eta^x}{\left(v^y_{t+1}\right)^{\alpha}} 
  \leq \frac{\eta^x}{\left(v^y_{0}\right)^{\alpha}},
\]
and
\[
  \eta_t \leq \frac{\eta^x}{\left(v^y_{t+1}\right)^{\alpha}} 
  = \frac{\eta^x}{\left(v^y_{t+1}\right)^{\alpha - \beta}\left(v^y_{t+1}\right)^{\beta}} 
  \leq \frac{\eta^x}{\left(v^y_{0}\right)^{\alpha - \beta}\left(v^y_{t+1}\right)^{\beta}},
\]
which, plugged into the previous result, implies
\begin{align*}
  \norm*{y_{t+1} - y^*_{t+1}}^2 
  &\leq \left(1 + \lambda_t + \frac{\tau \widehat{L}  G^2 \left(\eta^x\right)^2}{2 \left(v^y_{0}\right)^{2\alpha - \beta} \left(v^y_{t+1}\right)^{\beta} }\right) \norm*{y_{t+1} - y^*_t}^2 
  + \left(\rk^2 + \frac{\widehat{L}}{2 \tau}\right) \eta_t^2 \norm*{\nabla_x \widetilde{f}(x_t, y_t)}^2 \\
  &\fakeeq + \frac{\rk^2 \eta_t^2}{\lambda_t} \norm*{\nabla_x f(x_t, y_t)}^2
   +2 \eta_t \left(y_{t+1} - y^*_t\right)^\T\nabla y^*(x_t) \left(\nabla_x \widetilde{f}(x_t, y_t) - \nabla_x f(x_t, y_t)\right).
\end{align*}
Now we choose $\lambda_t = \frac{\mu \eta^y}{4 \left(v^y_{t+1}\right)^{\beta}}$
and $\tau = \frac{\mu \eta^y \left(v^y_{0}\right)^{2\alpha - \beta}}{2 \widehat{L}  G^2 \left(\eta^x\right)^2}$,
and get
\begin{align*}
  &\fakeeq \norm*{y_{t+1} - y^*_{t+1}}^2 \\
  &\leq \left(1 + \frac{\mu \eta^y}{2 \left(v^y_{t+1}\right)^{\beta}} \right) \norm*{y_{t+1} - y^*_t}^2 
  + \left(\rk^2 + \frac{\widehat{L}^2  G^2 \left(\eta^x\right)^2}{\mu \eta^y \left(v^y_{0}\right)^{2\alpha - \beta}}\right) \eta_t^2 \norm*{\nabla_x \widetilde{f}(x_t, y_t)}^2 \\
  &\fakeeq + \frac{4 \rk^2 \left(v^y_{t+1}\right)^{\beta} \eta_t^2}{\mu \eta^y} \norm*{\nabla_x f(x_t, y_t)}^2
   +2 \eta_t \left(y_{t+1} - y^*_t\right)^\T\nabla y^*(x_t) \left(\nabla_x \widetilde{f}(x_t, y_t) - \nabla_x f(x_t, y_t)\right).
\end{align*}
Then \Cref{lemma:basic_strongly_convex_expand} gives us
\begin{align*}
&\fakeeq \Ep{\sum_{t=t_0}^{T-1} \left( f(x_{t}, y^*_{t}) - f(x_{t}, y_t) \right)} \\
&\leq \Ep{\sum_{t=t_0}^{T-1} \left( \frac{1 - \rg_t \mu}{2 \rg_t} \norm*{y_t - y^*_{t}}^2 - \frac{1}{\rg_t (2 + \mu \rg_t)} \norm*{y_{t+1} - y^*_{t+1}}^2 \right)}
 + \Ep{\sum_{t=t_0}^{T-1} \frac{\rg_t}{2} \norm*{\nabla_y \widetilde{f}(x_{t}, y_t)}^2} \\
&\fakeeq+ \underbrace{ \Ep{\sum_{t=t_0}^{T-1} \frac{1}{\rg_t (2 + \mu \rg_t)} \left(\rk^2 + \frac{\widehat{L}^2  G^2 \left(\eta^x\right)^2}{\mu \eta^y \left(v^y_{0}\right)^{2\alpha - \beta}}\right) \eta_t^2 \norm*{\nabla_x \widetilde{f}(x_t, y_t)}^2 } }_{\mlabelterm{term:lv1}} \\
&\fakeeq+ \underbrace{ \Ep{\sum_{t=t_0}^{T-1} \frac{4 \rk^2 \left(v^y_{t+1}\right)^{\beta} \eta_t^2}{\rg_t (2 + \mu \rg_t) \mu \eta^y} \norm*{\nabla_x f(x_t, y_t)}^2} }_{\mlabelterm{term:lv2}} \\
&\fakeeq+ \underbrace{ \Ep{\sum_{t=t_0}^{T-1} \frac{2\eta_t}{\rg_t (2 + \mu \rg_t)} \left(y_{t+1} - y^*_t\right)^\T\nabla y^*(x_t) \left(\nabla_x \widetilde{f}(x_t, y_t) - \nabla_x f(x_t, y_t)\right) } }_{\mlabelterm{term:lv3}} \\
\end{align*}
Now we proceed to bound each term.

\paragraph{\Refterm{term:lv1}}
\begin{align*}
  \mRefterm{term:lv1}
   &\leq \left(\rk^2 + \frac{\widehat{L}^2  G^2 \left(\eta^x\right)^2}{\mu \eta^y \left(v^y_{0}\right)^{2\alpha - \beta}}\right) 
    \Ep{\sum_{t=t_0}^{T-1} \frac{\eta_t^2}{2\rg_t}  \norm*{\nabla_x \widetilde{f}(x_t, y_t)}^2 }  \\
   &= \left(\rk^2 + \frac{\widehat{L}^2  G^2 \left(\eta^x\right)^2}{\mu \eta^y \left(v^y_{0}\right)^{2\alpha - \beta}}\right) 
    \Ep{\sum_{t=t_0}^{T-1} \frac{\left(\eta^x\right)^2 \left(v^y_{t+1}\right)^{\beta}}{2 \eta^y \max\left\{v^x_{t+1}, v^y_{t+1}\right\}^{2\alpha}}  \norm*{\nabla_x \widetilde{f}(x_t, y_t)}^2 }  \\
   &\leq \left(\rk^2 + \frac{\widehat{L}^2  G^2 \left(\eta^x\right)^2}{\mu \eta^y \left(v^y_{0}\right)^{2\alpha - \beta}}\right) 
    \Ep{\sum_{t=t_0}^{T-1} \frac{\left(\eta^x\right)^2 \left(v^y_{t+1}\right)^{\beta}}{2 \eta^y \left(v^y_{t+1}\right)^{\beta} \left(v^y_{t+1}\right)^{\alpha - \beta} \left(v^x_{t+1}\right)^{\alpha} }  \norm*{\nabla_x \widetilde{f}(x_t, y_t)}^2 }  \\
   &\leq \left(\rk^2 + \frac{\widehat{L}^2  G^2 \left(\eta^x\right)^2}{\mu \eta^y \left(v^y_{0}\right)^{2\alpha - \beta}}\right) 
    \Ep{\sum_{t=t_0}^{T-1} \frac{\left(\eta^x\right)^2 }{2 \eta^y \left(v^y_{0}\right)^{\alpha - \beta} \left(v^x_{t+1}\right)^{\alpha} }  \norm*{\nabla_x \widetilde{f}(x_t, y_t)}^2 }  \\
   &\leq \left(\rk^2 + \frac{\widehat{L}^2  G^2 \left(\eta^x\right)^2}{\mu \eta^y \left(v^y_{0}\right)^{2\alpha - \beta}}\right) 
    \Ep{\frac{\left(\eta^x\right)^2 }{2 \eta^y \left(v^y_{0}\right)^{\alpha - \beta} } 
    \left(\frac{v^x_{0}}{\left(v^x_{0}\right)^{\alpha}} + \sum_{t=0}^{T-1} \frac{1}{\left(v^x_{t+1}\right)^{\alpha} }  \norm*{\nabla_x \widetilde{f}(x_t, y_t)}^2\right) }  \\
   &\leq \left(\rk^2 + \frac{\widehat{L}^2  G^2 \left(\eta^x\right)^2}{\mu \eta^y \left(v^y_{0}\right)^{2\alpha - \beta}}\right) 
    \frac{\left(\eta^x\right)^2 }{2 (1 - \alpha) \eta^y \left(v^y_{0}\right)^{\alpha - \beta} }
    \Ep{ \left(v^x_{T}\right)^{1-\alpha}},
\end{align*}
where we used \Cref{lemma:sum} in the last step.

\paragraph{\Refterm{term:lv2}}
\begin{align*}
  \mRefterm{term:lv2}
   &\leq \Ep{\sum_{t=t_0}^{T-1} \frac{2 \rk^2 \left(v^y_{t+1}\right)^{\beta} \eta_t^2}{\rg_t \mu \eta^y} \norm*{\nabla_x f(x_t, y_t)}^2}  \\
   &= \frac{2\rk^2 \left(\eta^x\right)^2}{\mu \left(\eta^y\right)^2}
   \Ep{\sum_{t=t_0}^{T-1} \frac{\left(v^y_{t+1}\right)^{2\beta}}{\max\left\{v^x_{t+1}, v^y_{t+1}\right\}^{2\alpha}} \norm*{\nabla_x f(x_t, y_t)}^2}  \\
   &\leq \frac{2\rk^2 \left(\eta^x\right)^2}{\mu \left(\eta^y\right)^2}
   \Ep{\sum_{t=t_0}^{T-1} \frac{\left(v^y_{t+1}\right)^{2\beta}}{\left(v^y_{t+1}\right)^{2\alpha}} \norm*{\nabla_x f(x_t, y_t)}^2}  \\
   &\leq \frac{2\rk^2 \left(\eta^x\right)^2}{\mu \left(\eta^y\right)^2}
   \Ep{\frac{1}{\left(v^y_{t_0+1}\right)^{2\alpha - 2\beta}} \sum_{t=t_0}^{T-1} \norm*{\nabla_x f(x_t, y_t)}^2}  \\
   &\leq \frac{2\rk^2 \left(\eta^x\right)^2}{\mu \left(\eta^y\right)^2 C^{2\alpha - 2\beta}}
   \Ep{ \sum_{t=t_0}^{T-1} \norm*{\nabla_x f(x_t, y_t)}^2}  \\
\end{align*}

\paragraph{\Refterm{term:lv3}}
For simplicity, denote $m_t := \frac{2}{\rg_t (2 + \mu \rg_t)} \left(y_{t+1} - y^*_t\right)^\T\nabla y^*(x_t) \left(\nabla_x \widetilde{f}(x_t, y_t) - \nabla_x f(x_t, y_t)\right) $
Since $y^*(\cdot)$ is $\rk$-Lipschitz as in \Cref{lemma:lip_y_star}, $|m_t|$ can be upper bounded as
\begin{align*}
  |m_t| 
  &\leq \frac{1}{\rg_t} \norm*{y_{t+1} - y^*_t} \norm*{\nabla y^*(x_t)} 
  \left(\norm*{\nabla_x \widetilde{f}(x_t, y_t)} + \norm*{\nabla_x f(x_t, y_t)}\right) \\
  &\leq \frac{\rk}{\rg_t} \norm*{y_{t+1} - y^*_t}
  \left(\norm*{\nabla_x \widetilde{f}(x_t, y_t)} + \norm*{\nabla_x f(x_t, y_t)}\right) \\
  &\leq \frac{\rk}{\rg_t} \norm*{\cP_{\cY}\left(y_{t} + \rg_t \nabla_y \widetilde{f}(x_t, y_t)\right) - y^*_t}
  \left(\norm*{\nabla_x \widetilde{f}(x_t, y_t)} + \norm*{\nabla_x f(x_t, y_t)}\right) \\
  &\leq \frac{\rk}{\rg_t} \norm*{y_{t} + \rg_t \nabla_y \widetilde{f}(x_t, y_t) - y^*_t}
  \left(\norm*{\nabla_x \widetilde{f}(x_t, y_t)} + \norm*{\nabla_x f(x_t, y_t)}\right) \\
  &\leq \frac{\rk}{\rg_t} \left(\norm*{y_{t} - y^*_t} + \norm*{\rg_t \nabla_y \widetilde{f}(x_t, y_t)}\right)
  \left(\norm*{\nabla_x \widetilde{f}(x_t, y_t)} + \norm*{\nabla_x f(x_t, y_t)}\right) \\
  &\leq \frac{\rk}{\rg_t} \left(\frac{1}{\mu} \norm*{\nabla_y f(x_t, y_t)} + \norm*{\rg_t \nabla_y \widetilde{f}(x_t, y_t)}\right)
  \left(\norm*{\nabla_x \widetilde{f}(x_t, y_t)} + \norm*{\nabla_x f(x_t, y_t)}\right) \\
  &\leq \underbrace{ \frac{2G\rk}{\rg_{T-1}} \left(\frac{G}{\mu} + \frac{\eta^y  G}{\left(v^y_{0}\right)^{\beta}}\right)
   }_{M}.
\end{align*}
Also note that $\rg_t$ and $y_{t+1}$ does not depend on $\xi^x_t$,
so $\Ep[\xi^x_t]{m_t} = 0$. Next, we look at \Refterm{term:lv3}.
\begin{align*}
  \mRefterm{term:lv3}
  &=\Ep{\sum_{t=t_0}^{T-1} \eta_t m_t} \\
  &= \Ep{ \eta_{t_0} m_{t_0}
    + \sum_{t=t_0+1}^{T-1} \eta_{t-1} m_t
  + \sum_{t=t_0+1}^{T-1} \left(\eta_t - \eta_{t-1} \right) m_t } \\
  &\leq \Ep{ \frac{\eta^x}{\left(v^y_{0}\right)^{\alpha}} M
    + \sum_{t=t_0+1}^{T-1} \eta_{t-1} \Ep[\xi^x_{t}]{ m_t}
  + \sum_{t=t_0+1}^{T-1} \left(\eta_{t-1} - \eta_{t} \right) (-m_t) } \\
  &\leq \Ep{ \frac{\eta^x}{\left(v^y_{0}\right)^{\alpha}} M
  + \sum_{t=t_0+1}^{T-1} \left(\eta_{t-1} - \eta_{t} \right) M} \\
  &\leq \Ep{ \frac{2\eta^x}{\left(v^y_{0}\right)^{\alpha}} M} \\
  &= \left(\frac{1}{\mu} + \frac{\eta^y }{\left(v^y_{0}\right)^{\beta}}\right) 
    \frac{4\rk\eta^xG^2}{\eta^y \left(v^y_{0}\right)^{\alpha}} \Ep{ \left(v^y_{T}\right)^{\beta}}.
\end{align*}
Summarizing all the results, we finish the proof.

\end{proof}

\begin{lemma} \label{lemma:telescope_distance}
  Under the same setting as \Cref{theorem:tiada_stoc},
  we have
  \begin{align*}
    &\fakeeq \Ep{\sum_{t=0}^{T-1} \left( \frac{1 - \rg_t \mu}{2 \rg_t} \norm*{y_t - y^*_{t}}^2 - \frac{1}{\rg_t (2 + \mu \rg_t)} \norm*{y_{t+1} - y^*_{t+1}}^2 \right)} \\
    &\leq \frac{\left(v^y_{0}\right)^{\beta} G^2}{2\mu^2 \eta^y}
    + \frac{\left(2\beta G\right)^{\frac{1}{1 - \beta} + 2} G^2}{
      4 \mu^{\frac{1}{1-\beta} + 3} \left(\eta^y\right)^{\frac{1}{1-\beta} + 2}\left(v^y_{0}\right)^{2 - 2\beta}}.
  \end{align*}
\end{lemma}

\begin{proof}
  \begin{align*}
    &\fakeeq \Ep{\sum_{t=0}^{T-1} \left( \frac{1 - \rg_t \mu}{2 \rg_t} \norm*{y_t - y^*_{t}}^2 - \frac{1}{\rg_t (2 + \mu \rg_t)} \norm*{y_{t+1} - y^*_{t+1}}^2 \right)} \\
    &\leq \left(\frac{\left(v^y_{0}\right)^{\beta}}{2\eta^y} - \frac{\mu}{2} \right) \norm*{y_{0} - y_{0}^*}^2
    + \frac{1}{2\eta^y}\sum_{t=1}^{T-1} \left( \left(v^y_{t+1}\right)^{\beta} - \frac{\mu \eta^y}{2} - \left(v^y_{t}\right)^{\beta} - \frac{\mu^2 \left(\eta^y\right)^2}{4 \left(v^y_{t}\right)^{\beta} + 2 \mu \eta^y} \right) \norm*{y_t - y^*_{t}}^2 \\
    &\leq \frac{\left(v^y_{0}\right)^{\beta} G^2}{2\mu^2 \eta^y}
    + \frac{1}{2\eta^y} \underbrace{ \sum_{t=1}^{T-1} \left( \left(v^y_{t+1}\right)^{\beta} - \frac{\mu \eta^y}{2} - \left(v^y_{t}\right)^{\beta} \right) \norm*{y_t - y^*_{t}}^2 }_{\mlabelterm{term:td1}}.
  \end{align*}
For \Refterm{term:td1},we will bound it using the same strategy as in \citep{yang2022nest}.
The general idea is to show that $\left(v^y_{t+1}\right)^{\beta} - \frac{\mu \eta^y}{2} - \left(v^y_{t}\right)^{\beta}$
is positive for only a constant number of times. If the term is positive
at iteration $t$, then we have
\begin{align*}
  0 &< \left(v^y_{t+1}\right)^{\beta} - \left(v^y_{t}\right)^{\beta} - \frac{\mu \eta^y}{2} \\
    &= \left(v^y_{t} + \norm*{\nabla_y \widetilde{f}(x_{t}, y_t)}^2 \right)^{\beta} - \left(v^y_{t}\right)^{\beta} - \frac{\mu \eta^y}{2} \\
    &= \left(v^y_{t}\right)^{\beta} \left(1 + \frac{\norm*{\nabla_y \widetilde{f}(x_{t}, y_t)}^2}{v^y_{t}} \right)^{\beta} - \left(v^y_{t}\right)^{\beta} - \frac{\mu \eta^y}{2} \\
    &\leq \left(v^y_{t}\right)^{\beta} \left(1 + \frac{\beta \norm*{\nabla_y \widetilde{f}(x_{t}, y_t)}^2}{v^y_{t}} \right) - \left(v^y_{t}\right)^{\beta} - \frac{\mu \eta^y}{2} \\
    &= \frac{\beta \norm*{\nabla_y \widetilde{f}(x_{t}, y_t)}^2}{\left(v^y_{t}\right)^{1-\beta}} - \frac{\mu \eta^y}{2}, \numberthis \label{eq:pe1}
\end{align*}
where in the last inequality we used Bernoulli’s inequality.
By rearranging the terms, we have the two following conditions
\begin{align*}
  \begin{cases}
    \norm*{\nabla_y \widetilde{f}(x_{t}, y_t)}^2 > \frac{\mu \eta^y}{2 \beta} \left(v^y_{t}\right)^{1 - \beta}
      \geq \frac{\mu \eta^y}{2 \beta} \left(v^y_{0}\right)^{1 - \beta} \\
      \left(v^y_{t}\right)^{1-\beta} < \frac{2\beta}{\mu \eta^y} \norm*{\nabla_y \widetilde{f}(x_{t}, y_t)}^2 
        \leq \frac{2\beta  G}{\mu \eta^y},
  \end{cases}
\end{align*}
This indicates that at each time the term is positive, the gradient norm
must be large enough and the accumulated gradient norm, i.e., $v^y_{t+1}$,
must be small enough. Therefore, we can have at most
\[
  \frac{\left(\frac{2\beta  G}{\mu \eta^y}\right)^{\frac{1}{1-\beta}}}{\frac{\mu \eta^y}{2 \beta} \left(v^y_{0}\right)^{1 - \beta}}
\]
constant number of iterations when the term is positive. When the term is positive, it is
also upper bounded by using the result from \Cref{eq:pe1}:
\begin{align*}
  \left( \left(v^y_{t+1}\right)^{\beta} - \frac{\mu \eta^y}{2} - \left(v^y_{t}\right)^{\beta} \right) \norm*{y_t - y^*_{t}}^2
  &\leq \frac{\beta \norm*{\nabla_y \widetilde{f}(x_{t}, y_t)}^2}{\left(v^y_{t}\right)^{1-\beta}} \norm*{y_t - y^*_{t}}^2 \\
  &\leq \frac{\beta  G^2}{\left(v^y_{0}\right)^{1-\beta}} \norm*{y_t - y^*_{t}}^2 \\
  &\leq \frac{\beta  G^2}{\mu^2 \left(v^y_{0}\right)^{1-\beta}} \norm*{\nabla_y f(x_t, y_t)}^2 \\
  &\leq \frac{\beta  G^4}{\mu^2 \left(v^y_{0}\right)^{1-\beta}} \\
\end{align*}
which is a constant. In total, \Refterm{term:td1} is bounded by
\[
  \frac{\left(2\beta G\right)^{\frac{1}{1 - \beta} + 2} G^2}{2
  \mu^{\frac{1}{1-\beta} + 3} \left(\eta^y\right)^{\frac{1}{1-\beta} + 1}\left(v^y_{0}\right)^{2 - 2\beta}}.
\]
Bringing it back, we get the desired result.
\end{proof}

\begin{lemma} \label{lemma:y_part}
  Under the same setting as \Cref{theorem:tiada_stoc},
  for any constant $C$, we have 
  \begin{align*}
    &\fakeeq \Ep{\sum_{t=0}^{T-1} \left( f(x_{t}, y^*_{t}) - f(x_{t}, y_t) \right)} \\
    &\leq \frac{2\rk^2 \left(\eta^x\right)^2}{\mu \left(\eta^y\right)^2 C^{2\alpha - 2\beta}}
       \Ep{ \sum_{t=0}^{T-1} \norm*{\nabla_x f(x_t, y_t)}^2}
    + \frac{\eta^y}{2(1-\beta)} \Ep{\left(v^y_{T}\right)^{1-\beta}} \\
    &\fakeeq+ \left(\frac{1}{\mu} + \frac{\eta^y }{\left(v^y_{0}\right)^{\beta}}\right) 
        \frac{4\rk\eta^x G^2 }{\eta^y \left(v^y_{0}\right)^{\alpha}} \Ep{ \left(v^y_{T}\right)^{\beta}} \\
    &\fakeeq+ \frac{\rk^2 \left(\mu \eta^y C^{\beta} + 2 C^{2\beta}\right) \left(\eta^x\right)^2}{2 \mu \left(\eta^y\right)^2 } 
 \Ep{  \frac{1 + \log v^x_{T} - \log v^x_{0}}{\left(v^x_{0}\right)^{2\alpha-1}} \cdot \mathbf{1}_{\alpha \geq 0.5} 
  +  \frac{\left(v_{T}^x\right)^{1-2\alpha}}{1 - 2\alpha} \cdot \mathbf{1}_{\alpha < 0.5}   } \\
    &\fakeeq+ \left(\rk^2 + \frac{\widehat{L}^2  G^2 \left(\eta^x\right)^2}{\mu \eta^y \left(v^y_{0}\right)^{2\alpha - \beta}}\right) 
        \frac{\left(\eta^x\right)^2 }{2 (1 - \alpha) \eta^y \left(v^y_{0}\right)^{\alpha - \beta} }
        \Ep{ \left(v^x_{T}\right)^{1-\alpha}} \\
    &\fakeeq+ \frac{\left(v^y_{0}\right)^{\beta} G^2}{2\mu^2 \eta^y}
    + \frac{\left(2\beta G\right)^{\frac{1}{1 - \beta} + 2} G^2}{
      4 \mu^{\frac{1}{1-\beta} + 3} \left(\eta^y\right)^{\frac{1}{1-\beta} + 2}\left(v^y_{0}\right)^{2 - 2\beta}}.
  \end{align*}
\end{lemma}

\begin{proof}
  By \Cref{lemma:v_small} and \Cref{lemma:v_large}, we have for any constant $C$,
\begin{align*}
&\fakeeq \Ep{\sum_{t=0}^{T-1} \left( f(x_{t}, y^*_{t}) - f(x_{t}, y_t) \right)} \\
&\leq \Ep{\sum_{t=0}^{T-1} \left( \frac{1 - \rg_t \mu}{2 \rg_t} \norm*{y_t - y^*_{t}}^2 - \frac{1}{\rg_t (2 + \mu \rg_t)} \norm*{y_{t+1} - y^*_{t+1}}^2 \right)} \\
&\fakeeq + \Ep{\sum_{t=0}^{T-1} \frac{\rg_t}{2} \norm*{\nabla_y \widetilde{f}(x_{t}, y_t)}^2}
+ \frac{2\rk^2 \left(\eta^x\right)^2}{\mu \left(\eta^y\right)^2 C^{2\alpha - 2\beta}}
   \Ep{ \sum_{t=0}^{T-1} \norm*{\nabla_x f(x_t, y_t)}^2} \\
&\fakeeq + \frac{\rk^2 \left(\mu \eta^y C^{\beta} + 2 C^{2\beta}\right) \left(\eta^x\right)^2}{2 \mu \left(\eta^y\right)^2 } 
 \Ep{  \frac{1 + \log v^x_{T} - \log v^x_{0}}{\left(v^x_{0}\right)^{2\alpha-1}} \cdot \mathbf{1}_{\alpha \geq 0.5} 
  +  \frac{\left(v_{T}^x\right)^{1-2\alpha}}{1 - 2\alpha} \cdot \mathbf{1}_{\alpha < 0.5}   } \\
&\fakeeq+ \left(\rk^2 + \frac{\widehat{L}^2  G^2 \left(\eta^x\right)^2}{\mu \eta^y \left(v^y_{0}\right)^{2\alpha - \beta}}\right) 
    \frac{\left(\eta^x\right)^2 }{2 (1 - \alpha) \eta^y \left(v^y_{0}\right)^{\alpha - \beta} }
    \Ep{ \left(v^x_{T}\right)^{1-\alpha}} \\
&\fakeeq
+ \left(\frac{1}{\mu} + \frac{\eta^y  }{\left(v^y_{0}\right)^{\beta}}\right) 
    \frac{4\rk\eta^xG^2}{\eta^y \left(v^y_{0}\right)^{\alpha}} \Ep{ \left(v^y_{T}\right)^{\beta}}.
\end{align*}
The first term can be bounded by \Cref{lemma:telescope_distance}.
For the second term, we have
\begin{align*}
  \Ep{\sum_{t=0}^{T-1} \frac{\rg_t}{2} \norm*{\nabla_y \widetilde{f}(x_{t}, y_t)}^2}
  &= \Ep{\sum_{t=0}^{T-1} \frac{\eta^y}{2\left(v^y_{t+1}\right)^{\beta}} \norm*{\nabla_y \widetilde{f}(x_{t}, y_t)}^2} \\
  &\leq \frac{\eta^y}{2} \Ep{\frac{v^y_{0}}{\left(v^y_{0}\right)^{\beta}} + \sum_{t=0}^{T-1} \frac{1}{\left(v^y_{t+1}\right)^{\beta}} \norm*{\nabla_y \widetilde{f}(x_{t}, y_t)}^2} \\
  &\leq \frac{\eta^y}{2(1-\beta)} \Ep{\left(v^y_{T}\right)^{1-\beta}},
\end{align*}
where the last inequality follows from \Cref{lemma:sum}.
Then the proof is completed.
\end{proof}

\subsection{Proof of Theorem~\ref{theorem:tiada_stoc}}
\label{sec:proof_stoc}

We present a formal version of \Cref{theorem:tiada_stoc}.
\begin{theorem}[stochastic setting] \label{theorem:tiada_stoc_formal}
  Under \Cref{assume:strong-convex,assume:smoothness,assume:stochastic_grad,assume:lipschitz,assume:interior_optimal,assume:bound_func_val},
  \Cref{alg:tiada} with stochastic gradient
  oracles satisfies that for any $0 < \beta < \alpha < 1$, after $T$ iterations,
  \begin{align*}
  &\fakeeq \Ep{\frac{1}{T}\sum_{t=0}^{T-1}\norm*{\nabla_x f(x_t, y_t)}^2}  \\
  &\leq \frac{4 \Delta \Phi G^{2\alpha}}{\eta^x T^{1-\alpha}} 
  + \left( \frac{4 l\rk \eta^x}{1- \ra} 
  + \left(\rk^2 + \frac{\widehat{L}^2  G^2 \left(\eta^x\right)^2}{\mu \eta^y \left(v^y_{0}\right)^{2\alpha - \beta}}\right) 
  \frac{2 l \rk \left(\eta^x\right)^2 }{(1 - \alpha) \eta^y \left(v^y_{0}\right)^{\alpha - \beta} } \right)
  \frac{G^{2(1-\alpha)}}{T^{\alpha}} \\
  &\fakeeq + \frac{2 l \rk \eta^y G^{2(1-\beta)}}{(1-\beta)T^{\beta}}
  + \left(\frac{1}{\mu} + \frac{\eta^y  }{\left(v^y_{0}\right)^{\beta}}\right) 
  \frac{16 l \rk^2 \eta^xG^{2(1+\beta)}}{\eta^y \left(v^y_{0}\right)^{\alpha} T^{1-\beta}}  \\
  &\fakeeq + \frac{2 \rk^4 \left(\mu \eta^y C^{\beta} + 2 C^{2\beta}\right) \left(\eta^x\right)^2}{\left(\eta^y\right)^2 } 
  \left(  \frac{1 + \log (G^2 T) - \log v^x_{0}}{\left(v^x_{0}\right)^{2\alpha-1} T} \cdot \mathbf{1}_{\alpha \geq 0.5} 
  +  \frac{G^{2(1-2\alpha)} }{(1 - 2\alpha) T^{2\alpha}} \cdot \mathbf{1}_{\alpha < 0.5}    \right) \\
  &\fakeeq
  + \frac{2 \rk^2 \left(v^y_{0}\right)^{\beta} G^2}{\mu \eta^y T}
  + \frac{l \rk \left(2\beta G\right)^{\frac{1}{1 - \beta} + 2} G^2}{
    \mu^{\frac{1}{1-\beta} + 3} \left(\eta^y\right)^{\frac{1}{1-\beta} + 2}\left(v^y_{0}\right)^{2 - 2\beta} T},
\end{align*}
and
\begin{align*}
  &\fakeeq \Ep{\frac{1}{T}\sum_{t=0}^{T-1}\norm*{\nabla_y f(x_t, y_t)}^2} \\
    &\leq \frac{4 \rk^3 \left(\eta^x\right)^2}{\left(\eta^y\right)^2 C^{2\alpha - 2\beta}}
       \Ep{ \frac{1}{T} \sum_{t=0}^{T-1} \norm*{\nabla_x f(x_t, y_t)}^2}
       + \frac{l\eta^y G^{2-2\beta}}{(1-\beta)T^{\beta}}
    + \left(\frac{1}{\mu} + \frac{\eta^y  }{\left(v^y_{0}\right)^{\beta}}\right) 
    \frac{8l\rk\eta^x G^{2+2\beta}}{\eta^y \left(v^y_{0}\right)^{\alpha}T^{1-\beta}}  \\
    &\fakeeq+ \frac{\rk^3 \left(\mu \eta^y C^{\beta} + 2 C^{2\beta}\right) \left(\eta^x\right)^2}{ \left(\eta^y\right)^2} 
    \left(  \frac{1 + \log TG^2 - \log v^x_{0}}{\left(v^x_{0}\right)^{2\alpha-1} T} \cdot \mathbf{1}_{\alpha \geq 0.5} 
    +  \frac{G^{2-4\alpha}}{(1 - 2\alpha)T^{2\alpha}} \cdot \mathbf{1}_{\alpha < 0.5}   \right) \\
    &\fakeeq+ \left(\rk^2 + \frac{\widehat{L}^2  G^2 \left(\eta^x\right)^2}{\mu \eta^y \left(v^y_{0}\right)^{2\alpha - \beta}}\right) 
    \frac{l \left(\eta^x\right)^2 G^{2 - 2\alpha}}{(1 - \alpha) \eta^y \left(v^y_{0}\right)^{\alpha - \beta} T^{\alpha}}
    + \frac{\rk \left(v^y_{0}\right)^{\beta} G^2}{\mu \eta^y T}
    + \frac{2 l \left(2\beta G\right)^{\frac{1}{1 - \beta} + 2} G^2}{
      4  \mu^{\frac{1}{1-\beta} + 3} \left(\eta^y\right)^{\frac{1}{1-\beta} + 2}\left(v^y_{0}\right)^{2 - 2\beta} T}.
\end{align*}
\end{theorem}

\begin{proof}
By smoothness of the primal function, we have
\begin{align*}
  \Phi(x_{t+1}) - \Phi(x_{t}) \leq - \eta_t 
    \inp*{\nabla \Phi(x_{t})}{\nabla_x \widetilde{f}(x_t, y_t)}
    + l\rk \eta_t^2 \norm*{\nabla_x \widetilde{f}(x_t, y_t)}^2.
\end{align*}
By multiplying $\frac{1}{\eta_t}$ on both sides and taking the expectation
w.r.t. the noise of current iteration, we have
\begin{align*}
  &\fakeeq \Ep{\frac{\Phi(x_{t+1}) - \Phi(x_{t})}{\eta_t}}  \\
  &\leq -
    \inp*{\nabla \Phi(x_{t})}{\nabla_x f(x_t, y_t)}
    + l\rk \Ep{\eta_t \norm*{\nabla_x \widetilde{f}(x_t, y_t)}^2} \\
  &= -\norm*{\nabla_x f(x_t, y_t)}^2 
    + \inp*{\nabla_x f(x_t, y_t) - \nabla \Phi(x_{t})}{\nabla_x f(x_t, y_t)}
    + l\rk \Ep{\eta_t \norm*{\nabla_x \widetilde{f}(x_t, y_t)}^2} \\
  &\leq -\norm*{\nabla_x f(x_t, y_t)}^2 
    + \frac{1}{2}\norm*{\nabla_x f(x_t, y_t) - \nabla \Phi(x_{t})}^2 + \frac{1}{2}\norm*{\nabla_x f(x_t, y_t)}^2
    + l\rk \Ep{\eta_t \norm*{\nabla_x \widetilde{f}(x_t, y_t)}^2} \\
  &= -\frac{1}{2}\norm*{\nabla_x f(x_t, y_t)}^2 
    + \frac{1}{2}\norm*{\nabla_x f(x_t, y_t) - \nabla \Phi(x_{t})}^2
    + l\rk \Ep{\eta_t \norm*{\nabla_x \widetilde{f}(x_t, y_t)}^2} \\
\end{align*}
Summing over $t=0$ to $T-1$, rearranging and taking total expectation, we get 
\begin{align*}
  &\fakeeq \Ep{\sum_{t=0}^{T-1}\norm*{\nabla_x f(x_t, y_t)}^2}  \\
  &\leq \underbrace{2 \Ep{\sum_{t=0}^{T-1}\frac{\Phi(x_{t}) - \Phi(x_{t+1})}{\eta_t}}}_{\text{\labelterm{term:31}}}
  + \underbrace{2 l\rk \Ep{\sum_{t=0}^{T-1}\eta_t \norm*{\nabla_x \widetilde{f}(x_t, y_t)}^2}}_{\text{\labelterm{term:32}}}
  + \underbrace{\Ep{\sum_{t=0}^{T-1}\norm*{\nabla_x f(x_t, y_t) - \nabla \Phi(x_{t})}^2}}_{\text{\labelterm{term:33}}}.
  \numberthis \label{eq:stoch_smooth}
\end{align*}

\paragraph{\Refterm{term:31}}
\begin{align*}
  2 \Ep{\sum_{t=0}^{T-1}\frac{\Phi(x_{t}) - \Phi(x_{t+1})}{\eta_t}}
  &\leq 2 \Ep{\frac{\Phi(x_{0})}{\eta_0} - \frac{\Phi(x_T)}{\eta_{T-1}} + \sum_{t=1}^{T-1} \Phi(x_{t})
  \left( \frac{1}{\eta_{t}} - \frac{1}{\eta_{t-1}}\right)} \\
  &\leq 2 \Ep{\frac{\Phi_{\max}}{\eta_0} - \frac{\Phi^*}{\eta_{T-1}} + \sum_{t=1}^{T-1} \Phi_{\max}
  \left( \frac{1}{\eta_{t}} - \frac{1}{\eta_{t-1}}\right)} \\
  &= 2 \Ep{\frac{\Delta \Phi}{\eta_{T-1}}} 
  = 2 \Ep{\frac{\Delta \Phi}{\eta^x} \max\left\{ v^x_{T}, v^y_{T} \right\}^{\ra} }.
\end{align*}

\paragraph{\Refterm{term:32}}
\begin{align*}
  2 l\rk \sum_{t=0}^{T-1}\Ep{\eta_t \norm*{\nabla_x \widetilde{f}(x_t, y_t)}^2}
  &= 2 l\rk \Ep{\sum_{t=0}^{T-1} \frac{\eta^x}{\max\left\{v^x_{t+1}, v^y_{t+1}\right\}^{\alpha}} \norm*{\nabla_x \widetilde{f}(x_t, y_t)}^2} \\
  &\leq 2 l\rk \eta^x \Ep{\sum_{t=0}^{T-1} \frac{1}{\left(v^x_{t+1}\right)^{\alpha}} \norm*{\nabla_x \widetilde{f}(x_t, y_t)}^2} \\
  &\leq 2 l\rk \eta^x \Ep{ \left(\frac{v^x_{0}}{\left(v^x_{0}\right)^{\alpha}} + \sum_{t=0}^{T-1} \frac{1}{\left(v^x_{t+1}\right)^{\alpha}} \norm*{\nabla_x \widetilde{f}(x_t, y_t)}^2\right)} \\
  &\leq  \frac{2 l\rk \eta^x}{1- \ra} \Ep{  \left(v^x_{T}\right)^{1-\alpha} }.
\end{align*}

\paragraph{\Refterm{term:33}}
According to the smoothness of $f(x_t, \cdot)$, we have
\begin{align*}
  \Ep{\sum_{t=0}^{T-1}\norm*{\nabla_x f(x_t, y_t) - \nabla \Phi(x_{t})}^2}
  \leq l^2 \Ep{ \sum_{t=0}^{T-1}\norm*{y_t - y^*_t}^2}
  \leq 2 l\rk \Ep{\sum_{t=0}^{T-1} \left(f(x_{t}, y^*_{t}) - f(x_{t}, y_{t})\right)},
\end{align*}
where the last inequality follows the strong-concavity of $y$.
Now we let 
\begin{align*}
  C = \left( \frac{8 l \rk^3 \left(\eta^x\right)^2}{\mu \left(\eta^y\right)^2} \right)^{\frac{1}{2\alpha - 2\beta}},
\end{align*}
and apply \Cref{lemma:y_part}, in total, we have
\begin{align*}
  &\fakeeq \Ep{\sum_{t=0}^{T-1}\norm*{\nabla_x f(x_t, y_t)}^2}  \\
  &\leq \frac{1}{2} \Ep{ \sum_{t=0}^{T-1} \norm*{\nabla_x f(x_t, y_t)}^2}
  + 2 \Ep{\frac{\Delta \Phi}{\eta^x} \max\left\{ v^x_{T}, v^y_{T} \right\}^{\ra} }
  + \frac{2 l\rk \eta^x}{1- \ra} \Ep{  \left(v^x_{T}\right)^{1-\alpha} } \\
  &\fakeeq + \frac{l \rk \eta^y}{1-\beta} \Ep{\left(v^y_{T}\right)^{1-\beta}}
  + \left(\frac{1}{\mu} + \frac{\eta^y  }{\left(v^y_{0}\right)^{\beta}}\right) 
      \frac{8 l \rk^2 \eta^xG^2}{\eta^y \left(v^y_{0}\right)^{\alpha}} \Ep{ \left(v^y_{T}\right)^{\beta}} \\
  &\fakeeq + \frac{\rk^4 \left(\mu \eta^y C^{\beta} + 2 C^{2\beta}\right) \left(\eta^x\right)^2}{\left(\eta^y\right)^2 } 
 \Ep{  \frac{1 + \log v^x_{T} - \log v^x_{0}}{\left(v^x_{0}\right)^{2\alpha-1}} \cdot \mathbf{1}_{\alpha \geq 0.5} 
  +  \frac{\left(v_{T}^x\right)^{1-2\alpha}}{1 - 2\alpha} \cdot \mathbf{1}_{\alpha < 0.5}   } \\
  &\fakeeq+ \left(\rk^2 + \frac{\widehat{L}^2  G^2 \left(\eta^x\right)^2}{\mu \eta^y \left(v^y_{0}\right)^{2\alpha - \beta}}\right) 
      \frac{l \rk \left(\eta^x\right)^2 }{(1 - \alpha) \eta^y \left(v^y_{0}\right)^{\alpha - \beta} }
      \Ep{ \left(v^x_{T}\right)^{1-\alpha}}
  + \frac{\rk^2 \left(v^y_{0}\right)^{\beta} G^2}{\mu \eta^y} \\
  &\fakeeq + \frac{l \rk \left(2\beta G\right)^{\frac{1}{1 - \beta} + 2} G^2}{
    2 \mu^{\frac{1}{1-\beta} + 3} \left(\eta^y\right)^{\frac{1}{1-\beta} + 2}\left(v^y_{0}\right)^{2 - 2\beta}}.
\end{align*}

It remains to bound $\left(v^z_{T}\right)^{m}$ for $z \in \{x, y\}$
and $m \geq 0$:
\begin{align*}
  \left(v^z_{T}\right)^{m} \leq \left(T G^2\right)^m.
\end{align*}
Bringing it back, we conclude our proof.
The bound of $\Ep{\frac{1}{T}\sum_{t=0}^{T-1}\norm*{\nabla_y f(x_t, y_t)}^2}$
can be implied by \Cref{lemma:y_part}.
  
\end{proof}

\subsection{TiAda without Accessing Opponent's Gradients}
\label{sec:no_max}

The effective stepsize of $x$ requires the knowledge of gradients of $y$,
i.e., $v_{t+1}^y$.
At the end of \Cref{sec:theory}, we discussed the situation when such information is
not available. Now we formally introduce the algorithm and present the convergence result.

\begin{algorithm}[ht] 
    \caption{TiAda without MAX}
    \setstretch{1.25}
    \begin{algorithmic}[1]
      \STATE \textbf{Input:} $(x_0, y_0)$, $v^x_0 > 0$, $v^y_0 > 0$, $\eta^x > 0$,
          $\eta^y > 0$, $\alpha > 0$, $\beta > 0$ and $\alpha > \beta$.
        \FOR{$t = 0,1,2,...$}
            \STATE sample i.i.d. $\xi^x_t$ and $\xi^y_t$, and let $g_t^x = \nabla_x F(x_t, y_t; \xi^x_t)$ and $g_t^y = \nabla_y F(x_t, y_t; \xi^y_t)$ 
            \STATE $v_{t+1}^x = v_t^x + \norm*{g_t^x}^2$ and $v_{t+1}^y = v_t^y + \norm*{g_t^y}^2$ 
            \STATE $x_{t+1} = x_t - \frac{\eta^x}{\left(v^x_{t+1}\right)^{\alpha}} g^x_{t}$ and
            $ y_{t+1} = \cP_{\cY}\left( y_t + \frac{\eta^y}{\left(v^y_{t+1}\right)^{\beta}} g^y_{t} \right)$
        \ENDFOR
    \end{algorithmic} \label{alg:tiada_wo_max}
\end{algorithm}

\begin{theorem}[stochastic]
   Under \Cref{assume:strong-convex,assume:smoothness,assume:stochastic_grad,assume:bound_func_val},
  \Cref{alg:tiada_wo_max} with stochastic gradient
  oracles satisfies that for any $0 < \beta < \alpha < 1$, after $T$ iterations,
  \begin{align*}
  &\fakeeq \Ep{\frac{1}{T} \sum_{t=0}^{T-1}\norm*{\nabla_x f(x_t, y_t)}^2} \\
  &\leq \frac{2 \Delta \Phi G^{2\alpha} }{ \eta^x T^{1-\alpha}}
  + \frac{2 l\rk \eta^x G^{2-2\alpha}}{(1- \ra) T^{\alpha}}
  + \left( \frac{\left(v^y_{0}\right)^{\beta} G^2}{2\mu^2 \eta^y}
    + \frac{\left(2\beta G\right)^{\frac{1}{1 - \beta} + 2} G^2}{
      4 \mu^{\frac{1}{1-\beta} + 3} \left(\eta^y\right)^{\frac{1}{1-\beta} + 2}\left(v^y_{0}\right)^{2 - 2\beta}} \right) \frac{1}{T}
  + \frac{\eta^y G^{2-2\beta} }{2(1 - \beta) T^{\beta}} \\
  &\fakeeq + \left(\frac{\left(\eta^x\right)^2\rk^2}{2\left(v^y_{0}\right)^{\beta}\eta^y} + \frac{\left(\eta^x\right)^2\rk^2}{\mu(\eta^y)^2}\right)
  \left(\frac{\left(1 + \log G^2T - \log v^x_{0}\right) G^{4\beta} }{\left(v^x_{0}\right)^{2\alpha-1}T^{1-2\beta}} \cdot \mathbf{1}_{\alpha \geq 0.5}
  +  \frac{G^{2-4\alpha + 4\beta}}{(1 - 2\alpha)T^{2\alpha-2\beta}} \cdot \mathbf{1}_{\alpha < 0.5}\right),
  \end{align*}
  and
  \begin{align*}
  &\fakeeq \Ep{\frac{1}{T}\sum_{t=0}^{T-1}\norm*{\nabla_y f(x_t, y_t)}^2}  \\
  &\leq \left( \frac{\rk \left(v^y_{0}\right)^{\beta} G^2}{\mu \eta^y}
    + \frac{2l\left(2\beta G\right)^{\frac{1}{1 - \beta} + 2} G^2}{
      4 \mu^{\frac{1}{1-\beta} + 3} \left(\eta^y\right)^{\frac{1}{1-\beta} + 2}\left(v^y_{0}\right)^{2 - 2\beta}} \right) \frac{1}{T}
  + \frac{l \eta^y G^{2-2\beta} }{(1 - \beta) T^{\beta}} \\
  &\fakeeq + \left(\frac{l\left(\eta^x\right)^2\rk^2}{\left(v^y_{0}\right)^{\beta}\eta^y} + \frac{2\left(\eta^x\right)^2\rk^3}{(\eta^y)^2}\right)
  \left(\frac{\left(1 + \log G^2T - \log v^x_{0}\right) G^{4\beta} }{\left(v^x_{0}\right)^{2\alpha-1}T^{1-2\beta}} \cdot \mathbf{1}_{\alpha \geq 0.5}
  +  \frac{G^{2-4\alpha + 4\beta}}{(1 - 2\alpha)T^{2\alpha-2\beta}} \cdot \mathbf{1}_{\alpha < 0.5}\right).
  \end{align*}
\end{theorem}
\begin{remark}
  The best rate achievable is $\cOt\left(\epsilon^{-6}\right)$ by choosing
  $\alpha=1/2$ and $\beta=1/3$.
\end{remark}
\begin{proof}
\Cref{lemma:basic_strongly_convex_expand,lemma:telescope_distance}
can be directly used here because they do not have or expand the effective
stepsize of $x$, i.e., $\eta_t$. This is also the case for the beginning part
of \Cref{sec:proof_stoc}, the proof of \Cref{theorem:tiada_stoc}, up to
\Cref{eq:stoch_smooth}. However, we need to bound Terms~\eqref{term:31},
\eqref{term:32} and \eqref{term:33} in \Cref{eq:stoch_smooth} differently.
According to our assumption on bounded stochastic gradients,
we know that $v^x_{T}$ and $v^y_{T}$
are both upper bounded by $TG^2$, which we will use throughout the proof.

\paragraph{\Refterm{term:31}}
\begin{align*}
  2 \Ep{\sum_{t=0}^{T-1}\frac{\Phi(x_{t}) - \Phi(x_{t+1})}{\eta_t}}
  &\leq 2 \Ep{\frac{\Phi(x_{0})}{\eta_0} - \frac{\Phi(x_T)}{\eta_{T-1}} + \sum_{t=1}^{T-1} \Phi(x_{t})
  \left( \frac{1}{\eta_{t}} - \frac{1}{\eta_{t-1}}\right)} \\
  &\leq 2 \Ep{\frac{\Phi_{\max}}{\eta_0}  - \frac{\Phi^*}{\eta_{T-1}} + \sum_{t=1}^{T-1} \Phi_{\max}
  \left( \frac{1}{\eta_{t}} - \frac{1}{\eta_{t-1}}\right)} \\
  &= 2 \Ep{\frac{\Delta \Phi}{\eta_{T-1}}} 
  = 2 \Ep{\frac{\Delta \Phi}{\eta^x} \left(v^x_{T}\right)^{\ra} }
  \leq \frac{2 \Delta \Phi G^{2\alpha} T^{\alpha}}{\eta^x} .
\end{align*}

\paragraph{\Refterm{term:32}}
\begin{align*}
  2 l\rk \sum_{t=0}^{T-1}\Ep{\eta_t \norm*{\nabla_x \widetilde{f}(x_t, y_t)}^2}
  &= 2 l\rk \eta^x \Ep{\sum_{t=0}^{T-1} \frac{1}{\left(v^x_{t+1}\right)^{\alpha}} \norm*{\nabla_x \widetilde{f}(x_t, y_t)}^2} \\
  &\leq 2 l\rk \eta^x \Ep{ \left(\frac{v^x_{0}}{\left(v^x_{0}\right)^{\alpha}} + \sum_{t=0}^{T-1} \frac{1}{\left(v^x_{t+1}\right)^{\alpha}} \norm*{\nabla_x \widetilde{f}(x_t, y_t)}^2\right)} \\
  &\leq  \frac{2 l\rk \eta^x}{1- \ra} \Ep{  \left(v^x_{T}\right)^{1-\alpha} }
  \leq \frac{2 l\rk \eta^x G^{2-2\alpha}T^{1-\alpha}}{1- \ra}.
\end{align*}

\paragraph{\Refterm{term:33}}
According to the smoothness and strong-concavity of $f(x_t, \cdot)$, we have
\begin{align*}
  \Ep{\sum_{t=0}^{T-1}\norm*{\nabla_x f(x_t, y_t) - \nabla \Phi(x_{t})}^2}
  \leq l^2 \Ep{ \sum_{t=0}^{T-1}\norm*{y_t - y^*_t}^2}
  \leq 2 l\rk \Ep{\sum_{t=0}^{T-1} \left(f(x_{t}, y^*_{t}) - f(x_{t}, y_{t})\right)}.
\end{align*}
To bound the RHS, we use Young's inequality and have
\begin{align*}
  \norm*{y_{t+1} - y^*_{t+1}}^2
  \leq (1 + \lambda_t) \norm*{y_{t+1} - y^*_{t}}^2 + \left(1 + \frac{1}{\lambda_t}\right) \norm*{y^*_{t+1} - y^*_{t}}^2.
\end{align*}
Then applying \Cref{lemma:basic_strongly_convex_expand} with $\lambda_t = \frac{\mu \rg_t}{2}$ gives us
\begin{align*}
&\fakeeq \Ep{\sum_{t=0}^{T-1} \left( f(x_{t}, y^*_{t}) - f(x_{t}, y_t) \right)} \\
&\leq \Ep{\sum_{t=0}^{T-1} \left( \frac{1 - \rg_t \mu}{2 \rg_t} \norm*{y_t - y^*_{t}}^2 - \frac{1}{\rg_t (2 + \mu \rg_t)} \norm*{y_{t+1} - y^*_{t+1}}^2 \right)} \\
&\fakeeq + \underbrace{\Ep{\sum_{t=0}^{T-1} \frac{\rg_t}{2} \norm*{\nabla_y \widetilde{f}(x_{t}, y_t)}^2}}_{\mlabelterm{term:cd1}}
+ \underbrace{\Ep{\sum_{t=0}^{T-1} \frac{\left(1 + \frac{2}{\mu \rg_t}\right)}{\rg_t (2 + \mu \rg_t)} \norm*{y^*_{t+1} - y^*_{t}}^2 }}_{\mlabelterm{term:cd2}},
\end{align*}
where the first term is $\cO\left(1\right)$ according to \Cref{lemma:telescope_distance}.
The other two terms can be bounded as follow.
\begin{align*}
  &\fakeeq \mRefterm{term:cd1} \\
  &\leq \Ep{\frac{\eta^y}{2} \left( \frac{v^y_{0}}{\left(v^y_{0}\right)^{\beta}} + \sum_{t=0}^{T-1} \frac{1}{\left(v^y_{t+1}\right)^{\beta}} \norm*{\nabla_y \widetilde{f}(x_t, y_t)}^2 \right)}
  \leq \Ep{\frac{\eta^y}{2(1 - \beta)} \left(v^y_{T}\right)^{1-\beta}}
  \leq \frac{\eta^y G^{2-2\beta} T^{1-\beta}}{2(1 - \beta)}.
\end{align*}

\begin{align*}
  &\fakeeq \mRefterm{term:cd2} \\
  &= \Ep{\sum_{t=0}^{T-1} \left(\frac{1}{\left(v^y_{t+1}\right)^{\beta}} + \frac{2}{\mu \eta^y}\right) 
  \frac{\left(v^y_{t+1}\right)^{2\beta}}{2\eta^y (1 + \lambda_t)} \norm*{y^*_{t+1} - y^*_{t}}^2} \\
  &\leq \left(\frac{1}{2\left(v^y_{0}\right)^{\beta}\eta^y} + \frac{1}{\mu(\eta^y)^2}\right)
  \Ep{ \sum_{t=0}^{T-1} \left(v^y_{t+1}\right)^{2 \beta} \norm*{y^*_{t+1} - y^*_{t}}^2} \\
  &\leq \left(\frac{1}{2\left(v^y_{0}\right)^{\beta}\eta^y} + \frac{1}{\mu(\eta^y)^2}\right)
  \Ep{ \left(v^y_{T}\right)^{2 \beta} \sum_{t=0}^{T-1}  \norm*{y^*_{t+1} - y^*_{t}}^2} \\
  &\leq \left(\frac{\rk^2}{2\left(v^y_{0}\right)^{\beta}\eta^y} + \frac{\rk^2}{\mu(\eta^y)^2}\right)
  \Ep{ \left(v^y_{T}\right)^{2 \beta} \sum_{t=0}^{T-1} \norm*{x_{t+1} - x_{t}}^2} \\
  &= \left(\frac{\left(\eta^x\right)^2\rk^2}{2\left(v^y_{0}\right)^{\beta}\eta^y} + \frac{\left(\eta^x\right)^2\rk^2}{\mu(\eta^y)^2}\right)
  \Ep{ \left(v^y_{T}\right)^{2 \beta} \sum_{t=0}^{T-1} \frac{1}{\left(v^x_{t+1}\right)^{2\alpha}} \norm*{\nabla_x \widetilde{f}(x_t, y_t)}^2} \\
  &\leq \left(\frac{\left(\eta^x\right)^2\rk^2}{2\left(v^y_{0}\right)^{\beta}\eta^y} + \frac{\left(\eta^x\right)^2\rk^2}{\mu(\eta^y)^2}\right)
  \Ep{ \left(v^y_{T}\right)^{2 \beta}
    \left(\frac{1 + \log v^x_{T} - \log v^x_{0}}{\left(v^x_{0}\right)^{2\alpha-1}} \cdot \mathbf{1}_{\alpha \geq 0.5} 
  +  \frac{\left(v_{T}^x\right)^{1-2\alpha}}{1 - 2\alpha} \cdot \mathbf{1}_{\alpha < 0.5}\right)  } \\
  &\leq \left(\frac{\left(\eta^x\right)^2\rk^2}{2\left(v^y_{0}\right)^{\beta}\eta^y} + \frac{\left(\eta^x\right)^2\rk^2}{\mu(\eta^y)^2}\right)
  \left(\frac{\left(1 + \log G^2T - \log v^x_{0}\right) G^{4\beta} T^{2\beta}}{\left(v^x_{0}\right)^{2\alpha-1}} \cdot \mathbf{1}_{\alpha \geq 0.5} 
  +  \frac{G^{2-4\alpha + 4\beta}T^{1-2\alpha+2\beta}}{1 - 2\alpha} \cdot \mathbf{1}_{\alpha < 0.5}\right),
\end{align*}
where we used the the Lipschitzness of $y^*(\cdot)$ in the third inequality.

Summarizing all the terms, we finish the proof.

\end{proof}

\end{document}